\documentclass[12pt,reqno]{amsart}

\usepackage[lite]{amsrefs}

\usepackage{amssymb}
\usepackage[hidelinks]{hyperref}

\usepackage[all,cmtip]{xy}

\usepackage[shortlabels]{enumitem}

\usepackage{geometry}
\usepackage{cases}

\newcommand{\N}{\mathbb{N}}
\newcommand{\Z}{\mathbb{Z}}
\newcommand{\R}{\mathbb{R}}

\newcommand{\pa}[1]{\left(#1\right)}
\newcommand{\pb}[1]{\left[#1\right]}
\newcommand{\pc}[1]{\left\{#1\right\}}
\newcommand{\pd}[1]{\left|#1\right|}
\newcommand{\pe}[1]{\left\lfloor#1\right\rfloor}

\newcommand{\mP}{\mathcal{P}}
\newcommand{\mF}{\mathcal{F}}
\newcommand{\mN}{\mathcal{N}}
\newcommand{\mL}{\mathcal{L}}
\newcommand{\mS}{\mathcal{S}}
\newcommand{\mD}{\mathcal{D}}
\newcommand{\mA}{\mathcal{A}}
\newcommand{\mZ}{\mathcal{Z}}

\numberwithin{equation}{section}

\theoremstyle{plain} 
\newtheorem{theorem}{Theorem}[section]

\newtheorem{lemma}[theorem]{Lemma}

\theoremstyle{definition}
\newtheorem{definition}[theorem]{Definition}

\theoremstyle{remark}
\newtheorem{remark}[theorem]{Remark}

\begin{document}

\title[Integer Roots of \(LA2\)-type function in the closed rotated square region]{Integer Roots of \(LA2\)-type function in the closed rotated square region}

\author{ONG KUN YI}
\author{Eddie Shahril Bin Ismail}
\address{Department of Mathematical Sciences, Faculty of Science and Technology, Universiti Kebangsaan Malaysia, Bangi, Selangor, Malaysia}


\email{kunyi070101@gmail.com, esbi@ukm.edu.my}



\date{}

\begin{abstract}
    Let \(\mA\) be the set of all Diophantine equations of the form \(au^2 + buv + cv^2 + du + ev + f = 0\), where \(a,b,c,d,e,f \in \Z\) and \(a > 0\). One way to solve the equation \(A \in \mA\) is by applying Lagrange's method which was introduced over 200 years ago. In this paper, we consider a self-defined Diophantine equation \(A \in \mA\), which we called the \textit{\(LA2\)-type equation}, motivated by results of Teckan, \"Ozko{\c c}, Fenolahy, Ramanantsoa and Totohasina. We provide some properties of \(LA2\)-type equations, and determine the set of integer solutions of equation \(A \in \mZ(1)\), where \(\mZ(1)\) is the set of all \(LA2\)-type equations such that \(A\) can be rewrite as Pell's equation \(\tilde{u} - \tau\tilde{v}^2 = 1\). In addition, we show that there exist positive integers \(M'_l\), \(l = 1,2,3,4\) such that for any \(x \in \R, x \geq \mL := \max\{M'_l: l \in \{1,2,3,4\}\}\), the formula of the number of pairwise integer solutions to the equation \(A \in \mZ(1)\) in the region enclosed by the equation \(|u| + |v| \leq x\) in the \(uv\) plane, can be determined and proved. As a consequence, we characterize the set of integer solutions satisfying \(A \in \mZ(1)\) in the region enclosed by the equation \(|u| + |v| \leq x\) in the \(uv\) plane.
\end{abstract}

\maketitle

\section{Introduction}
Denote the set \(\N\), \(\Z\) and \(\R\) as the set of positive integers, set of integers and set of real numbers respectively. Consider the general quadratic Diophantine equation
\begin{equation}
    \label{eq:1.1}
    au^2 + buv + cv^2 + du + ev + f = 0,
\end{equation}
where one is required to find integer solutions of \(\pa{u,v} \in \Z^2\), given \(a,b,c,d,e,f \in \Z\) and \(a > 0\). If \(\gcd\pa{a,b,c,d,e,f} = K\), then we can just divide both sides of Equation \eqref{eq:1.1} by \(K\), so we may assume \(K=1\). The general method of solving Equation \eqref{eq:1.1} has a long history (see \cites{Dickson19, Legendre97, Niven1942} for instance). Several special forms of Equation \eqref{eq:1.1} also being widely explored by researchers, such as \(d = e = 0\), \(d = f = 0\) and \(e = f = 0\) (see \cites{Andrica2017, Keskin2010, Matthews2002, Matthews2021}). In this paper, we discuss a method to solve Equation \eqref{eq:1.1}, which was given by Lagrange \cite{Lagrange68} over 200 years ago. Let \(D = b^2 - 4ac\), \(E = bd - 2ae\) and \(F = d^2 - 4af\), then Lagrange's method was to rewrite Equation \eqref{eq:1.1} as
\begin{equation}
    \label{eq:1.2}
    DV^2 = \pa{Dv + E}^2 + DF - E^2,
\end{equation}
where \(V = 2au + bv + d\). Clearly if \(N = E^2 - DF\), then Equation \eqref{eq:1.2} can be written as
\begin{equation}
    \label{eq:1.3}
    U^2 - DV^2 = N,
\end{equation}
where \(U = Dv + E\). Any solution \(\pa{U,V} \in \Z^2\) of Equation \eqref{eq:1.3} such that there are integers \(u,v\) for which \(U = Dv + E\) and \(V = 2au + bv + d\) will provide a solution \(\pa{u,v} \in \Z^2\) to Equation \eqref{eq:1.1}. If \(D\) is a non-square positive integer, then Equation \eqref{eq:1.1} represents a hyperbola in the Cartesian plane (see \cite{Owings1970}), and Equation \eqref{eq:1.3} with \(N \neq 0\) becomes a general Pell-type equation, where in 2024 Tamang and Singh \cite{Tamang2024} studied its solvability by using elementary and quadratic ring methods. \\

Diophantine Equation \eqref{eq:1.1} with the application of Lagrange's method has been applied extensively by number theory researchers, and it has been a subject of great interest. Here we provide some references. In 2010, Tekcan and \"Ozko{\c c} \cite{Tekcan2010} showed that the Diophantine equation
\begin{align}
    \label{eq:1.4}
    x^2 - (t^2 + t)y^2 - (4t-2)x + (4t^2 + 4t)y = 0,
\end{align}
where \(t \in \N\), consists of infinitely many integer solutions \((x,y) \in \Z^2\). After applying Lagrange's method, Equation \eqref{eq:1.4} can be simplified to \(X^2 - (t^2 + t)Y^2 = 1\) with \(X = x - 2t - 1\) and \(Y = y - 2\). In 2024, Fenolahy, Ramanantsoa and Totohasina \cite{Fenolahy2024} showed that the Diophantine equation
\begin{align}
    \label{eq:1.5}
    x^2 - ay^2 - bx - cy - d = 0,
\end{align}
where \(a,b,c,d \in \N\) and \(a\) is a non-square, exists integer solution \((x,y) \in \Z^2\) by the T-transformation method if and only if there exist positive integers \(\alpha\) and \(\beta\) such that \(b = 2\alpha\) and \(c = 2a \beta\). If \(d = a\beta^2\), then after applying Lagrange's method, Equation \eqref{eq:1.5} can be simplified to \(X^2 - aY^2 = \alpha^2\) with \(X = x - \alpha\) and \(Y = y + \beta\). \\

Notice that Equation \eqref{eq:1.4} and Equation \eqref{eq:1.5} have some similarities that are worth highlighting, in terms of their coefficients \(a,b,c,d,e,f \in \Z\), as compared to Equation \eqref{eq:1.1}. For example, it is not hard to see that \(a = 1\), \(b = 0\) and \(d\) is even. In general, we characterize them as a special type of Diophantine Equation \eqref{eq:1.1}, so-called \emph{\(LA2\)-type equation}, which we will define formally in a moment (Definition \ref{def:1.2}). \\

Let \(\tau\) be a non-square positive integer. An equation of the form
\begin{equation}
    \label{eq:1.6}
    \tilde{u}^2 - \tau \tilde{v}^2 = 1,
\end{equation}
is usually called \textit{Pell's equation}, where the unknowns are the integers \(\tilde{u}\) and \(\tilde{v}\), has a long history (see \cites{Titu2010, Burton2010, Rosen2005, Weil1984} for instance). A classical theorem asserts that the set of integer solutions of Equation \eqref{eq:1.6}, we denote as \(\mP_{\tau}\), is non-trivial, has infinitely many elements and is the union of five non-empty \textit{disjoint} sets \(\mP^{(k)}_{\tau}\) for \(k= 0,1,2,3,4\), i.e. \(\mP_{\tau} = \displaystyle\bigcup_{k=0}^4 \mP^{(k)}_{\tau}\) with
\begin{align}
    \label{eq:1.7}
    \mP^{(0)}_{\tau} &= \pc{(1,0), (-1,0)}, \\
    \label{eq:1.8}
    \mP^{(l)}_{\tau} &= \bigcup_{m=1}^{\infty}\pc{\pa{(-1)^{\pe{\frac{l}{2}}}u_m, (-1)^{\pe{\frac{l-1}{2}}}v_m}} \ \text{if} \ l \in \{1,2,3,4\},
\end{align}
where for any \(m \in \N\), \(u_m\) and \(v_m\) are positive integers defined as
\begin{align}
    \label{eq:1.9}
    u_m &= \dfrac{1}{2}\Bigl\{\pa{\alpha + \beta\sqrt{\tau}}^m + \pa{\alpha - \beta\sqrt{\tau}}^m\Bigr\} = \pe{\dfrac{1}{2}\pa{\alpha + \beta\sqrt{\tau}}^m} + 1, \\
    \label{eq:1.10}
    v_m &= \dfrac{1}{2\sqrt{\tau}}\Bigl\{\pa{\alpha + \beta\sqrt{\tau}}^m - \pa{\alpha - \beta\sqrt{\tau}}^m\Bigr\} = \pe{\dfrac{1}{2\sqrt{\tau}}\pa{\alpha + \beta\sqrt{\tau}}^m},
\end{align}
and \((\alpha, \beta) \in \N^2\) is the so-called \textit{fundamental solution} of Equation \eqref{eq:1.6}, which can be determined by
\[\alpha + \beta \sqrt{\tau} := \inf \Bigl\{\tilde{u}+ \tilde{v}\sqrt{\tau} : (\tilde{u}, \tilde{v}) \in \Z^2, \tilde{u}^2 - \tau \tilde{v}^2 = 1, \tilde{u}+ \tilde{v}\sqrt{\tau} > 1\Bigr\}.\]

\subsection{\(LA2\)-type Equation} Let \(\mA\) be the set of all Diophantine equations of the form Equation \eqref{eq:1.1}:
\begin{align*}
    &\mA = \Bigl\{au^2 + buv + cv^2 + du + ev + f = 0 : \\
    &\hspace{50pt} a,b,c,d,e,f \in \Z, a > 0, \gcd(a,b,c,d,e,f) = 1 \Bigr\}.
\end{align*}

\begin{remark}
    We say \(A_{a,b,c,d,e,f}(u,v) \in \mA\) is that \(A_{a,b,c,d,e,f}(u,v)\) is a Diophantine equation of the form \(au^2 + buv + cv^2 + du + ev + f = 0\) that satisfies the conditions in \(\mA\). For simplicity, we will write \(A := A_{a,b,c,d,e,f}(u,v)\) for the rest of the paper, and we understand that \(A\), with variables \(u\) and \(v\), depends on integers (coefficients) \(a,b,c,d,e\) and \(f\). \\
\end{remark}

For any Diophantine equation \(A \in \mA\), consider the values \(D, E, F\) and \(N\) (all depending on \(A\)) as
\begin{align}
    \label{eq:1.11}
    D = b^2 - 4ac, \hspace{10pt} E = bd - 2ae, \hspace{10pt} F = d^2 - 4af, \hspace{10pt} N = E^2 - DF.
\end{align}
In the rest of the paper, unless otherwise specified, we let the values of \(D\), \(E\), \(F\) and \(N\) be integers (all depending on \(A\)) as defined in Equation \eqref{eq:1.11}. \\

We are now ready to define \textit{\(LA2\)-type equation}: \\

\begin{definition}
    \label{def:1.2}
    We say a Diophantine equation \(A \in \mA\) is a \textit{\(LA2\)-type equation} if \(A\) satisfies the following four conditions:
    \begin{enumerate}[(i)]
        \item \(D\) is a non-square positive integer.
        \item \(D \mid E\).
        \item \(2a \mid b\), \(a \mid c\) and \(2a \mid d\).
        \item \(4a^2D \mid N\).
    \end{enumerate}
    We denote the set \(\mZ\) as the set of all \(LA2\)-type equations. Also for any \(j \in \Z\), we define the set \(\mZ(j) \subseteq \mZ\) as
    \[\mZ(j) := \pc{A \in \mZ : j = \dfrac{N}{-4a^2D}}.\]
\end{definition}

\vspace{20pt}
Based on Definition \ref{def:1.2}, one can verify that Equation \eqref{eq:1.4} and Equation \eqref{eq:1.5} are both \(LA2\)-type equations. Specifically, Equation \eqref{eq:1.4} belongs to \(\mZ(1)\) and Equation \eqref{eq:1.5} belongs to \(\mZ(\alpha^2)\). Note that from Definition \ref{def:1.2} condition (iii), \(2a \mid b\) implies \(2 \mid b\), which means that for any \(LA2\)-type equation, we allow \(b\) to be any even integer, extending the form \(b = 0\) in Equation \eqref{eq:1.4} and Equation \eqref{eq:1.5}. \\

\subsection{Main Results} In this paper, we are interested in studying \(LA2\)-type equation where it can be rewrite as Equation \eqref{eq:1.6} after applying Lagrange's method. From now onwards, we let \(\lambda = \dfrac{b}{2}\), \(\tau = \dfrac{D}{4}\), and \((\alpha,\beta)\) be the fundamental solution of Equation \eqref{eq:1.6} with \(\tau = \dfrac{D}{4}\). Note that \(\lambda \in \Z\), and \(\tau\) is a non-square positive integer. The first main result describes some properties of \(LA2\)-type equations, and determines the set of integer solutions of equation \(A \in \mZ(1)\). \\

\begin{theorem}
    \label{thm:1.3}
    Let \(A \in \mZ(j)\) for some \(j \in \Z\). The following are all true:
    \begin{enumerate}[label=\normalfont(\alph*)]
        \item \(a = 1\).
        \item Equation \(A\) can be rewrite as equation \(\tilde{u}^2 - \tau\tilde{v}^2 = j\), where
        \[\pa{\tilde{u}, \tilde{v}} = \pa{u + \lambda v + \dfrac{d}{2}, v + \dfrac{E}{D}}.\]
        \item If \(j = 1\), then the set of integer solutions of equation \(A \in \mZ(1)\), we denote as \(\mD_{A}\), is non-empty, has infinitely many elements, and is of the form
        \begin{align}
            \label{eq:1.12}
            \mD_{A} = \pc{(u,v) \in \Z^2: u = \tilde{u} - \lambda \tilde{v} + \dfrac{E}{D}\lambda- \dfrac{d}{2}, v = \tilde{v} - \dfrac{E}{D}, (\tilde{u}, \tilde{v}) \in \mP_{\tau}},
        \end{align}
        where \(\mP_{\tau}\) is the set of integer solutions of Equation \eqref{eq:1.6}.
    \end{enumerate}
\end{theorem}

\vspace{20pt}
Let \(A \in \mZ(1)\). Define a function \(\mF_{A}: \Z^2 \to \Z\) with
\begin{equation}    
    \label{eq:1.13}
    \mF_{A}(u,v) = u^2 + buv + cv^2 + du + ev + f,
\end{equation}
where the equation \(\mF_{A}(u,v) = 0\) is equivalent to \(A\) (from Theorem \ref{thm:1.3}(a), we let \(a = 1\)). Let \(x \in \R\), \(x > 0\), and consider the set
\begin{align}
    \label{eq:1.14}
    \mD_{A}(x) = \Bigl\{(u,v) \in \Z^2: \mF_{A}(u,v) = 0, |u| + |v| \leq x\Bigr\},
\end{align}
which is the set of lattice points \((u,v)\) in the region enclosed by \(|u| + |v| \leq x\) in the \(uv\) plane (the region is a square with side lengths \(\sqrt{2}x\) rotated by 45 degrees from the origin), satisfying equation \(A\). In the rest of paper, whenever we mention the function \(\mF_{A}(u,v)\) as defined in Equation \eqref{eq:1.13}, or the set \(\mD_{A}(x)\) as defined in Equation \eqref{eq:1.14}, we always assume that \(A \in \mZ(1)\). \\

Let \(\pd{\mD_{A}(x)}\) to be the number of elements in the set \(\mD_{A}(x)\). Note that for any finite number \(x > 0\), \(\pd{\mD_{A}(x)}\) is bounded above by the total number of lattice points \((u,v)\) in the region enclosed by \(|u| + |v| \leq x\) in the \(uv\) plane. In mathematical notation, \(\pd{\mD_{A}(x)} \leq \pe{x}^2 + \pa{\pe{x}+1}^2 < +\infty\) for any finite \(x > 0\). Also from Theorem \ref{thm:1.3}(c), note that \(\displaystyle\lim_{x \to +\infty}\mD_{A}(x) = \mD_{A}\), which implies \(\displaystyle\lim_{x \to +\infty} \pd{\mD_{A}(x)} = +\infty\). \\

Two interesting questions would be: Given a particular \(x > 0\), can we determine the number of elements in the set \(\mD_{A}(x)\), i.e. \(\pd{\mD_{A}(x)}\)? Given a particular \(x > 0\), can we determine what exactly is the set of solutions \(\mD_{A}(x)\)? Both questions will be our major focus in this paper. In this paper, we provide the formula of \(\pd{\mD_{A}(x)}\) (Theorem \ref{thm:1.4}), and we show that it is \textit{eventually} true for any \(x \geq \mL := \max\Bigl\{M'_l: l \in \{1,2,3,4\}\Bigr\}\), where the values of \(M'_l\) are \textit{finite} positive integers as defined in Equation \eqref{eq:3.1}. As a corollary, we provide all elements in the set \(\mD_{A}(x)\) for any \(x \geq \mL\) (Theorem \ref{thm:1.5}). \\

Let \(l \in \{1,2,3,4\}\) and \(m \in \N\). Define integers \(s^{(l)}_m\) and \(t^{(l)}_m\) as
\begin{align}
    \label{eq:1.18}
    s_m^{(l)} &= (-1)^{\pe{\frac{l}{2}}}u_m - (-1)^{\pe{\frac{l-1}{2}}}\lambda v_m + \pa{\dfrac{E}{D}\lambda- \dfrac{d}{2}}, \\
    \label{eq:1.19}
    t_m^{(l)} &= (-1)^{\pe{\frac{l-1}{2}}}v_m - \dfrac{E}{D},
\end{align}
where \(u_m\) and \(v_m\) are integers as defined in Equation \eqref{eq:1.9} and Equation \eqref{eq:1.10} respectively. We now propose the second main result, the formula of \(\pd{\mD_{A}(x)}\): \\

\begin{theorem}
    \label{thm:1.4}
    Let \(A \in \mZ(1)\). There exists positive integer \(\mL\) such that for any real number \(x \geq \mL\),
    \[\Bigl|\mD_{A}(x)\Bigr| = 2 + \displaystyle\sum_{l=1}^4\pe{\dfrac{\log\pa{\pe{x} - R_l + 1- Q_l} - \log P_l + \log \pa{2\sqrt{\tau}}}{\log \pa{\alpha + \beta\sqrt{\tau}}}},\]
    where for \(l \in \{1,2,3,4\}\), \(P_l\) is a positive real number as defined in Equation \eqref{eq:1.15}, and \(Q_l\) and \(R_l\) are integers as defined in Equation \eqref{eq:1.16} and Equation \eqref{eq:1.17} respectively, as follows:
    \begin{align}
        \label{eq:1.15}
        P_l &= \left\{
        \begin{aligned}
            &1-\lambda-(-1)^{l}\sqrt{\tau} \ & \ &\text{if} \ \lambda < (-1)^{l-1}\sqrt{\tau}, \\
            &1+\lambda+(-1)^{l}\sqrt{\tau} \ & \ &\text{if} \ \lambda > (-1)^{l-1}\sqrt{\tau}.
        \end{aligned}
        \right. \\
        \label{eq:1.16}
        Q_l &= \left\{
        \begin{aligned}
            &(-1)^{\pe{\frac{l-1}{2}}}\pa{\dfrac{E}{D}\lambda - \dfrac{E}{D} - \dfrac{d}{2}} \ & \ &\text{if} \ \lambda < (-1)^{l-1}\sqrt{\tau}, \\
            &(-1)^{\pe{\frac{l+1}{2}}}\pa{\dfrac{E}{D}\lambda + \dfrac{E}{D} - \dfrac{d}{2}} \ & \ &\text{if} \ \lambda > (-1)^{l-1}\sqrt{\tau}.
        \end{aligned}
        \right. \\
        \label{eq:1.17}
        R_l &= \left\{
        \begin{aligned}
            &1 \ & \ &\text{if} \ - \sqrt{\tau} + \dfrac{1-(-1)^l}{2} < \lambda < \sqrt{\tau} - \dfrac{1+(-1)^l}{2}, \\
            &0 \ & \ &\text{if} \ \lambda < - \sqrt{\tau} + \dfrac{1-(-1)^l}{2} \ \text{or} \ \lambda > \sqrt{\tau} - \dfrac{1+(-1)^l}{2}.
        \end{aligned}
        \right.
    \end{align}
    In particular, we have
    \[\mL := \max\Bigl\{M'_l : l \in \{1,2,3,4\}\Bigr\},\]
    where for \(l \in \{1,2,3,4\}\), \(M'_l\) is a positive integer defined as
    \begin{align}
        \label{eq:3.1}
        M'_l := \max\pc{\pd{s_{1}^{(l)}}, \pd{s_{N_0}^{(l)}}, \pd{s_{N_l}^{(l)}}} + \max\pc{\pd{t_{1}^{(l)}}, \pd{t_{N_0}^{(l)}}, \pd{t_{N_l}^{(l)}}},
    \end{align}
    where \(s_{m}^{(l)}\) and \(t_{m}^{(l)}\) are integers as defined in Equation \eqref{eq:1.18} and Equation \eqref{eq:1.19} respectively, and values \(N_0\) and \(N_l\) are positive integers as defined in Equation \eqref{eq:3.2} and Equation \eqref{eq:3.3} respectively, as follows:
    \begin{align}
        \label{eq:3.2}
        N_0 &:= \left\{
        \begin{aligned}
            &\min\pc{m \in \N : u_m > |\lambda|v_m + \pd{\dfrac{E}{D}\lambda- \dfrac{d}{2}}, v_m > \dfrac{\pd{E}}{D}} \ & \ &\text{if} \ |\lambda| < \sqrt{\tau}, \\
            &\min\pc{m \in \N : u_m < |\lambda|v_m - \pd{\dfrac{E}{D}\lambda- \dfrac{d}{2}}, v_m > \dfrac{\pd{E}}{D}} \ & \ &\text{if} \ |\lambda| > \sqrt{\tau}.
        \end{aligned}
        \right. \\
        \label{eq:3.3}
        N_l &:= \left\{
        \begin{aligned}
            &1 \ & \ &\text{if} \ |\lambda| < \sqrt{\tau}, \\
            &\max\pc{1, \left\lceil \dfrac{\log \pa{1 + \pd{\lambda} - (-1)^l \frac{\pd{\lambda}}{\lambda}\sqrt{\tau}} - \log\pa{2\sqrt{\tau}}}{\log\pa{\alpha + \beta\sqrt{\tau}}} \right\rceil} \ & \ &\text{if} \ |\lambda| > \sqrt{\tau}.
        \end{aligned}
        \right.
    \end{align}
\end{theorem}

\vspace{20pt}
\noindent By applying Theorem \ref{thm:1.3} and Theorem \ref{thm:1.4}, we characterize the set \(\mD_{A}(x)\) as below. \\

\begin{theorem}
    \label{thm:1.5}
    Let \(A \in \mZ(1)\). There exists positive integer \(\mL\) such that for any real number \(x \geq \mL\), the set \(\mD_{A}(x)\) is of the form
    \begin{align*}
        &\mD_{A}(x) \\
        &= \pc{\pa{1 + \dfrac{E}{D}\lambda- \dfrac{d}{2}, - \dfrac{E}{D}}, \pa{-1 + \dfrac{E}{D}\lambda- \dfrac{d}{2}, - \dfrac{E}{D}}} \\
        &\hspace{20pt} \cup \bigcup_{l=1}^{4}\pc{\pa{s_m^{(l)}, t_m^{(l)}} : 1 \leq m \leq \pe{\dfrac{\log\pa{\pe{x} - R_l + 1- Q_l} - \log P_l + \log \pa{2\sqrt{\tau}}}{\log \pa{\alpha + \beta\sqrt{\tau}}}}},
    \end{align*}
    where for \(l \in \{1,2,3,4\}\), \(P_l\) is a positive real number as defined in Equation \eqref{eq:1.15}, and  \(s_m^{(l)}\), \(t_m^{(l)}\), \(Q_l\) and \(R_l\) are integers as defined in Equation \eqref{eq:1.18}, Equation \eqref{eq:1.19}, Equation \eqref{eq:1.16} and Equation \eqref{eq:1.17} respectively. In particular, \(\mL \in \N\) is the same value as described in Theorem \ref{thm:1.4}.
\end{theorem}

\vspace{10pt}
\begin{remark}
    The results for \(\mD_{A}(x)\) and \(\pd{\mD_{A}(x)}\) are only true for all \(x \geq \mL\). In other words, we do not describe the case when \(x < \mL\). Nevertheless, the values of \(M'_l\) as defined in Equation \eqref{eq:3.1} are \textit{finite}, thus for any \(x < \mL\), all elements in the set \(\mD_{A}(x)\) (and its number of elements \(\pd{\mD_{A}(x)}\)) can be computed manually or by using computer. \\
\end{remark}

\subsection{Outline of the Paper} In Section 2, we show the proof of Theorem \ref{thm:1.3}. In Section 3, we show the positive integer \(N_0\) is well-defined as defined in Equation \eqref{eq:3.2} (Lemma \ref{lem:3.4}). This implies positive integers \(M'_l\) for \(l = 1,2,3,4\) are well-defined and therefore \(\mL \in \N\) is well-defined. In Section 4, we define functions \(s^{(l)}(m)\) and \(t^{(l)}(m)\) for \(l = 1,2,3,4\) with domain \(m \in [1,+\infty)\) such that the restriction to the positive integers are equivalent to Equation \eqref{eq:1.18} and Equation \eqref{eq:1.19} respectively. The aim is to show the function \(\pd{s^{(l)}(m)} + \pd{t^{(l)}(m)}\) is strictly increasing in domain \([N_0, +\infty)\), where \(N_0\) is defined as Equation \eqref{eq:3.2} (Lemma \ref{lem:4.3}(b)), and it tends to \(+\infty\) as \(m \to +\infty\) (Lemma \ref{lem:4.3}(c)). In Section 5, we introduce and show the existence of the value \(N'_l\), where \(l \in \{1,2,3,4\}\) is defined as Equation \eqref{eq:5.10}, and show that any integer \(1 \leq m \leq \pe{N'_l}\) satisfies the condition \(\pd{s^{(l)}(m)} + \pd{t^{(l)}(m)} \leq M'_l\) (Lemma \ref{lem:5.5}(c)), while any integer \(m \geq \pe{N'_l} + 1\) does not satisfy the condition (Lemma \ref{lem:5.5}(a)). In Section 6, we define sets \(\mD_{A}^{(k)}(x)\) for \(k = 0,1,2,3,4\) so that \(\mD_A(x) = \displaystyle\bigcup_{k=0}^4\mD_A^{(k)}(x)\). We show that \(\mD_{A}^{(k)}(x)\) are all disjoint sets for \(k = 0,1,2,3,4\) (Lemma \ref{lem:6.1}), so that \(\pd{\mD_A(x)}\) can be rewrite in terms of \(\pd{\mD_{A}^{(k)}(x)}\). In Section 7, we provide formula for \(\pd{\mD_{A}^{(l)}(x)}\) for all real numbers \(x \geq M'_l\), where \(l \in \{1,2,3,4\}\) (Lemma \ref{lem:7.5}). Moreover, we provide formula for \(\pd{\mD_{A}^{(0)}(x)}\) for all real numbers \(x \geq \mL := \max \Bigl\{M'_l : l \in \{1,2,3,4\}\Bigr\}\) (Lemma \ref{lem:7.7}). Finally, in Section 8, we show the proof of Theorem \ref{thm:1.4}, and in Section 9, we show the proof of Theorem \ref{thm:1.5}.

\vspace{5pt}
\section{Proof of Theorem \ref{thm:1.3}}
\begin{proof}
    \textit{Proof of} (a)\textit{.} From Definition \ref{def:1.2}, condition (iii) says that \(a\) divides \(b\), \(c\) and \(d\). We claim that \(a\) also divides both \(e\) and \(f\) so that \(\gcd(a,b,c,d,e,f) = a\), and since \(A \in \mA\), it implies \(a = 1\). \\

    Note that \(4a^2 \mid b^2\) and \(a \mid c\) by Definition \ref{def:1.2}(iii). So \(4a^2 \mid D = b^2-4ac\). Since \(D \mid E\) from Definition \ref{def:1.2}(ii), so \(4a^2 \mid E = bd - 2ae\). Also from Definition \ref{def:1.2}(iii), \(4a^2 \mid bd\), so \(4a^2 \mid 2ae\) which implies \(a \mid  e\). \\

    From Equation \eqref{eq:1.11}, \(N = E^2 - DF\). Since \(A \in \mZ(j)\), \(j = \dfrac{N}{-4a^2D}\). So \(N = E^2 - DF = -4a^2Dj\). Then \(F = \pa{\dfrac{E}{D}}E + 4a^2j\) where \(\dfrac{E}{D} \in \Z\) since \(D \mid E\) by assumption. By previous proof, \(4a^2 \mid E\), so \(4a^2 \mid F = d^2 - 4af\). From Definition \ref{def:1.2}(iii), \(2a \mid d\), so \(4a^2 \mid d^2\), therefore \(4a^2 \mid 4af\) which implies \(a \mid f\). The result is then follows. \\
    
    \textit{Proof of} (b)\textit{.} We will apply Lagrange's method as discussed in Section 1. From Equation \eqref{eq:1.3}, equation \(A\) can be rewrite as \(U^2 - DV^2 = N\), where \(U = Dv + E\) and \(V = 2au + bv + d\). Since \(A \in \mZ(j)\), \(j = \dfrac{N}{-4a^2D}\) which is equivalent to \(N = -4a^2Dj\). So
    \begin{align}
        \notag
        U^2 - DV^2 &= N, \\
        \notag
        \pa{Dv + E}^2 - D\pa{2au + bv + d}^2 &= -4a^2Dj, \\
        \label{eq:2.1}
        \pa{u + \dfrac{b}{2a}v + \dfrac{d}{2a}}^2 - \dfrac{D}{4a^2}\pa{v + \dfrac{E}{D}}^2 &= j.
    \end{align}
    From Theorem \ref{thm:1.3}(a), \(a = 1\), so Equation \eqref{eq:2.1} can be rewrite as equation
    \begin{align}
        \label{eq:2.2}
        \tilde{u}^2 - \tau\tilde{v}^2 = j,
    \end{align}
    where \(\pa{\tilde{u}, \tilde{v}} = \pa{u + \lambda v + \dfrac{d}{2}, v + \dfrac{E}{D}}\). \\
    
    \textit{Proof of} (c)\textit{.} Based on discussion in Section 1, Theorem \ref{thm:1.3}(b) implies that any solution \(\pa{\tilde{u},\tilde{v}} \in \Z^2\) to Equation \eqref{eq:2.2} (with \(j = 1\)) such that there are integers \(u,v\) for which \(\tilde{u} = u + \lambda v + \dfrac{d}{2}\) and \(\tilde{v} = v + \dfrac{E}{D}\) will provide a solution \((u,v) \in \Z^2\) to equation \(A \in \mZ(1)\). \\

    Notice that Equation \eqref{eq:2.2} (with \(j = 1\)) is equivalent to Equation \eqref{eq:1.6}. Hence, the set of integer solutions to Equation \eqref{eq:2.2} (with \(j = 1\)) is the set \(\mP_{\tau}\), where \(\mP_{\tau}\) is non-empty and has infinitely many elements. For any \(\pa{\tilde{u},\tilde{v}} \in \mP_{\tau}\), we obtain \textit{unique} solution \((u,v) \in \Z^2\) to equation \(A \in \mZ(1)\) where
    \begin{align*}
        \tilde{v} = v + \dfrac{E}{D} &\implies v = \tilde{v} - \dfrac{E}{D}, \\
        \tilde{u} = u + \lambda v + \dfrac{d}{2} &\implies u = \tilde{u} - \lambda\pa{\tilde{v} - \dfrac{E}{D}} - \dfrac{d}{2} = \tilde{u} - \lambda\tilde{v} + \dfrac{E}{D}\lambda- \dfrac{d}{2}.
    \end{align*}
    Therefore, the set of integer solutions to \(A \in \mZ(1)\) is the set \(\mD_{A}\) as defined in Equation \eqref{eq:1.12}.
\end{proof}

\vspace{5pt}
\section{Positive Integers \(M'_l\) for \(l = 1,2,3,4\)}
The main objective of this section is to show positive integer \(N_0\) is well-defined, so that \(M'_l\) is well-defined for all \(l = 1,2,3,4\), and therefore \(\mL \in \N\) is well-defined. In other words, we would like to show the existence of \(N_0\) as defined in Equation \eqref{eq:3.2}. We begin by reconsidering Equation \eqref{eq:1.9} and Equation \eqref{eq:1.10} as functions with domain \([1,+\infty)\). \\

Define functions \(u: [1, +\infty) \to [\alpha, +\infty)\) and \(v: [1, +\infty) \to [\beta, +\infty)\) with
\begin{align}
    \label{eq:3.4}
    u(m) &= \dfrac{1}{2}\Bigl\{\pa{\alpha + \beta\sqrt{\tau}}^m + \pa{\alpha - \beta\sqrt{\tau}}^m\Bigr\}, \\
    \label{eq:3.5}
    v(m) &= \dfrac{1}{2\sqrt{\tau}}\Bigl\{\pa{\alpha + \beta\sqrt{\tau}}^m - \pa{\alpha - \beta\sqrt{\tau}}^m\Bigr\},
\end{align}
where the restriction of both functions to the positive integers are values \(u_m\) and \(v_m\) as defined in Equation \eqref{eq:1.9} and Equation \eqref{eq:1.10} respectively.

\vspace{5pt}
\begin{lemma}
    \label{lem:3.1}
    Consider both functions \(u(m)\) and \(v(m)\) in Equation \eqref{eq:3.4} and Equation \eqref{eq:3.5} respectively. The following are all true:
    \begin{enumerate}[label=\normalfont(\alph*)]
        \item For any \(m \in [1,+\infty)\), \(u(m)^2 - \tau v(m)^2 = 1\).
        \item Both functions \(u(m)\) and \(v(m)\) are strictly increasing.
        \item \(\displaystyle\lim_{m \to +\infty}u(m) = \displaystyle\lim_{m \to +\infty}v(m) = +\infty\).
        \item The function \(w: [1, +\infty) \to [1, +\infty)\) with \(w(m) = \dfrac{u(m)}{v(m)}\) is strictly decreasing.
        \item \(\displaystyle\lim_{m \to +\infty}w(m) = \sqrt{\tau}\).
    \end{enumerate}
\end{lemma}

\begin{proof}
    \textit{Proof of} (a)\textit{.} By substituting Equation \eqref{eq:3.4} and Equation \eqref{eq:3.5},
    \begin{align*}
        &u(m)^2 - \tau v(m)^2 \\
        &= \pa{\dfrac{1}{2}\pc{\pa{\alpha + \beta\sqrt{\tau}}^m + \pa{\alpha - \beta\sqrt{\tau}}^m}}^2 \\
        &\hspace{100pt} - \tau \pa{\dfrac{1}{2\sqrt{\tau}}\pc{\pa{\alpha + \beta\sqrt{\tau}}^m - \pa{\alpha - \beta\sqrt{\tau}}^m}}^2,
    \end{align*}
    \begin{align*}
        &= \dfrac{1}{4}\pa{\pc{\pa{\alpha + \beta\sqrt{\tau}}^m + \pa{\alpha - \beta\sqrt{\tau}}^m}^2 - \pc{\pa{\alpha + \beta\sqrt{\tau}}^m - \pa{\alpha - \beta\sqrt{\tau}}^m}^2}, \\
        &= \dfrac{1}{4}\pc{2\pa{\alpha + \beta\sqrt{\tau}}^m}\pc{2\pa{\alpha - \beta\sqrt{\tau}}^m}, \\
        &= (\alpha^2 - \beta^2\tau)^m, \\
        &= 1. \\
    \end{align*}
    
    \textit{Proof of} (b)\textit{.} Consider the derivatives \(u'(m)\) and \(v'(m)\). By applying \(\alpha - \beta\sqrt{\tau} = \pa{\alpha + \beta\sqrt{\tau}}^{-1}\), substituting \(u(m)\) as Equation \eqref{eq:3.4} and substituting \(v(m)\) as Equation \eqref{eq:3.5}, we obtain
    \begin{align}
        \notag
        u'(m) &= \dfrac{1}{2}\pc{(\alpha + \beta\sqrt{\tau})^m\log(\alpha + \beta\sqrt{\tau}) + (\alpha - \beta\sqrt{\tau})^m\log(\alpha - \beta\sqrt{\tau})}, \\
        \notag
        &= \dfrac{1}{2}\pc{(\alpha + \beta\sqrt{\tau})^m\log(\alpha + \beta\sqrt{\tau}) - (\alpha - \beta\sqrt{\tau})^m\log(\alpha + \beta\sqrt{\tau})}, \\
        \notag
        &= \dfrac{1}{2}\log(\alpha + \beta\sqrt{\tau})
        \pc{\pa{\alpha + \beta\sqrt{\tau}}^m - \pa{\alpha - \beta\sqrt{\tau}}^m}, \\
        \label{eq:3.6}
         u'(m) &= \pa{\sqrt{\tau}}\log(\alpha + \beta\sqrt{\tau})v(m) > 0,
    \end{align}
    \begin{align}
        \notag
        v'(m) &= \dfrac{1}{2\sqrt{\tau}}\pc{(\alpha + \beta\sqrt{\tau})^m\log(\alpha + \beta\sqrt{\tau}) - (\alpha - \beta\sqrt{\tau})^m\log(\alpha - \beta\sqrt{\tau})}, \\
        \notag
        &= \dfrac{1}{2\sqrt{\tau}}\pc{(\alpha + \beta\sqrt{\tau})^m\log(\alpha + \beta\sqrt{\tau}) + (\alpha - \beta\sqrt{\tau})^m\log(\alpha + \beta\sqrt{\tau})}, \\
        \notag
        &= \dfrac{1}{2\sqrt{\tau}}\log(\alpha + \beta\sqrt{\tau})\pc{(\alpha + \beta\sqrt{\tau})^m + (\alpha - \beta\sqrt{\tau})^m}, \\
        \label{eq:3.7}
        v'(m) &= \pa{\dfrac{1}{\sqrt{\tau}}}\log(\alpha + \beta\sqrt{\tau})u(m) > 0.
    \end{align}
    Therefore, both functions \(u(m)\) and \(v(m)\) are strictly increasing. \\

    \textit{Proof of} (c)\textit{.} By Lemma \ref{lem:3.1}(b), both functions \(u(m)\) and \(v(m)\) are strictly increasing. This implies both sequences \((u_m)\) and \((v_m)\) are strictly increasing. Since \(u_m\) and \(v_m\) are integers for all \(m \in \N\), we must have \((u_m) \to +\infty\) and \((v_m) \to +\infty\). This concludes that \(\displaystyle\lim_{m \to +\infty}u(m) = +\infty\) and \(\displaystyle\lim_{m \to +\infty}v(m) = +\infty\). \\

    \textit{Proof of} (d)\textit{.} Consider the derivatives \(w'(m)\). By applying Equation \eqref{eq:3.6}, Equation \eqref{eq:3.7} and Lemma \ref{lem:3.1}(a),
    \begin{align*}
        w'(m) &= \dfrac{u'(m)v(m) - u(m)v'(m)}{v(m)^2}, \\
        &= \dfrac{\pa{\sqrt{\tau}}\log(\alpha + \beta\sqrt{\tau})v(m)^2 - \pa{\dfrac{1}{\sqrt{\tau}}}\log(\alpha + \beta\sqrt{\tau})u(m)^2}{v(m)^2},
    \end{align*}
    \begin{align*}
        &= -\pa{\dfrac{1}{v(m)^2\sqrt{\tau}}}\log(\alpha + \beta\sqrt{\tau})\pa{u(m)^2 - \tau v(m)^2}, \\
        &= -\pa{\dfrac{1}{v(m)^2\sqrt{\tau}}}\log(\alpha + \beta\sqrt{\tau}) < 0,
    \end{align*}
    thus function \(w(m)\) is strictly decreasing. \\

    \textit{Proof of} (e)\textit{.} Given \(m \in \N\), the sequence \(\pa{\dfrac{u_m}{v_m}}\) is strictly decreasing and \(\pa{\dfrac{u_m}{v_m}} \to \sqrt{\tau}\) (see \cite{Burton2010}, Theorem 15.4 and Theorem 15.15). Together with Lemma \ref{lem:3.1}(d), we have
    \[\displaystyle\lim_{m \to +\infty} w(m) = \displaystyle\lim_{m \to +\infty} \dfrac{u(m)}{v(m)} = \displaystyle\lim_{m \to +\infty} \dfrac{u_m}{v_m} = \sqrt{\tau}.\]
\end{proof}

\vspace{5pt}
\begin{lemma}
    \label{lem:3.2}
    \(0 < \alpha - \beta\sqrt{\tau} < 1\). This implies \(\displaystyle\lim_{m \to +\infty} \pa{\alpha - \beta\sqrt{\tau}}^m = 0\).
\end{lemma}

\begin{proof}
The lemma follows directly from \(\alpha + \beta\sqrt{\tau} > 1\) and \(\alpha - \beta\sqrt{\tau} = \pa{\alpha + \beta\sqrt{\tau}}^{-1}\).
\end{proof}

\vspace{5pt}
\subsection{Existence of Positive Integer \(N_0\)} We require Lemma \ref{lem:3.3} below:

\vspace{5pt}
\begin{lemma}
    \label{lem:3.3}
    There exists \(N_0' \in \N\) such that for all real numbers \(m \geq N_0'\),
    \[\left\{
    \begin{aligned}
        u(m) &> |\lambda|v(m) + \pd{\dfrac{E}{D}\lambda- \dfrac{d}{2}} \hspace{10pt} \text{and} \hspace{10pt} v(m) > \dfrac{\pd{E}}{D} \ & \ &\text{if} \ |\lambda| < \sqrt{\tau}, \\
        u(m) &< |\lambda|v(m) - \pd{\dfrac{E}{D}\lambda- \dfrac{d}{2}} \hspace{10pt} \text{and} \hspace{10pt} v(m) > \dfrac{\pd{E}}{D} \ & \ &\text{if} \ |\lambda| > \sqrt{\tau}.
    \end{aligned}
    \right.
    \]
\end{lemma}

\begin{proof}
    \textit{Proof of case \(|\lambda| < \sqrt{\tau}\).} We begin with showing the existence of \(\mN_0' \in \N\) so that \(v(m) > \dfrac{\pd{E}}{D}\) for all \(m \geq \mN_0'\). If \(E = 0\) then \(\mN_0' = 1\) will works since \(v(m) \geq v(1) > 0\) for all \(m \geq 1\). Suppose that \(|E|> 0\). From Lemma \ref{lem:3.1}(c), \(\displaystyle\lim_{m \to +\infty}v(m) = +\infty\), hence there exists \(\mN_0' \in \N\) such that \(v(m) > \dfrac{\pd{E}}{D}\) for all \(m \geq \mN_0'\). \\
    
    Next, we show the existence of \(\mN_0'' \in \N\) so that \(u(m) > |\lambda|v(m) + \pd{\dfrac{E}{D}\lambda- \dfrac{d}{2}}\) for all \(m \geq \mN_0''\). From Lemma \ref{lem:3.1}(d) and Lemma \ref{lem:3.1}(e), \(\dfrac{u(m)}{v(m)} > \sqrt{\tau}\) for all \(m \in (0, +\infty)\), which is \(\dfrac{u(m)}{v(m)} > |\lambda| + \pa{\sqrt{\tau} - |\lambda|}\). From Lemma \ref{lem:3.1}(c), \(\displaystyle\lim_{m \to +\infty}v(m) = +\infty\), this implies \(\displaystyle\lim_{m \to +\infty}\dfrac{1}{v(m)}\pd{\dfrac{E}{D}\lambda- \dfrac{d}{2}} = 0\). Therefore for \(\sqrt{\tau} - |\lambda| > 0\) (by assumption), there exists \(\mN_0'' \in \N\) such that \(\sqrt{\tau} - |\lambda| > \dfrac{1}{v(m)}\pd{\dfrac{E}{D}\lambda- \dfrac{d}{2}}\) for all \(m \geq \mN_0''\). So, we have
    \begin{align*}
        \dfrac{u(m)}{v(m)} > |\lambda| + \pa{\sqrt{\tau} - |\lambda|} &> |\lambda| + \dfrac{1}{v(m)}\pd{\dfrac{E}{D}\lambda- \dfrac{d}{2}}, \\
        \implies u(m) &> |\lambda|v(m) + \pd{\dfrac{E}{D}\lambda- \dfrac{d}{2}},
    \end{align*}
    for all \(m \geq \mN_0''\). The existence of \(N_0'\) is then by just letting \(N_0' = \max\pc{\mN_0', \mN_0''} \in \N\). \\

    \textit{Proof of case \(|\lambda| > \sqrt{\tau}\).} By following the same approach in the case \(|\lambda| < \sqrt{\tau}\), there exists \(\mN_0' \in \N\) such that \(v(m) > \dfrac{\pd{E}}{D}\) for all \(m \geq \mN_0'\). We claim that there exists \(\mN_0'' \in \N\) such that \(u(m) < |\lambda|v(m) - \pd{\dfrac{E}{D}\lambda- \dfrac{d}{2}}\) for all \(m \geq \mN_0''\). From Lemma \ref{lem:3.1}(c), \(\displaystyle\lim_{m \to +\infty}v(m) = +\infty\), this implies \(\displaystyle\lim_{m \to +\infty}\dfrac{1}{v(m)}\pd{\dfrac{E}{D}\lambda- \dfrac{d}{2}} = 0\). Therefore for \(\dfrac{|\lambda| - \sqrt{\tau}}{2} > 0\) (by assumption), there exists \(\mN_0''' \in \N\) such that for all \(m \geq \mN_0'''\),
    \[\dfrac{1}{v(m)}\pd{\dfrac{E}{D}\lambda- \dfrac{d}{2}} < \dfrac{|\lambda| - \sqrt{\tau}}{2}, \ \text{or equivalently,} \ -\pa{\dfrac{|\lambda| - \sqrt{\tau}}{2}} < -\dfrac{1}{v(m)}\pd{\dfrac{E}{D}\lambda- \dfrac{d}{2}}.\]
    By applying Lemma \ref{lem:3.1}(d) and Lemma \ref{lem:3.1}(e), \(\displaystyle\lim_{m \to +\infty} \pa{\dfrac{u(m)}{v(m)} - \sqrt{\tau}} \to 0\). Hence for \(\dfrac{|\lambda| - \sqrt{\tau}}{2} > 0\) (by assumption), there exists \(\mN_0'' \in \N\), with \(\mN_0'' \geq \mN_0'''\) such that \(\dfrac{u(m)}{v(m)} - \sqrt{\tau} < \dfrac{|\lambda| - \sqrt{\tau}}{2}\) for all \(m \geq \mN_0''\). So we have
    \begin{align*}
        \dfrac{u(m)}{v(m)} - \sqrt{\tau} < \dfrac{|\lambda| - \sqrt{\tau}}{2} &= \pa{|\lambda| - \sqrt{\tau}} - \pa{\dfrac{|\lambda| - \sqrt{\tau}}{2}}, \\
        &< \pa{|\lambda| - \sqrt{\tau}} - \dfrac{1}{v(m)}\pd{\dfrac{E}{D}\lambda- \dfrac{d}{2}}, \\
        \implies u(m) &< |\lambda|v(m) - \pd{\dfrac{E}{D}\lambda- \dfrac{d}{2}},
    \end{align*}
    for all \(m \geq \mN_0''\). The existence of \(N_0'\) is then by just letting \(N_0' = \max\pc{\mN_0', \mN_0''} \in \N\).
\end{proof}

\vspace{5pt}
\begin{lemma}
    \label{lem:3.4}
    Positive integer \(N_0 \in \N\) is well-defined as defined in Equation \eqref{eq:3.2}.
\end{lemma}

\begin{proof}
    We show the existence of \(N_0\). Consider the set \(\mS_0\):
    \[\mS_0 = \left\{
    \begin{aligned}
        &\pc{m \in \N : u_m > |\lambda|v_m + \pd{\dfrac{E}{D}\lambda- \dfrac{d}{2}}, v_m > \dfrac{\pd{E}}{D}} \ & \ &\text{if} \ |\lambda| < \sqrt{\tau}, \\
        &\pc{m \in \N : u_m < |\lambda|v_m - \pd{\dfrac{E}{D}\lambda- \dfrac{d}{2}}, v_m > \dfrac{\pd{E}}{D}} \ & \ &\text{if} \ |\lambda| > \sqrt{\tau}.
    \end{aligned}
    \right.\]
    Note that \(\mS_0 \subseteq \N\) and \(\mS_0 \neq \varnothing\) (by Lemma \ref{lem:3.3}). By well-ordering principle, \(\min \mS_0\) exists, which is exactly \(N_0 \in \N\).
\end{proof}

\vspace{5pt}
\subsection{Lemmas Related to Positive Integer \(N_l\)} We end this section by listing a few more lemmas (Lemma \ref{lem:3.5}, Lemma \ref{lem:3.6}, Lemma \ref{lem:3.7}, and Lemma \ref{lem:3.8}). Lemma \ref{lem:3.8} will be useful in Section 7, Lemma \ref{lem:7.3}.

\vspace{5pt}
\begin{lemma}
    \label{lem:3.5}
    If \(|\lambda| < \sqrt{\tau}\), then for any \(m \in [1, +\infty)\),
    \begin{align*}
        \pd{\dfrac{1}{2\sqrt{\tau}}\pa{1 \pm \lambda - \sqrt{\tau}}\pa{\alpha - \beta\sqrt{\tau}}^m} < 1.
    \end{align*}
\end{lemma}

\begin{proof}
    By Lemma \ref{lem:3.2}, \(0 < \alpha - \beta\sqrt{\tau} < 1\) and \(\displaystyle\lim_{m \to +\infty} \pa{\alpha - \beta\sqrt{\tau}}^m = 0\), hence for any \(m \in [1, +\infty)\),
    \[\pd{\dfrac{1}{2\sqrt{\tau}}\pa{1 \pm \lambda - \sqrt{\tau}}\pa{\alpha - \beta\sqrt{\tau}}^m} \leq \pd{\dfrac{1}{2\sqrt{\tau}}\pa{1 \pm \lambda - \sqrt{\tau}}\pa{\alpha - \beta\sqrt{\tau}}}.\]
    So, it is sufficient to show \(\pd{\dfrac{1}{2\sqrt{\tau}}\pa{1 \pm \lambda - \sqrt{\tau}}\pa{\alpha - \beta\sqrt{\tau}}} < 1\). Since \(-\sqrt{\tau} < \lambda < \sqrt{\tau}\) (by assumption), then \(-\sqrt{\tau} < -\lambda < \sqrt{\tau}\). This follows that
    \begin{align*}
        -\sqrt{\tau} < \pm\lambda < \sqrt{\tau} &\implies 1-2\sqrt{\tau} < 1\pm\lambda-\sqrt{\tau} < 1, \\
        &\implies \dfrac{1}{2\sqrt{\tau}}\pa{1-2\sqrt{\tau}} < \dfrac{1}{2\sqrt{\tau}}\pa{1\pm\lambda-\sqrt{\tau}} < \dfrac{1}{2\sqrt{\tau}}, \\
        &\implies \dfrac{1}{2\sqrt{\tau}}-1 < \dfrac{1}{2\sqrt{\tau}}\pa{1\pm\lambda-\sqrt{\tau}} < \dfrac{1}{2\sqrt{\tau}}, \\
        &\implies -1 < \dfrac{1}{2\sqrt{\tau}}\pa{1 \pm \lambda-\sqrt{\tau}} < 1, \\
        &\implies \pd{\dfrac{1}{2\sqrt{\tau}}\pa{1 \pm \lambda - \sqrt{\tau}}} < 1, \\
        &\implies \pd{\dfrac{1}{2\sqrt{\tau}}\pa{1 \pm \lambda - \sqrt{\tau}}\pa{\alpha - \beta\sqrt{\tau}}} < 1.
    \end{align*}
\end{proof}

\vspace{5pt}
\begin{lemma}
    \label{lem:3.6}
    For any \(l \in \{1,2,3,4\}\), the value
    \[\dfrac{\log \pa{1 + \pd{\lambda} - (-1)^l \frac{\pd{\lambda}}{\lambda}\sqrt{\tau}} - \log\pa{2\sqrt{\tau}}}{\log\pa{\alpha + \beta\sqrt{\tau}}},\]
    is never an integer.
\end{lemma}

\begin{proof}
    Suppose that \(\dfrac{\log \pa{1 + \pd{\lambda} - (-1)^l \frac{\pd{\lambda}}{\lambda}\sqrt{\tau}} - \log\pa{2\sqrt{\tau}}}{\log\pa{\alpha + \beta\sqrt{\tau}}} = K\) for some \(K \in \Z\). This follows that
    \begin{align*}
        \log \pa{\dfrac{1}{2\sqrt{\tau}}\pa{1 + \pd{\lambda} - (-1)^l \frac{\pd{\lambda}}{\lambda}\sqrt{\tau}}} &= \log \pa{\alpha + \beta\sqrt{\tau}}^K, \\
        \implies \dfrac{1}{2\sqrt{\tau}}\pa{1 + \pd{\lambda} - (-1)^l \frac{\pd{\lambda}}{\lambda}\sqrt{\tau}} &= \pa{\alpha + \beta\sqrt{\tau}}^K,
    \end{align*}
    which implies
    \begin{align}
        \label{eq:3.8}
        1 + \pd{\lambda} &= \pa{2\sqrt{\tau}}\pa{\alpha + \beta\sqrt{\tau}}^K + (-1)^l \frac{\pd{\lambda}}{\lambda}\sqrt{\tau}.
    \end{align}
    From \(K \in \Z\) and \(\pa{\alpha + \beta\sqrt{\tau}}^{-1} = \alpha - \beta\sqrt{\tau}\), we have \(\pa{\alpha + \beta\sqrt{\tau}}^K = A + B\sqrt{\tau}\) for some \(A, B \in \Z\). Substituting into Equation \eqref{eq:3.8}, we get
    \begin{align*}
        1 + \pd{\lambda} &= \pa{2\sqrt{\tau}}\pa{A + B\sqrt{\tau}} + (-1)^l \frac{\pd{\lambda}}{\lambda}\sqrt{\tau} = 2B \tau + \pa{2A + (-1)^l \frac{\pd{\lambda}}{\lambda}}\sqrt{\tau}.
    \end{align*}
    Notice that \(1 + \pd{\lambda} \in \Z\), so we must have \(2A + (-1)^l \dfrac{\pd{\lambda}}{\lambda} = 0\). But this implies \(A = -\dfrac{(-1)^l \frac{\pd{\lambda}}{\lambda}}{2} \not\in \Z\), which is a contradiction.
\end{proof}

\vspace{5pt}
\begin{lemma}
    \label{lem:3.7}
    Let \(l \in \{1,2,3,4\}\) and \(N_l\) as defined in Equation \eqref{eq:3.3}. If \(|\lambda| > \sqrt{\tau}\), then
    \begin{align}
        \label{eq:3.9}
        \pd{\dfrac{1}{2\sqrt{\tau}}\pa{1 + \pd{\lambda} - (-1)^l \frac{\pd{\lambda}}{\lambda}\sqrt{\tau}}\pa{\alpha - \beta\sqrt{\tau}}^{N_l}} < 1.
    \end{align}
\end{lemma}

\begin{proof}
    From assumption, \(2\sqrt{\tau} > 0\) and \(1 + \pd{\lambda} - (-1)^l \dfrac{\pd{\lambda}}{\lambda}\sqrt{\tau} > 0\). By referring to definition of \(N_l\) as Equation \eqref{eq:3.3} and Lemma \ref{lem:3.6}, we have
    \begin{align}
        \notag
        N_l &\geq \left\lceil\dfrac{\log \pa{1 + \pd{\lambda} - (-1)^l \frac{\pd{\lambda}}{\lambda}\sqrt{\tau}} - \log\pa{2\sqrt{\tau}}}{\log\pa{\alpha + \beta\sqrt{\tau}}}\right\rceil, \\
        \label{eq:3.10}
        &> \dfrac{\log \pa{1 + \pd{\lambda} - (-1)^l \frac{\pd{\lambda}}{\lambda}\sqrt{\tau}} - \log\pa{2\sqrt{\tau}}}{\log\pa{\alpha + \beta\sqrt{\tau}}}.
    \end{align}
    By applying \(\alpha + \beta\sqrt{\tau} = \pa{\alpha - \beta\sqrt{\tau}}^{-1}\) and note that \(\alpha + \beta\sqrt{\tau} > 1\), Equation \eqref{eq:3.10} becomes
    \begin{align}
        \notag
        \log\pa{\alpha + \beta\sqrt{\tau}}^{N_l} &> \log \pa{\dfrac{1}{2\sqrt{\tau}}\pa{1 + \pd{\lambda} - (-1)^l \frac{\pd{\lambda}}{\lambda}\sqrt{\tau}}}, \\
        \notag
        \implies \pa{\alpha + \beta\sqrt{\tau}}^{N_l} &> \dfrac{1}{2\sqrt{\tau}}\pa{1 + \pd{\lambda} - (-1)^l \frac{\pd{\lambda}}{\lambda}\sqrt{\tau}}, \\
        \label{eq:3.11}
        \implies \pa{\alpha - \beta\sqrt{\tau}}^{-N_l} &> \dfrac{1}{2\sqrt{\tau}}\pa{1 + \pd{\lambda} - (-1)^l \frac{\pd{\lambda}}{\lambda}\sqrt{\tau}}.
    \end{align}
    Lemma \ref{lem:3.2} implies \(\pa{\alpha - \beta\sqrt{\tau}}^{N_l} > 0\). Multiplying both sides of Equation \eqref{eq:3.11} by \(\pa{\alpha - \beta\sqrt{\tau}}^{N_l}\) and taking absolute value to both sides, we get Equation \eqref{eq:3.9}.
\end{proof}

\vspace{5pt}
\begin{lemma}
    \label{lem:3.8}
    Let \(l \in \{1,2,3,4\}\) and \(N_l\) as defined in Equation \eqref{eq:3.3}. For all real numbers \(m \geq N_l\),
    \begin{numcases}{}
        \label{eq:3.12}
        \pd{\dfrac{1}{2\sqrt{\tau}}\pa{1 - \lambda  + (-1)^l\sqrt{\tau}}\pa{\alpha - \beta\sqrt{\tau}}^m} < 1 & if \ $\lambda < (-1)^{l-1}\sqrt{\tau}$, \\
        \label{eq:3.13}
        \pd{\dfrac{1}{2\sqrt{\tau}}\pa{1 + \lambda - (-1)^l\sqrt{\tau}}\pa{\alpha - \beta\sqrt{\tau}}^m} < 1 & if \ $\lambda > (-1)^{l-1}\sqrt{\tau}$.
    \end{numcases}
\end{lemma}

\begin{proof}
    Let \(|\lambda| < \sqrt{\tau}\). By Equation \eqref{eq:3.3}, \(N_l = 1\), hence we require to show:
    \vspace{5pt}
    \begin{enumerate}[(i)]
        \item If \(l \in \{1,3\}\), then Equation \eqref{eq:3.12} is true for all \(m \in [1,+\infty)\).
        \item If \(l \in \{2,4\}\), then Equation \eqref{eq:3.13} is true for all \(m \in [1,+\infty)\).
    \end{enumerate}
    \vspace{5pt}
    Both (i) and (ii) above are just the statement in Lemma \ref{lem:3.5}, so the result follows. \\
    
    Let \(\lambda < -\sqrt{\tau}\). So \(|\lambda| = - \lambda\) and therefore by Lemma \ref{lem:3.7}, \(N_l\) satisfies Equation \eqref{eq:3.12}. By Lemma \ref{lem:3.2}, \(\pa{\alpha - \beta\sqrt{\tau}}^m\) is strictly decreasing for \(m \geq N_l\). Since Equation \eqref{eq:3.12} is true for \(m = N_l\), then it must be true for all \(m \geq N_l\). Now let \(\lambda > \sqrt{\tau}\). So \(|\lambda| = \lambda\) and therefore by Lemma \ref{lem:3.7}, \(N_l\) satisfies Equation \eqref{eq:3.13}. By Lemma \ref{lem:3.2}, \(\pa{\alpha - \beta\sqrt{\tau}}^m\) is strictly decreasing for \(m \geq N_l\). Since Equation \eqref{eq:3.13} is true for \(m = N_l\), then it must be true for all \(m \geq N_l\).
\end{proof}

\vspace{5pt}
\section{Strictly Increasing Functions \(\pd{s^{(l)}(m)} + \pd{t^{(l)}(m)}\) where \(l \in \{1,2,3,4\}\)}
We begin this section by reconsidering Equation \eqref{eq:1.18} and Equation \eqref{eq:1.19} as functions with domain \([1,+\infty)\). \\

Let \(l \in \{1,2,3,4\}\), and define functions \(s^{(l)}: [1, +\infty) \to \R\) and \(t^{(l)}: [1, +\infty) \to \R\) with
\begin{align}
    \label{eq:4.1}
    s^{(l)}(m) &= (-1)^{\pe{\frac{l}{2}}}u(m) - (-1)^{\pe{\frac{l-1}{2}}}\lambda v(m) + \dfrac{E}{D}\lambda- \dfrac{d}{2}, \\
    \label{eq:4.2}
    t^{(l)}(m) &= (-1)^{\pe{\frac{l-1}{2}}}v(m) - \dfrac{E}{D},
\end{align}
where \(u(m)\) and \(v(m)\) are functions as defined in Equation \eqref{eq:3.4} and Equation \eqref{eq:3.5} respectively. Note that the restriction of both functions \(s^{(l)}(m)\) and \(t^{(l)}(m)\) to the positive integers are values \(s_m^{(l)}\) and \(t_m^{(l)}\) as defined in Equation \eqref{eq:1.18} and Equation \eqref{eq:1.19} respectively.

\vspace{5pt}
\begin{lemma}
    \label{lem:4.1}
    Let \(l \in \{1,2,3,4\}\), and consider functions \(s^{(l)}(m)\) and \(t^{(l)}(m)\) as in Equation \eqref{eq:4.1} and Equation \eqref{eq:4.2} respectively. The following are all true:
    \begin{enumerate}[label=\normalfont(\alph*)]
        \item If \(\lambda < \sqrt{\tau}\), then function \(s^{(1)}(m)\) is strictly increasing and function \(s^{(3)}(m)\) is strictly decreasing.
        \item If \(\lambda > \sqrt{\tau}\), then function \(s^{(1)}(m)\) is strictly decreasing and function \(s^{(3)}(m)\) is strictly increasing.
        \item If \(\lambda < -\sqrt{\tau}\), then function \(s^{(2)}(m)\) is strictly increasing and function \(s^{(4)}(m)\) is strictly decreasing.
        \item If \(\lambda > -\sqrt{\tau}\), then function \(s^{(2)}(m)\) is strictly decreasing and function \(s^{(4)}(m)\) is strictly increasing.
        \item Function \(t^{(l)}(m)\) is strictly increasing for \(l \in \{1,2\}\) and is strictly decreasing for \(l \in \{3,4\}\).
    \end{enumerate}
\end{lemma}

\begin{proof}
    \textit{Proof of} (a)\textit{.} If \(\lambda = 0\), then from Equation \eqref{eq:4.1} and Equation \eqref{eq:3.6}, the derivatives of \(s^{(1)}(m)\) and \(s^{(3)}(m)\) will be just
    \[\pa{s^{(1)}(m)}^{'} = u'(m) > 0 \hspace{20pt} \text{and} \hspace{20pt} \pa{s^{(3)}(m)}^{'} = -u'(m) < 0.\]
    Therefore, function \(s^{(1)}(m)\) is strictly increasing and function \(s^{(3)}(m)\) is strictly decreasing. \\
    
    Let \(\lambda < 0\). From Equation \eqref{eq:4.1}, consider the derivative of \(s^{(1)}(m)\) by applying the derivative \(u'(m)\) in Equation \eqref{eq:3.6} and the derivative \(v'(m)\) in Equation \eqref{eq:3.7}:
    \begin{align}
        \label{eq:4.3}
        \pa{s^{(1)}(m)}^{'} &= u'(m) - \lambda v'(m), \\
        \notag
        &= \pa{\sqrt{\tau}}\log(\alpha + \beta\sqrt{\tau})v(m) - \lambda \pa{\dfrac{1}{\sqrt{\tau}}}\log(\alpha + \beta\sqrt{\tau})u(m), \\
        \label{eq:4.4}
        &= \pa{\dfrac{-\lambda}{\sqrt{\tau}}}v(m)\log\pa{\alpha + \beta\sqrt{\tau}}\pa{\dfrac{u(m)}{v(m)} - \dfrac{\tau}{\lambda}}.
    \end{align}
    We have \(\pa{\dfrac{-\lambda}{\sqrt{\tau}}} > 0\) and \(v(m)\log\pa{\alpha + \beta\sqrt{\tau}} > 0\). From Lemma \ref{lem:3.1}(d) and Lemma \ref{lem:3.1}(e), \(\dfrac{u(m)}{v(m)} - \sqrt{\tau} > 0\). Since \(\lambda < \sqrt{\tau}\) and \(\lambda < 0\) by assumption, then \(-\sqrt{\tau} < -\dfrac{\tau}{\lambda}\). This gives us
    \[0 < \dfrac{u(m)}{v(m)} - \sqrt{\tau} < \dfrac{u(m)}{v(m)} - \dfrac{\tau}{\lambda},\]
    hence Equation \eqref{eq:4.4} is positive and so \(\pa{s^{(1)}(m)}^{'} > 0\). Consider the derivative of \(s^{(3)}(m)\). From Equation \eqref{eq:4.1} and Equation \eqref{eq:4.3},
    \[\pa{s^{(3)}(m)}^{'} = -u'(m) + \lambda v'(m) = -\pa{u'(m) - \lambda v'(m)} = -\pa{s^{(1)}(m)}^{'} < 0.\]
    Therefore, function \(s^{(1)}(m)\) is strictly increasing and function \(s^{(3)}(m)\) is strictly decreasing. \\

    Let \(\lambda > 0\). Since \(\lambda < \sqrt{\tau}\), then \(\lambda^2 < \tau\) which implies \(\lambda^2 \leq \tau - 1\) since \(\lambda \in \Z\). From Lemma \ref{lem:3.1}(d) and Lemma \ref{lem:3.1}(e), \(\dfrac{u(m)}{v(m)} - \sqrt{\tau} > 0\) and so \(u(m) - v(m)\sqrt{\tau} > 0\). We then have
    \begin{align*}
        u(m) &> v(m)\sqrt{\tau}, \\
        u(m) &> \sqrt{\tau} \ & \ &\text{since} \ v(m) \geq v(1) = \beta \geq 1 \ \text{for all} \ m \in [1, +\infty), \\
        u(m)^2 &> \tau, \\
        u(m)^2 &> \tau\pc{u(m)^2 - \tau v(m)^2} \ & \ &\text{since} \ u(m)^2 - \tau v(m)^2 = 1 \ \text{from Lemma \ref{lem:3.1}(a)}, \\
        \tau^2v(m)^2 &> u(m)^2 \pa{\tau - 1}, \\
        \tau^2v(m)^2 &> u(m)^2 \pa{\lambda^2} \ & \ &\text{since} \ \lambda^2 \leq \tau - 1 \ \text{from assumption,} \\
        \tau v(m) &> u(m) \lambda, \\
        \dfrac{\tau v(m)}{u(m)} &> \lambda,
    \end{align*}
    for all \(m \in [1, +\infty)\). By applying the derivative \(u'(m)\) in Equation \eqref{eq:3.6} and the derivative \(v'(m)\) in Equation \eqref{eq:3.7}:
    \[\lambda < \dfrac{\tau v(m)}{u(m)} = \dfrac{\sqrt{\tau} \log\pa{\alpha + \beta\sqrt{\tau}}v(m)}{\pa{\dfrac{1}{\sqrt{\tau}}}\log\pa{\alpha + \beta \sqrt{\tau}}u(m)} = \dfrac{u'(m)}{v'(m)},\]
    which implies \(u'(m) - \lambda v'(m) > 0\), so by Equation \eqref{eq:4.3}, we get \(\pa{s^{(1)}(m)}^{'} = u'(m) - \lambda v'(m) > 0\). Also, \(\pa{s^{(3)}(m)}^{'} = -\pa{s^{(1)}(m)}^{'} < 0\). Therefore, function \(s^{(1)}(m)\) is strictly increasing and function \(s^{(3)}(m)\) is strictly decreasing. \\

    \textit{Proof of} (b)\textit{.} From Equation \eqref{eq:4.4},
    \begin{align}
        \label{eq:4.5}
        \pa{s^{(1)}(m)}^{'} = -\pa{\dfrac{\lambda}{\sqrt{\tau}}}v(m)\log\pa{\alpha + \beta\sqrt{\tau}}\pa{\dfrac{u(m)}{v(m)} - \dfrac{\tau}{\lambda}}.
    \end{align}
    We have \(-\pa{\dfrac{\lambda}{\sqrt{\tau}}} < 0\) and \(v(m)\log\pa{\alpha + \beta\sqrt{\tau}} > 0\). From Lemma \ref{lem:3.1}(d) and Lemma \ref{lem:3.1}(e), \(\dfrac{u(m)}{v(m)} - \sqrt{\tau} > 0\). Since \(\lambda > \sqrt{\tau}\) by assumption, then \(-\sqrt{\tau} < -\dfrac{\tau}{\lambda}\). This gives us
    \[0 < \dfrac{u(m)}{v(m)} - \sqrt{\tau} < \dfrac{u(m)}{v(m)} - \dfrac{\tau}{\lambda},\]
    hence Equation \eqref{eq:4.5} is negative and so \(\pa{s^{(1)}(m)}^{'} < 0\). Also,
    \[\pa{s^{(3)}(m)}^{'} = -\pa{s^{(1)}(m)}^{'} > 0.\]
    Therefore, function \(s^{(1)}(m)\) is strictly decreasing and function \(s^{(3)}(m)\) is strictly increasing. \\

    \textit{Proof of} (c)\textit{.} From Equation \eqref{eq:4.1}, consider the derivative of \(s^{(2)}(m)\) by applying the derivative \(u'(m)\) in Equation \eqref{eq:3.6} and the derivative \(v'(m)\) in Equation \eqref{eq:3.7}:
    \begin{align}
        \label{eq:4.6}
        \pa{s^{(2)}(m)}^{'} &= -u'(m) - \lambda v'(m), \\
        \notag
        &= -\pa{\sqrt{\tau}}\log(\alpha + \beta\sqrt{\tau})v(m) - \lambda \pa{\dfrac{1}{\sqrt{\tau}}}\log(\alpha + \beta\sqrt{\tau})u(m), \\
        \label{eq:4.7}
        &= \pa{\dfrac{-\lambda}{\sqrt{\tau}}}v(m)\log\pa{\alpha + \beta\sqrt{\tau}}\pa{\dfrac{u(m)}{v(m)} + \dfrac{\tau}{\lambda}}.
    \end{align}
    We have \(\pa{\dfrac{-\lambda}{\sqrt{\tau}}} > 0\) and \(v(m)\log\pa{\alpha + \beta\sqrt{\tau}} > 0\). From Lemma \ref{lem:3.1}(d) and Lemma \ref{lem:3.1}(e), \(\dfrac{u(m)}{v(m)} - \sqrt{\tau} > 0\). Since \(\lambda < -\sqrt{\tau}\) by assumption, then \(-\sqrt{\tau} < \dfrac{\tau}{\lambda}\). This gives us
    \[0 < \dfrac{u(m)}{v(m)} - \sqrt{\tau} < \dfrac{u(m)}{v(m)} + \dfrac{\tau}{\lambda},\]
    hence Equation \eqref{eq:4.7} is positive and so \(\pa{s^{(2)}(m)}^{'} > 0\). Consider the derivative of \(s^{(4)}(m)\). From Equation \eqref{eq:4.1} and Equation \eqref{eq:4.6},
    \[\pa{s^{(4)}(m)}^{'} = u'(m) + \lambda v'(m) = -\pa{-u'(m) - \lambda v'(m)} = -\pa{s^{(2)}(m)}^{'} < 0.\]
    Therefore, function \(s^{(2)}(m)\) is strictly increasing and function \(s^{(4)}(m)\) is strictly decreasing. \\

    \textit{Proof of} (d)\textit{.} If \(\lambda = 0\), then from Equation \eqref{eq:4.1} and Equation \eqref{eq:3.6}, the derivatives of \(s^{(2)}(m)\) and \(s^{(4)}(m)\) will be just
    \[\pa{s^{(2)}(m)}^{'} = -u'(m) < 0 \hspace{20pt} \text{and} \hspace{20pt} \pa{s^{(4)}(m)}^{'} = u'(m) > 0.\]
    Therefore, function \(s^{(2)}(m)\) is strictly decreasing and function \(s^{(4)}(m)\) is strictly increasing. \\
    
    Let \(\lambda > 0\). From Equation \eqref{eq:4.7},
    \begin{align}
        \label{eq:4.8}
        \pa{s^{(2)}(m)}^{'} = -\pa{\dfrac{\lambda}{\sqrt{\tau}}}v(m)\log\pa{\alpha + \beta\sqrt{\tau}}\pa{\dfrac{u(m)}{v(m)} + \dfrac{\tau}{\lambda}}.
    \end{align}
    We have \(-\pa{\dfrac{\lambda}{\sqrt{\tau}}} < 0\) and \(v(m)\log\pa{\alpha + \beta\sqrt{\tau}} > 0\). From Lemma \ref{lem:3.1}(d) and Lemma \ref{lem:3.1}(e), \(\dfrac{u(m)}{v(m)} - \sqrt{\tau} > 0\). Since \(\lambda > -\sqrt{\tau}\) and \(\lambda > 0\) by assumption, then \(-\sqrt{\tau} < \dfrac{\tau}{\lambda}\). This gives us
    \[0 < \dfrac{u(m)}{v(m)} - \sqrt{\tau} < \dfrac{u(m)}{v(m)} + \dfrac{\tau}{\lambda},\]
    hence Equation \eqref{eq:4.8} is negative and so \(\pa{s^{(2)}(m)}^{'} < 0\). Also,
    \[\pa{s^{(4)}(m)}^{'} = -\pa{s^{(2)}(m)}^{'} > 0.\]
    Therefore, function \(s^{(2)}(m)\) is strictly decreasing and function \(s^{(4)}(m)\) is strictly increasing. \\

    Let \(\lambda < 0\). From \(\lambda > -\sqrt{\tau}\), then \(\lambda^2 < \tau\) which implies \(\lambda^2 \leq \tau - 1\) since \(\lambda \in \Z\). From Lemma \ref{lem:3.1}(d) and Lemma \ref{lem:3.1}(e), \(\dfrac{u(m)}{v(m)} - \sqrt{\tau} > 0\) and so \(u(m) - v(m)\sqrt{\tau} > 0\). We then have
    \begin{align*}
        u(m) &> v(m)\sqrt{\tau}, \\
        u(m) &> \sqrt{\tau} \ & \ &\text{since} \ v(m) \geq v(1) = \beta \geq 1 \ \text{for all} \ m \in [1, +\infty), \\
        u(m)^2 &> \tau, \\
        u(m)^2 &> \tau\pc{u(m)^2 - \tau v(m)^2} \ & \ &\text{since} \ u(m)^2 - \tau v(m)^2 = 1 \ \text{from Lemma \ref{lem:3.1}(a)}, \\
        \tau^2v(m)^2 &> u(m)^2 \pa{\tau - 1}, \\
        \tau^2v(m)^2 &> u(m)^2 \pa{\lambda^2} \ & \ &\text{since} \ \lambda^2 \leq \tau - 1 \ \text{from assumption,} \\
        \tau v(m) &> u(m) \pa{-\lambda}, \\
        \dfrac{\tau v(m)}{u(m)} &> -\lambda,
    \end{align*}
    for all \(m \in [1, +\infty)\). By applying the derivative \(u'(m)\) in Equation \eqref{eq:3.6} and the derivative \(v'(m)\) in Equation \eqref{eq:3.7}:
    \[-\lambda < \dfrac{\tau v(m)}{u(m)} = \dfrac{\sqrt{\tau} \log\pa{\alpha + \beta\sqrt{\tau}}v(m)}{\pa{\dfrac{1}{\sqrt{\tau}}}\log\pa{\alpha + \beta \sqrt{\tau}}u(m)} = \dfrac{u'(m)}{v'(m)},\]
    which implies \(-u'(m) - \lambda v'(m) < 0\), so by Equation \eqref{eq:4.6}, we get \(\pa{s^{(2)}(m)}^{'} = -u'(m) - \lambda v'(m) < 0\). Also, \(\pa{s^{(4)}(m)}^{'} = -\pa{s^{(2)}(m)}^{'} > 0\). Therefore, function \(s^{(2)}(m)\) is strictly decreasing and function \(s^{(4)}(m)\) is strictly increasing. \\

    \textit{Proof of} (e)\textit{.} From Equation \eqref{eq:4.2} and Equation \eqref{eq:3.7}, the derivatives of \(t^{(l)}(m)\) are \(\pa{t^{(l)}(m)}^{'} = v'(m) > 0\) for \(l \in \{1,2\}\) and \(\pa{t^{(l)}(m)}^{'} = -v'(m) < 0\) for \(l \in \{3,4\}\). Therefore, function \(t^{(l)}(m)\) is strictly increasing for \(l \in \{1,2\}\) and function \(t^{(l)}(m)\) is strictly decreasing for \(l \in \{3,4\}\).
\end{proof}

\vspace{5pt}
\begin{lemma}
    \label{lem:4.2}
    Let \(l \in \{1,2,3,4\}\) and \(N_0 \in \N\) as defined in Equation \eqref{eq:3.2}. The following are all true:
    \begin{enumerate}[label=\normalfont(\alph*)]
        \item If function \(s^{(l)}(m)\) is strictly increasing, then \(s^{(l)}(m) > 0\) for all real numbers \(m \geq N_0\). Otherwise if function \(s^{(l)}(m)\) is strictly decreasing, then \(s^{(l)}(m) < 0\) for all real numbers  \(m \geq N_0\).
        \item If function \(t^{(l)}(m)\) is strictly increasing, then \(t^{(l)}(m) > 0\) for all real numbers \(m \geq N_0\). Otherwise if function \(t^{(l)}(m)\) is strictly decreasing, then \(t^{(l)}(m) < 0\) for all real numbers \(m \geq N_0\).
    \end{enumerate}
\end{lemma}

\begin{proof}
    \textit{Proof of} (a)\textit{.} Notice that if \(s^{(l)}(m)\) is strictly increasing, it is sufficient to show \(s^{(l)}(N_0) = s^{(l)}_{N_0} > 0\), and similarly if \(s^{(l)}(m)\) is strictly decreasing, it is sufficient to show \(s^{(l)}(N_0) = s^{(l)}_{N_0} < 0\). \\
    
    Let \(l = 1\), and suppose that \(s^{(1)}(m)\) is strictly increasing. Then it must be the case \(\lambda < \sqrt{\tau}\), otherwise if not, then \(\lambda > \sqrt{\tau}\) and by Lemma \ref{lem:4.1}(b), function \(s^{(1)}(m)\) is strictly decreasing which is a contradiction. Hence, we have either \(|\lambda| < \sqrt{\tau}\) or \(\lambda < -\sqrt{\tau}\). If \(|\lambda| < \sqrt{\tau}\), then based on Equation \eqref{eq:3.2}, \(N_0\) satisfies
    \begin{align}
        \label{eq:4.9}
        u_{N_0} > |\lambda|v_{N_0} + \pd{\dfrac{E}{D}\lambda- \dfrac{d}{2}} \geq \lambda v_{N_0} - \pa{\dfrac{E}{D}\lambda- \dfrac{d}{2}},
    \end{align}
    and if \(\lambda < -\sqrt{\tau}\), then based on Equation \eqref{eq:3.2}, \(N_0\) satisfies
    \begin{align}
        \label{eq:4.10}
        -u_{N_0} < u_{N_0} < |\lambda|v_{N_0} - \pd{\dfrac{E}{D}\lambda- \dfrac{d}{2}} \leq -\lambda v_{N_0} + \pa{\dfrac{E}{D}\lambda- \dfrac{d}{2}}.
    \end{align}
    Both Equation \eqref{eq:4.9} and Equation \eqref{eq:4.10} satisfy \(u_{N_0} - \lambda v_{N_0} + \dfrac{E}{D}\lambda- \dfrac{d}{2} > 0\), and by Equation \eqref{eq:1.18}, this is equivalent to \(s_{N_0}^{(1)} > 0\). Now suppose that \(s^{(1)}(m)\) is strictly decreasing. By a similar argument as before, it must be the case \(\lambda > \sqrt{\tau}\), otherwise we reach a contradiction. Then based on Equation \eqref{eq:3.2}, \(N_0\) satisfies
    \begin{align*}
        u_{N_0} < |\lambda|v_{N_0} - \pd{\dfrac{E}{D}\lambda - \dfrac{d}{2}} \leq \lambda v_{N_0} - \pa{\dfrac{E}{D}\lambda- \dfrac{d}{2}},
    \end{align*}
    so it satisfies \(u_{N_0} - \lambda v_{N_0} + \dfrac{E}{D}\lambda - \dfrac{d}{2} < 0\), and by Equation \eqref{eq:1.18}, this is equivalent to \(s_{N_0}^{(1)} < 0\). \\

    Let \(l = 2\), and suppose that \(s^{(2)}(m)\) is strictly increasing. Then it must be the case \(\lambda < -\sqrt{\tau}\), otherwise if not, then \(\lambda > -\sqrt{\tau}\) and by Lemma \ref{lem:4.1}(d), function \(s^{(2)}(m)\) is strictly decreasing which is a contradiction. Then based on Equation \eqref{eq:3.2}, \(N_0\) satisfies
    \begin{align*}
        u_{N_0} < |\lambda|v_{N_0} - \pd{\dfrac{E}{D}\lambda- \dfrac{d}{2}} \leq -\lambda v_{N_0} + \pa{\dfrac{E}{D}\lambda- \dfrac{d}{2}},
    \end{align*}
    so it satisfies \(-u_{N_0} - \lambda v_{N_0} + \dfrac{E}{D}\lambda- \dfrac{d}{2} > 0\), and by Equation \eqref{eq:1.18}, this is equivalent to \(s_{N_0}^{(2)} > 0\). Now suppose that \(s^{(2)}(m)\) is strictly decreasing. By a similar argument as before, it must be the case \(\lambda > -\sqrt{\tau}\), otherwise we reach a contradiction. Hence, we have either \(|\lambda| < \sqrt{\tau}\) or \(\lambda > \sqrt{\tau}\). If \(|\lambda| < \sqrt{\tau}\), then based on Equation \eqref{eq:3.2}, \(N_0\) satisfies
    \begin{align}
        \label{eq:4.11}
        u_{N_0} > |\lambda|v_{N_0} + \pd{\dfrac{E}{D}\lambda- \dfrac{d}{2}} \geq -\lambda v_{N_0} + \pa{\dfrac{E}{D}\lambda- \dfrac{d}{2}},
    \end{align}
    and if \(\lambda > \sqrt{\tau}\), then based on Equation \eqref{eq:3.2}, \(N_0\) satisfies
    \begin{align}
        \label{eq:4.12}
        -u_{N_0} < u_{N_0} < |\lambda|v_{N_0} - \pd{\dfrac{E}{D}\lambda- \dfrac{d}{2}} \leq \lambda v_{N_0} - \pa{\dfrac{E}{D}\lambda- \dfrac{d}{2}}.
    \end{align}
    Both Equation \eqref{eq:4.11} and Equation \eqref{eq:4.12} satisfy \(-u_{N_0} - \lambda v_{N_0} + \dfrac{E}{D}\lambda- \dfrac{d}{2} < 0\), and by Equation \eqref{eq:1.18}, this is equivalent to \(s_{N_0}^{(2)} < 0\). \\

    Let \(l = 3\), and suppose that \(s^{(3)}(m)\) is strictly increasing. Then it must be the case \(\lambda > \sqrt{\tau}\), otherwise if not, then \(\lambda < \sqrt{\tau}\) and by Lemma \ref{lem:4.1}(a), function \(s^{(3)}(m)\) is strictly decreasing which is a contradiction. Then based on Equation \eqref{eq:3.2}, \(N_0\) satisfies
    \begin{align*}
        u_{N_0} < |\lambda|v_{N_0} - \pd{\dfrac{E}{D}\lambda- \dfrac{d}{2}} \leq \lambda v_{N_0} + \pa{\dfrac{E}{D}\lambda- \dfrac{d}{2}},
    \end{align*}
    so it satisfies \(-u_{N_0} + \lambda v_{N_0} + \dfrac{E}{D}\lambda- \dfrac{d}{2} > 0\), and by Equation \eqref{eq:1.18}, this is equivalent to \(s_{N_0}^{(3)} > 0\). Now suppose that \(s^{(3)}(m)\) is strictly decreasing. By a similar argument as before, it must be the case \(\lambda < \sqrt{\tau}\), otherwise we reach a contradiction. Hence, we have either \(|\lambda| < \sqrt{\tau}\) or \(\lambda < -\sqrt{\tau}\). If \(|\lambda| < \sqrt{\tau}\), then based on Equation \eqref{eq:3.2}, \(N_0\) satisfies
    \begin{align}
        \label{eq:4.13}
        u_{N_0} > |\lambda|v_{N_0} + \pd{\dfrac{E}{D}\lambda- \dfrac{d}{2}} \geq \lambda v_{N_0} + \pa{\dfrac{E}{D}\lambda- \dfrac{d}{2}},
    \end{align}
    and if \(\lambda < -\sqrt{\tau}\), then based on Equation \eqref{eq:3.2}, \(N_0\) satisfies
    \begin{align}
        \label{eq:4.14}
        -u_{N_0} < u_{N_0} < |\lambda|v_{N_0} - \pd{\dfrac{E}{D}\lambda- \dfrac{d}{2}} \leq -\lambda v_{N_0} - \pa{\dfrac{E}{D}\lambda- \dfrac{d}{2}}.
    \end{align}
    Both Equation \eqref{eq:4.13} and Equation \eqref{eq:4.14} satisfy \(-u_{N_0} +\lambda v_{N_0} + \dfrac{E}{D}\lambda- \dfrac{d}{2} < 0\), and by Equation \eqref{eq:1.18}, this is equivalent to \(s_{N_0}^{(3)} < 0\). \\

    Let \(l = 4\), and suppose that \(s^{(4)}(m)\) is strictly increasing. Then it must be the case \(\lambda > -\sqrt{\tau}\), otherwise if not, then \(\lambda < -\sqrt{\tau}\) and by Lemma \ref{lem:4.1}(c), function \(s^{(4)}(m)\) is strictly decreasing which is a contradiction. Hence, we have either \(|\lambda| < \sqrt{\tau}\) or \(\lambda > \sqrt{\tau}\). If \(|\lambda| < \sqrt{\tau}\), then based on Equation \eqref{eq:3.2}, \(N_0\) satisfies
    \begin{align}
        \label{eq:4.15}
        u_{N_0} > |\lambda|v_{N_0} + \pd{\dfrac{E}{D}\lambda- \dfrac{d}{2}} \geq -\lambda v_{N_0} - \pa{\dfrac{E}{D}\lambda- \dfrac{d}{2}},
    \end{align}
    and if \(\lambda > \sqrt{\tau}\), then based on Equation \eqref{eq:3.2}, \(N_0\) satisfies
    \begin{align}
        \label{eq:4.16}
        -u_{N_0} < u_{N_0} < |\lambda|v_{N_0} - \pd{\dfrac{E}{D}\lambda- \dfrac{d}{2}} \leq \lambda v_{N_0} + \pa{\dfrac{E}{D}\lambda- \dfrac{d}{2}}.
    \end{align}
    Both Equation \eqref{eq:4.15} and Equation \eqref{eq:4.16} satisfy \(u_{N_0} + \lambda v_{N_0} + \dfrac{E}{D}\lambda- \dfrac{d}{2} > 0\), and by Equation \eqref{eq:1.18}, this is equivalent to \(s_{N_0}^{(4)} > 0\). Now suppose that \(s^{(4)}(m)\) is strictly decreasing. By a similar argument as before, it must be the case \(\lambda < -\sqrt{\tau}\), otherwise we reach a contradiction. Then based on Equation \eqref{eq:3.2}, \(N_0\) satisfies
    \begin{align*}
        u_{N_0} < |\lambda|v_{N_0} - \pd{\dfrac{E}{D}\lambda- \dfrac{d}{2}} \leq -\lambda v_{N_0} - \pa{\dfrac{E}{D}\lambda- \dfrac{d}{2}},
    \end{align*}
    so it satisfies \(u_{N_0} + \lambda v_{N_0} + \dfrac{E}{D}\lambda- \dfrac{d}{2} < 0\), and by Equation \eqref{eq:1.18}, this is equivalent to \(s_{N_0}^{(4)} < 0\). \\

    \textit{Proof of} (b)\textit{.} Similar to the proof in (a), it is sufficient to show \(t^{(l)}(N_0) = t^{(l)}_{N_0} > 0\) if \(t^{(l)}(m)\) is strictly increasing, and it is sufficient to show \(t^{(l)}(N_0) = t^{(l)}_{N_0} < 0\) if \(t^{(l)}(m)\) is strictly decreasing. \\

    Suppose that \(t^{(l)}(m)\) is strictly increasing. Then it must be \(l \in \{1,2\}\), otherwise if not, then \(l \in \{3,4\}\) and by Lemma \ref{lem:4.1}(e), \(t^{(l)}(m)\) is strictly decreasing which is a contradiction. Based on Equation \eqref{eq:3.2}, \(N_0\) satisfies \(v_{N_0} > \dfrac{\pd{E}}{D} \geq \dfrac{E}{D}\) which is equivalent to \(v_{N_0} - \dfrac{E}{D} > 0\). By Equation \eqref{eq:1.19}, this is just equivalent to \(t_{N_0}^{(l)} > 0\) for \(l \in \{1,2\}\). \\

    Suppose that \(t^{(l)}(m)\) is strictly decreasing. By a similar argument as before, it must be \(l \in \{3,4\}\), otherwise we reach a contradiction. Based on Equation \eqref{eq:3.2}, \(N_0\) satisfies \(v_{N_0} > \dfrac{\pd{E}}{D} \geq - \dfrac{E}{D}\) which is equivalent to \(-v_{N_0} - \dfrac{E}{D} < 0\). By Equation \eqref{eq:1.19}, this is just equivalent to \(t_{N_0}^{(l)} < 0\) for \(l \in \{3,4\}\).
\end{proof}

\vspace{5pt}
\begin{lemma}
    \label{lem:4.3}
    Let \(l \in \{1,2,3,4\}\) and \(N_0 \in \N\) as defined in Equation \eqref{eq:3.2}. The following are all true:
    \begin{enumerate}[label=\normalfont(\alph*)]
        \item All functions \(\pd{s^{(l)}(m)}\), \(\pd{t^{(l)}(m)}\) and \(\pd{s^{(l)}(m)} + \pd{t^{(l)}(m)}\) are continuous on domain \(m \in [1, +\infty)\).
        \item All functions \(\pd{s^{(l)}(m)}\), \(\pd{t^{(l)}(m)}\) and \(\pd{s^{(l)}(m)} + \pd{t^{(l)}(m)}\) are strictly increasing for all real numbers \(m \geq N_0\).
        \item \(\displaystyle\lim_{m \to +\infty} \pd{s^{(l)}(m)} = \displaystyle\lim_{m \to +\infty} \pd{t^{(l)}(m)} = \displaystyle\lim_{m \to +\infty} \pa{\pd{s^{(l)}(m)} + \pd{t^{(l)}(m)}} = +\infty\).
    \end{enumerate}
\end{lemma}

\begin{proof}
    \textit{Proof of} (a)\textit{.} This is obvious, since both \(s^{(l)}(m)\) and \(s^{(l)}(m)\) as defined in Equation \eqref{eq:4.1} and Equation \eqref{eq:4.2} respectively are continuous functions on domain \(m \in [1, +\infty)\), and absolute value function is continuous on \(\R\). \\
    
    \textit{Proof of} (b)\textit{.} From Lemma \ref{lem:4.1}(a), (b), (c) and (d), the function \(s^{(l)}(m)\) is either strictly increasing or strictly decreasing. Suppose that \(s^{(l)}(m)\) is strictly increasing, then \(s^{(l)}(m_1) < s^{(l)}(m_2)\) for all \(1 \leq m_1 < m_2 < + \infty\). By Lemma \ref{lem:4.2}(a), we have \(s^{(l)}(m) > 0\) for all \(m \geq N_0\). So for any \(N_0 \leq m_1 < m_2 < + \infty\), \(\pd{s^{(l)}(m_1)} = s^{(l)}(m_1) < s^{(l)}(m_2) = \pd{s^{(l)}(m_2)}\), which implies function \(\pd{s^{(l)}(m)}\) is strictly increasing for all \(m \geq N_0\). Now suppose that \(s^{(l)}(m)\) is strictly decreasing, then \(s^{(l)}(m_1) > s^{(l)}(m_2)\) for all \(1 \leq m_1 < m_2 < + \infty\). By Lemma \ref{lem:4.2}(a), we have \(s^{(l)}(m) < 0\) for all \(m \geq N_0\). So for any \(N_0 \leq m_1 < m_2 < + \infty\), \(\pd{s^{(l)}(m_1)} = -s^{(l)}(m_1) < -s^{(l)}(m_2) = \pd{s^{(l)}(m_2)}\), which implies function \(\pd{s^{(l)}(m)}\) is strictly increasing for all \(m \geq N_0\). \\

    From Lemma \ref{lem:4.1}(e), the function \(t^{(l)}(m)\) is either strictly increasing or strictly decreasing. Suppose that \(t^{(l)}(m)\) is strictly increasing, then \(t^{(l)}(m_1) < t^{(l)}(m_2)\) for all \(1 \leq m_1 < m_2 < + \infty\). By Lemma \ref{lem:4.2}(b), we have \(t^{(l)}(m) > 0\) for all \(m \geq N_0\). So for any \(N_0 \leq m_1 < m_2 < + \infty\), \(\pd{t^{(l)}(m_1)} = t^{(l)}(m_1) < t^{(l)}(m_2) = \pd{t^{(l)}(m_2)}\), which implies function \(\pd{t^{(l)}(m)}\) is strictly increasing for all \(m \geq N_0\). Now suppose that \(t^{(l)}(m)\) is strictly decreasing, then \(t^{(l)}(m_1) > t^{(l)}(m_2)\) for all \(1 \leq m_1 < m_2 < + \infty\). By Lemma \ref{lem:4.2}(a), we have \(t^{(l)}(m) < 0\) for all \(m \geq N_0\). So for any \(N_0 \leq m_1 < m_2 < + \infty\), \(\pd{t^{(l)}(m_1)} = -t^{(l)}(m_1) < -t^{(l)}(m_2) = \pd{t^{(l)}(m_2)}\), which implies function \(\pd{t^{(l)}(m)}\) is strictly increasing for all \(m \geq N_0\). \\

    We just showed both functions \(\pd{s^{(l)}(m)}\) and \(\pd{t^{(l)}(m)}\) are strictly increasing for all real numbers \(m \geq N_0\). Since \(\pd{s^{(l)}(m)} \geq 0\) and \(\pd{t^{(l)}(m)} \geq 0\), this follows that their sum, \(\pd{s^{(l)}(m)} + \pd{t^{(l)}(m)}\) must also strictly increasing for all real numbers \(m \geq N_0\). \\

    \textit{Proof of} (c)\textit{.} By Lemma \ref{lem:4.3}(b), both functions \(\pd{s^{(l)}(m)}\) and \(\pd{t^{(l)}(m)}\) are strictly increasing for all \(m \geq N_0\). This implies both sequences \(\pa{\pd{s_m^{(l)}}}\) and \(\pa{\pd{t_m^{(l)}}}\) are strictly increasing. Since \(\pd{s_m^{(l)}}\) and \(\pd{t_m^{(l)}}\) are both integers for all \(m \in \N\), we must have \(\pa{\pd{s_m^{(l)}}} \to +\infty\) and \(\pa{\pd{t_m^{(l)}}} \to +\infty\). This concludes that \(\displaystyle\lim_{m \to +\infty}\pd{s^{(l)}(m)} = +\infty\) and \(\displaystyle\lim_{m \to +\infty}\pd{t^{(l)}(m)} = +\infty\). As a consequence, we obtain
    \[\displaystyle\lim_{m \to +\infty}\pa{\pd{s^{(l)}(m)} + \pd{t^{(l)}(m)}} = +\infty.\]
\end{proof}

\vspace{5pt}
\section{Behavior of \(\pd{s^{(l)}(m)} + \pd{t^{(l)}(m)}\) with Respect to \(M'_l\) where \(l \in \{1,2,3,4\}\)}

We begin by considering another form of positive integer \(M'_l\):

\vspace{5pt}
\begin{lemma}
    \label{lem:5.1}
    Let \(l \in \{1,2,3,4\}\) and \(M'_l\) as defined in Equation \eqref{eq:3.1}. Then
    \begin{align}
        \label{eq:5.1}
        M'_l = \max\pc{\pd{s_{1}^{(l)}}, \pd{s_{\max\pc{N_0, N_l}}^{(l)}}} + \max\pc{\pd{t_{1}^{(l)}}, \pd{t_{\max\pc{N_0, N_l}}^{(l)}}},
    \end{align}
    where \(s_{m}^{(l)}\), \(t_{m}^{(l)}\), \(N_0\) and \(N_l\) are defined as Equation \eqref{eq:1.18}, Equation \eqref{eq:1.19}, Equation \eqref{eq:3.2} and Equation \eqref{eq:3.3} respectively.
\end{lemma}

\begin{proof}
    Let \(l \in \{1,2,3,4\}\). We separate the proof into two cases, \(N_0 \geq N_l\) and \(N_0 < N_l\). \\

    \textit{Proof of the case \(N_0 \geq N_l\).} We have \(\max\pc{N_0, N_l} = N_0\), and by comparing Equation \eqref{eq:3.1} and Equation \eqref{eq:5.1}, we require to show
    \begin{align}
        \notag
        &\max\pc{\pd{s_{1}^{(l)}}, \pd{s_{N_0}^{(l)}}, \pd{s_{N_l}^{(l)}}} + \max\pc{\pd{t_{1}^{(l)}}, \pd{t_{N_0}^{(l)}}, \pd{t_{N_l}^{(l)}}} \\
        \label{eq:5.2}
        &\hspace{100pt} = \max\pc{\pd{s_{1}^{(l)}}, \pd{s_{N_0}^{(l)}}} + \max\pc{\pd{t_{1}^{(l)}}, \pd{t_{N_0}^{(l)}}}.
    \end{align}
    The idea is to show
    \begin{align}
        \label{eq:5.3}
        \pd{s^{(l)}(m)} &\leq \max\pc{\pd{s^{(l)}(1)}, \pd{s^{(l)}(N_0)}}, \ \text{and} \\
        \label{eq:5.4}
        \pd{t^{(l)}(m)} &\leq \max\pc{\pd{t^{(l)}(1)}, \pd{t^{(l)}(N_0)}},
    \end{align}
    for all \(m \in [1, N_0]\), and since \(N_l \in [1,N_0]\), hence
    \begin{align*}
        \max\pc{\pd{s_{1}^{(l)}}, \pd{s_{N_0}^{(l)}}} &= \max\pc{\pd{s_{1}^{(l)}}, \pd{s_{N_0}^{(l)}}, \pd{s_{N_l}^{(l)}}}, \ \text{and} \\
        \max\pc{\pd{t_{1}^{(l)}}, \pd{t_{N_0}^{(l)}}} &= \max\pc{\pd{t_{1}^{(l)}}, \pd{t_{N_0}^{(l)}}, \pd{t_{N_l}^{(l)}}},
    \end{align*}
    and therefore Equation \eqref{eq:5.2} follows. \\

    \textit{Proof of \(\pd{s^{(l)}(m)} \leq \max\pc{\pd{s^{(l)}(1)}, \pd{s^{(l)}(N_0)}}\) for all \(m \in [1, N_0]\).} From Lemma \ref{lem:4.1}(a), (b), (c) and (d), the function \(s^{(l)}(m)\) is either strictly increasing or strictly decreasing for \(m \in [1,+\infty)\). Suppose that \(s^{(l)}(m)\) is strictly increasing. Consider the value \(s^{(l)}(1)\). If \(s^{(l)}(1) \geq 0\), then
    \[0 \leq s^{(l)}(1) \leq s^{(l)}(m) \leq s^{(l)}(N_0) \implies 0 \leq \pd{s^{(l)}(1)} \leq \pd{s^{(l)}(m)} \leq \pd{s^{(l)}(N_0)},\]
    for any \(m \in [1,N_0]\). Therefore for any \(m \in [1, N_0]\), it satisfies Equation \eqref{eq:5.3}. Suppose that \(s^{(l)}(1) < 0\). From Lemma \ref{lem:4.2}(a), \(s^{(l)}(N_0) > 0\). Since function \(s^{(l)}(m)\) is continuous in the interval \(m \in [1, N_0]\), then by Intermediate Value Theorem, there exists \(c \in (1, N_0)\) such that \(s^{(l)}(c) = 0\). Since \(s^{(l)}(m)\) is strictly increasing, we have
    \[s^{(l)}(1) \leq s^{(l)}(m) < s^{(l)}(c) = 0 \implies 0 = \pd{s^{(l)}(c)} < \pd{s^{(l)}(m)} \leq \pd{s^{(l)}(1)},\]
    for all \(m \in [1,c)\), and
    \[0 = s^{(l)}(c) < s^{(l)}(m) \leq s^{(l)}(N_0) \implies 0 = \pd{s^{(l)}(c)} < \pd{s^{(l)}(m)} \leq \pd{s^{(l)}(N_0)},\]
    for all \(m \in (c,N_0]\). Therefore for any \(m \in [1, N_0]\), it satisfies Equation \eqref{eq:5.3}. \\

    Suppose that \(s^{(l)}(m)\) is strictly decreasing. Consider the value \(s^{(l)}(1)\). If \(s^{(l)}(1) \leq 0\), then
    \[0 \geq s^{(l)}(1) \geq s^{(l)}(m) \geq s^{(l)}(N_0) \implies 0 \leq \pd{s^{(l)}(1)} \leq \pd{s^{(l)}(m)} \leq \pd{s^{(l)}(N_0)},\]
    for any \(m \in [1,N_0]\). Therefore for any \(m \in [1, N_0]\), it satisfies Equation \eqref{eq:5.3}. Suppose that \(s^{(l)}(1) > 0\). From Lemma \ref{lem:4.2}(a), \(s^{(l)}(N_0) < 0\). Since function \(s^{(l)}(m)\) is continuous in the interval \(m \in [1, N_0]\), then by Intermediate Value Theorem, there exists \(c \in (1, N_0)\) such that \(s^{(l)}(c) = 0\). Since \(s^{(l)}(m)\) is strictly decreasing, we have
    \[s^{(l)}(1) \geq s^{(l)}(m) > s^{(l)}(c) = 0 \implies 0 = \pd{s^{(l)}(c)} < \pd{s^{(l)}(m)} \leq \pd{s^{(l)}(1)},\]
    for all \(m \in [1,c)\), and
    \[0 = s^{(l)}(c) > s^{(l)}(m) \geq s^{(l)}(N_0) \implies 0 = \pd{s^{(l)}(c)} < \pd{s^{(l)}(m)} \leq \pd{s^{(l)}(N_0)},\]
    for all \(m \in (c,N_0]\). Therefore for any \(m \in [1, N_0]\), it satisfies Equation \eqref{eq:5.3}. \\

    \textit{Proof of \(\pd{t^{(l)}(m)} \leq \max\pc{\pd{t^{(l)}(1)}, \pd{t^{(l)}(N_0)}}\) for all \(m \in [1, N_0]\).} From Lemma \ref{lem:4.1}(e), the function \(t^{(l)}(m)\) is either strictly increasing or strictly decreasing for \(m \in [1,+\infty)\). Suppose that \(t^{(l)}(m)\) is strictly increasing. Consider the value \(t^{(l)}(1)\). If \(t^{(l)}(1) \geq 0\), then
    \[0 \leq t^{(l)}(1) \leq t^{(l)}(m) \leq t^{(l)}(N_0) \implies 0 \leq \pd{t^{(l)}(1)} \leq \pd{t^{(l)}(m)} \leq \pd{t^{(l)}(N_0)},\]
    for any \(m \in [1,N_0]\). Therefore for any \(m \in [1, N_0]\), it satisfies Equation \eqref{eq:5.4}. Suppose that \(t^{(l)}(1) < 0\). From Lemma \ref{lem:4.2}(b), \(t^{(l)}(N_0) > 0\). Since function \(t^{(l)}(m)\) is continuous in the interval \(m \in [1, N_0]\), then by Intermediate Value Theorem, there exists \(c \in (1, N_0)\) such that \(t^{(l)}(c) = 0\). Since \(t^{(l)}(m)\) is strictly increasing, we have
    \[t^{(l)}(1) \leq t^{(l)}(m) < t^{(l)}(c) = 0 \implies 0 = \pd{t^{(l)}(c)} < \pd{t^{(l)}(m)} \leq \pd{t^{(l)}(1)},\]
    for all \(m \in [1,c)\), and
    \[0 = t^{(l)}(c) < t^{(l)}(m) \leq t^{(l)}(N_0) \implies 0 = \pd{t^{(l)}(c)} < \pd{t^{(l)}(m)} \leq \pd{t^{(l)}(N_0)},\]
    for all \(m \in (c,N_0]\). Therefore for any \(m \in [1, N_0]\), it satisfies Equation \eqref{eq:5.4}. \\

    Suppose that \(t^{(l)}(m)\) is strictly decreasing. Consider the value \(t^{(l)}(1)\). If \(t^{(l)}(1) \leq 0\), then
    \[0 \geq t^{(l)}(1) \geq t^{(l)}(m) \geq t^{(l)}(N_0) \implies 0 \leq \pd{t^{(l)}(1)} \leq \pd{t^{(l)}(m)} \leq \pd{t^{(l)}(N_0)},\]
    for any \(m \in [1,N_0]\). Therefore for any \(m \in [1, N_0]\), it satisfies Equation \eqref{eq:5.4}. Suppose that \(t^{(l)}(1) > 0\). From Lemma \ref{lem:4.2}(b), \(t^{(l)}(N_0) < 0\). Since function \(t^{(l)}(m)\) is continuous in the interval \(m \in [1, N_0]\), then by Intermediate Value Theorem, there exists \(c \in (1, N_0)\) such that \(t^{(l)}(c) = 0\). Since \(t^{(l)}(m)\) is strictly decreasing, we have
    \[t^{(l)}(1) \geq t^{(l)}(m) > t^{(l)}(c) = 0 \implies 0 = \pd{t^{(l)}(c)} < \pd{t^{(l)}(m)} \leq \pd{t^{(l)}(1)},\]
    for all \(m \in [1,c)\), and
    \[0 = t^{(l)}(c) > t^{(l)}(m) \geq t^{(l)}(N_0) \implies 0 = \pd{t^{(l)}(c)} < \pd{t^{(l)}(m)} \leq \pd{t^{(l)}(N_0)},\]
    for all \(m \in (c,N_0]\). Therefore for any \(m \in [1, N_0]\), it satisfies Equation \eqref{eq:5.4}. \\

    \textit{Proof of the case \(N_0 < N_l\).} We have \(\max\pc{N_0, N_l} = N_l\), and by comparing Equation \eqref{eq:3.1} and Equation \eqref{eq:5.1}, we require to show
    \begin{align}
        \notag
        &\max\pc{\pd{s_{1}^{(l)}}, \pd{s_{N_0}^{(l)}}, \pd{s_{N_l}^{(l)}}} + \max\pc{\pd{t_{1}^{(l)}}, \pd{t_{N_0}^{(l)}}, \pd{t_{N_l}^{(l)}}} \\
        \label{eq:5.5}
        &\hspace{100pt} = \max\pc{\pd{s_{1}^{(l)}}, \pd{s_{N_l}^{(l)}}} + \max\pc{\pd{t_{1}^{(l)}}, \pd{t_{N_l}^{(l)}}}.
    \end{align}
    From Lemma \ref{lem:4.3}(b), \(\pd{s^{(l)}(N_l)} > \pd{s^{(l)}(N_0)}\) and \(\pd{t^{(l)}(N_l)} > \pd{t^{(l)}(N_0)}\), hence
    \begin{align*}
        \max\pc{\pd{s_{1}^{(l)}}, \pd{s_{N_l}^{(l)}}} &= \max\pc{\pd{s_{1}^{(l)}}, \pd{s_{N_0}^{(l)}}, \pd{s_{N_l}^{(l)}}}, \ \text{and} \\
        \max\pc{\pd{t_{1}^{(l)}}, \pd{t_{N_l}^{(l)}}} &= \max\pc{\pd{t_{1}^{(l)}}, \pd{t_{N_0}^{(l)}}, \pd{t_{N_l}^{(l)}}},
    \end{align*}
    and therefore Equation \eqref{eq:5.5} follows.
\end{proof}

\vspace{5pt}
\begin{lemma}
    \label{lem:5.2}
    Let \(l \in \{1,2,3,4\}\) and \(M'_l\) as defined in Equation \eqref{eq:5.1}. For any \(n \in [M'_l, +\infty)\), there exists \(m \in [1,+\infty)\) such that \(\pd{s^{(l)}(m)} + \pd{t^{(l)}(m)} = n\).
\end{lemma}

\begin{proof}
    By Lemma \ref{lem:4.3}(a), the function \(\pd{s^{(l)}(m)} + \pd{t^{(l)}(m)}\) is continuous on domain \(m \in [1,+\infty)\). Therefore, the function \(\pd{s^{(l)}(m)} + \pd{t^{(l)}(m)}\) has the intermediate value property on interval \([1,+\infty)\). By Lemma \ref{lem:4.3}(b) and Lemma \ref{lem:4.3}(c), the function \(\pd{s^{(l)}(m)} + \pd{t^{(l)}(m)}\) is strictly increasing for all \(m \geq N_0\) and
    \[\displaystyle\lim_{m \to +\infty} \pa{\pd{s^{(l)}(m)} + \pd{t^{(l)}(m)}} = +\infty.\]
    Together with the intermediate value property, this implies that for any \(n \in \R\) in the interval \(\left[\pd{s^{(l)}(N_0)} + \pd{t^{(l)}(N_0)}, +\infty\right)\), there exists \(m \in [N_0,+\infty)\) such that \(\pd{s^{(l)}(m)} + \pd{t^{(l)}(m)} = n\). Notice that
    \begin{align*}
        &\pd{s^{(l)}(N_0)} + \pd{t^{(l)}(N_0)} \\
        &= \pd{s_{N_0}^{(l)}} + \pd{t_{N_0}^{(l)}}, \\
        &\leq \max\pc{\pd{s_{1}^{(l)}}, \pd{s_{\max\pc{N_0, N_l}}^{(l)}}} + \max\pc{\pd{t_{1}^{(l)}}, \pd{t_{\max\pc{N_0, N_l}}^{(l)}}},
    \end{align*}
    which implies
    \begin{align}
        \label{eq:5.6}
        \pd{s^{(l)}(N_0)} + \pd{t^{(l)}(N_0)} &\leq M'_l.
    \end{align}
    This implies
    \[[M'_l, +\infty) \subseteq \left[\pd{s^{(l)}(N_0)} + \pd{t^{(l)}(N_0)}, +\infty\right),\]
    hence for any \(n \in [M'_l, +\infty)\), there exists \(m \in [N_0,+\infty)\) such that \(\pd{s^{(l)}(m)} + \pd{t^{(l)}(m)} = n\). The lemma is then follows.
\end{proof}

\vspace{5pt}
\begin{lemma}
    \label{lem:5.3}
    Let \(l \in \{1,2,3,4\}\), and \(M'_l\), \(N_0\) and \(N_l\) are defined as Equation \eqref{eq:5.1}, Equation \eqref{eq:3.2} and Equation \eqref{eq:3.3} respectively. For any real number \(x > M'_l\), let \(m \in [1,+\infty)\) such that \(\pd{s^{(l)}(m)} + \pd{t^{(l)}(m)} = x\). Then \(m > \max\{N_0, N_l\}\).
\end{lemma}

\begin{proof}
    Let \(l \in \{1,2,3,4\}\). We separate the proof into two cases, \(N_0 \geq N_l\) and \(N_0 < N_l\). \\

    \textit{Proof of the case \(N_0 \geq N_l\).} We have \(\max\pc{N_0, N_l} = N_0\) and from Equation \eqref{eq:5.1},
    \begin{align}
        \label{eq:5.7}
        M'_l = \max\pc{\pd{s_{1}^{(l)}}, \pd{s_{N_0}^{(l)}}} + \max\pc{\pd{t_{1}^{(l)}}, \pd{t_{N_0}^{(l)}}}.
    \end{align}
    Hence, we require to show if \(\pd{s^{(l)}(m)} + \pd{t^{(l)}(m)} > M'_l\), where \(M'_l\) is as Equation \eqref{eq:5.7}, then \(m > N_0\). From the proof in Lemma \ref{lem:5.1}, we already showed
    \begin{align*}
        \pd{s^{(l)}(m)} &\leq \max\pc{\pd{s^{(l)}(1)}, \pd{s^{(l)}(N_0)}}, \ \text{and} \\
        \pd{t^{(l)}(m)} &\leq \max\pc{\pd{t^{(l)}(1)}, \pd{t^{(l)}(N_0)}},
    \end{align*}
    for all \(m \in [1, N_0]\) (Equation \eqref{eq:5.3} and Equation \eqref{eq:5.4}), so we obtain \(\pd{s^{(l)}(m)} + \pd{t^{(l)}(m)} \leq M'_l\) for all \(m \in [1, N_0]\). This implies if there exists \(m \in [1,+\infty)\) such that \(\pd{s^{(l)}(m)} + \pd{t^{(l)}(m)} > M'_l\), then it must be \(m > N_0\). \\
    
    \textit{Proof of the case \(N_0 < N_l\).} We have \(\max\pc{N_0, N_l} = N_l\) and from Equation \eqref{eq:5.1},
    \begin{align}
        \label{eq:5.8}
        M'_l = \max\pc{\pd{s_{1}^{(l)}}, \pd{s_{N_l}^{(l)}}} + \max\pc{\pd{t_{1}^{(l)}}, \pd{t_{N_l}^{(l)}}}.
    \end{align}
    Hence, we require to show if \(\pd{s^{(l)}(m)} + \pd{t^{(l)}(m)} > M'_l\), where \(M'_l\) is as Equation \eqref{eq:5.8}, then \(m > N_l\). From Lemma \ref{lem:4.3}(b), \(\pd{s^{(l)}(N_l)} > \pd{s^{(l)}(N_0)}\) and \(\pd{t^{(l)}(N_l)} > \pd{t^{(l)}(N_0)}\). Then from Equation \eqref{eq:5.8},
    \begin{align*}
        M'_l &= \max\pc{\pd{s^{(l)}(1)}, \pd{s^{(l)}(N_l)}} + \max\pc{\pd{t^{(l)}(1)}, \pd{t^{(l)}(N_l)}}, \\
        &\geq \max\pc{\pd{s^{(l)}(1)}, \pd{s^{(l)}(N_0)}} + \max\pc{\pd{t^{(l)}(1)}, \pd{t^{(l)}(N_0)}},
    \end{align*}
    where the last line is the value \(M'_l\) in the form of Equation \eqref{eq:5.7}. Hence, we can conclude \(m > N_0\) based on the previous proof. Therefore, it is sufficient to show \(m \notin (N_0, N_l]\). Suppose that \(m \in (N_0, N_l]\). By Lemma \ref{lem:4.3}(b), \(\pd{s^{(l)}(m)} \leq \pd{s^{(l)}(N_l)}\) and \(\pd{t^{(l)}(m)} \leq \pd{t^{(l)}(N_l)}\). This implies that
    \begin{align*}
        \pd{s^{(l)}(m)} + \pd{t^{(l)}(m)} &\leq \pd{s^{(l)}(N_l)} + \pd{t^{(l)}(N_l)}, \\
        &\leq \max\pc{\pd{s^{(l)}(1)}, \pd{s^{(l)}(N_l)}} + \max\pc{\pd{t^{(l)}(1)}, \pd{t^{(l)}(N_l)}},
    \end{align*}
    where the last line is the value \(M'_l\) in the form Equation \eqref{eq:5.8}. But we get \(\pd{s^{(l)}(m)} + \pd{t^{(l)}(m)} \leq M'_l\), where \(M'_l\) is as Equation \eqref{eq:5.8}, which contradicts with the assumption \(\pd{s^{(l)}(m)} + \pd{t^{(l)}(m)} > M'_l\). So, \(m \notin (N_0, N_l]\), as desired.
\end{proof}

\vspace{5pt}
\begin{lemma}
    \label{lem:5.4}
    Let \(l \in \{1,2,3,4\}\) and \(M'_l\) as defined in Equation \eqref{eq:5.1}. For any \(n \in (M'_l, +\infty)\), there exists unique \(m \in [1,+\infty)\) such that \(\pd{s^{(l)}(m)} + \pd{t^{(l)}(m)} = n\).
\end{lemma}

\begin{proof}
    Let \(n \in (M'_l, +\infty)\), and suppose that there exist \(m_1, m_2 \in [1,+\infty)\) with \(m_1 > m_2\) such that \(\pd{s^{(l)}(m_1)} + \pd{t^{(l)}(m_1)} = \pd{s^{(l)}(m_2)} + \pd{t^{(l)}(m_2)} = n\). Since \(n > M'_l\), then by Lemma \ref{lem:5.3}, \(m_1 > m_2 > \max\{N_0, N_l\}\). But \(m_1 > m_2 > N_0\), then by Lemma \ref{lem:4.3}(b), \(n = \pd{s^{(l)}(m_1)} + \pd{t^{(l)}(m_1)} > \pd{s^{(l)}(m_2)} + \pd{t^{(l)}(m_2)} = n\), which is a contradiction. Hence, \(m\) is unique.
\end{proof}

\vspace{5pt}
\subsection{Existence of \(N'_l\) where \(l \in \{1,2,3,4\}\)} Let \(l \in \{1,2,3,4\}\) and \(M'_l\) as defined in Equation \eqref{eq:5.1}. Consider the set
\begin{align}
    \label{eq:5.9}
    \Bigl\{m \in [1,+\infty): \pd{s^{(l)}(m)} + \pd{t^{(l)}(m)} \leq M'_l\Bigr\}.
\end{align}
From Lemma \ref{lem:5.2}, there exists \(m \in [1,+\infty)\) such that \(\pd{s^{(l)}(m)} + \pd{t^{(l)}(m)} = M'_l\), therefore the set \eqref{eq:5.9} is non empty. Pick any \(k \in (M'_l, +\infty)\). By Lemma \ref{lem:5.4}, there exists unique \(p \in [1,+\infty)\) such that \(\pd{s^{(l)}(p)} + \pd{t^{(l)}(p)} = k\). From Lemma \ref{eq:5.7}, \(p > \max\{N_0, N_l\} \geq N_0\). By Lemma \ref{eq:4.3}(b), \(\pd{s^{(l)}(m)} + \pd{t^{(l)}(m)}\) is strictly increasing for any \(m \geq N_0\), hence
\[ \pd{s^{(l)}(m)} + \pd{t^{(l)}(m)} \geq \pd{s^{(l)}(p)} + \pd{t^{(l)}(p)} = k > M'_l,\]
for all \(m \in [p, +\infty)\). Therefore the set \eqref{eq:5.9} must be equivalent to the set
\begin{align*}
    \Bigl\{m \in [1,p): \pd{s^{(l)}(m)} + \pd{t^{(l)}(m)} \leq M'_l\Bigr\},
\end{align*}
which is bounded above by \(p\). By the Axiom of completeness, the least upper bound of \eqref{eq:5.9} exists. Let \(N'_l\) to be its supremum,
\begin{align}
    \label{eq:5.10}
    N'_l := \sup \Bigl\{m \in [1,+\infty): \pd{s^{(l)}(m)} + \pd{t^{(l)}(m)} \leq M'_l\Bigr\}.
\end{align}

\vspace{5pt}
\begin{lemma}
    \label{lem:5.5}
    Let \(l \in \{1,2,3,4\}\), and \(M'_l\) and \(N'_l\) are defined as Equation \eqref{eq:5.1} and Equation \eqref{eq:5.10} respectively. The following are all true:
    \begin{enumerate}[label=\normalfont(\alph*)]
        \item For any \(m \in \pa{N'_l, +\infty}\), \(\pd{s^{(l)}(m)} + \pd{t^{(l)}(m)} > M'_l\).
        \item \(\pd{s^{(l)}(N'_l)} + \pd{t^{(l)}(N'_l)} = M'_l\). In other words, \(N'_l\) belongs to the set \eqref{eq:5.9} and therefore \(N'_l\) can be defined as the maximum of \eqref{eq:5.9}.
        \item For any \(m \in \pb{1, N'_l}\), \(\pd{s^{(l)}(m)} + \pd{t^{(l)}(m)} \leq M'_l\).
    \end{enumerate}
\end{lemma}

\begin{proof}
    \textit{Proof of} (a)\textit{.} Suppose that there exists \(m > N'_l\) such that \(\pd{s^{(l)}(m)} + \pd{t^{(l)}(m)} \leq M'_l\), then \(m\) belongs to the set \eqref{eq:5.9}. However since \(N'_l\) is the upper bound of the set \eqref{eq:5.9}, we get \(m \leq N'_l\), which is a contradiction. Therefore \(\pd{s^{(l)}(m)} + \pd{t^{(l)}(m)} > M'_l\) for any \(m \in \pa{N'_l, +\infty}\). \\
    
    \textit{Proof of} (b)\textit{.} Our idea is to rule out the two cases \(\pd{s^{(l)}(N'_l)} + \pd{t^{(l)}(N'_l)} < M'_l\) and \(\pd{s^{(l)}(N'_l)} + \pd{t^{(l)}(N'_l)} > M'_l\), hence it must be \(\pd{s^{(l)}(N'_l)} + \pd{t^{(l)}(N'_l)} = M'_l\). First, suppose that \(\pd{s^{(l)}(N'_l)} + \pd{t^{(l)}(N'_l)} < M'_l\). Then \(\pd{s^{(l)}(N'_l)} + \pd{t^{(l)}(N'_l)} = M'_l - \varepsilon\) for some \(\varepsilon > 0\). We will show there exists \(m > N'_l\) such that \(\pd{s^{(l)}(m)} + \pd{t^{(l)}(m)} = M'_l - \dfrac{\varepsilon}{2} < M'_l\), which contradicts the assumption that \(N'_l\) is the upper bound of set \eqref{eq:5.9}. From Equation \eqref{eq:5.6}, \(\pd{s^{(l)}(N_0)} + \pd{t^{(l)}(N_0)} \leq M'_l\), so \(N_0\) belongs to the set \eqref{eq:5.9}. Since \(N'_l\) is the upper bound of the set \eqref{eq:5.9}, then \(N'_l \geq N_0\). By Lemma \ref{eq:4.3}(a), (b) and (c), the function \(\pd{s^{(l)}(m)} + \pd{t^{(l)}(m)}\) is continuous, strictly increasing for any \(m \geq N_0\), and \(\displaystyle\lim_{m \to +\infty} \pa{\pd{s^{(l)}(m)} + \pd{t^{(l)}(m)}} = +\infty\). Therefore, for any \(n > \pd{s^{(l)}(N'_l)} + \pd{t^{(l)}(N'_l)} = M'_l - \varepsilon\), there exists \(m > N'_l\) such that \(\pd{s^{(l)}(m)} + \pd{t^{(l)}(m)} = n\). We are done once we pick \(n = M'_l - \dfrac{\varepsilon}{2} > M'_l - \varepsilon\). \\

    Next, suppose that \(\pd{s^{(l)}(N'_l)} + \pd{t^{(l)}(N'_l)} > M'_l\). Then \(\pd{s^{(l)}(N'_l)} + \pd{t^{(l)}(N'_l)} = M'_l + \varepsilon\) for some \(\varepsilon > 0\). By Lemma \ref{lem:5.4}, there exists unique \(m \in [1,+\infty)\) such that \(\pd{s^{(l)}(m)} + \pd{t^{(l)}(m)} = M'_l + \dfrac{\varepsilon}{2}\). Since \(M'_l +\varepsilon > M'_l + \dfrac{\varepsilon}{2} > M'_l\), then from Lemma \ref{lem:5.3}, \(N'_l > m > \max\{N_0, N_l\} \geq N_0\). The rest is to claim that \(m\) is an upper bound of set \eqref{eq:5.9}, and since \(m < N'_l\), this contradicts the assumption that \(N'_l\) is the least upper bound. Notice that it is sufficient to show that for any element \(m' \in (m, N'_l]\), \(m'\) does not belong to \eqref{eq:5.9}. Pick any \(m' \in (m, N'_l]\), then \(m' > m > N_0\). Since \(\pd{s^{(l)}(m)} + \pd{t^{(l)}(m)}\) is strictly increasing for any \(m \geq N_0\), then \(\pd{s^{(l)}(m')} + \pd{t^{(l)}(m')} > \pd{s^{(l)}(m)} + \pd{t^{(l)}(m)} = M'_l + \dfrac{\varepsilon}{2} > M'_l\), which implies \(m'\) does not belong to \eqref{eq:5.9}. \\
    
    \textit{Proof of} (c)\textit{.} From Equation \eqref{eq:5.6}, \(\pd{s^{(l)}(N_0)} + \pd{t^{(l)}(N_0)} \leq M'_l\), so \(N_0\) belongs to the set \eqref{eq:5.9}. Since \(N'_l\) is the upper bound of the set \eqref{eq:5.9}, then \(N'_l \geq N_0\). Suppose that \(N'_l = N_0\). Pick any \(m \in [1,N_0]\), then \(m \leq N_0 \leq \max\pc{N_0, N_l}\). By the contrapositive statement of Lemma \ref{lem:5.3}, \(\pd{s^{(l)}(m)} + \pd{t^{(l)}(m)} \leq M'_l\), as desired. \\
    
    Suppose that \(N'_l > N_0\). Since we just showed \(\pd{s^{(l)}(m)} + \pd{t^{(l)}(m)} \leq M'_l\) for all \(m \in [1, N_0]\), and also from Lemma \ref{lem:5.5}(b), \(\pd{s^{(l)}(N'_l)} + \pd{t^{(l)}(N'_l)} = M'_l\), so we are left to show for the case \(m \in (N_0, N'_l)\). Suppose that there exists \(m \in (N_0, N'_l)\) such that \(\pd{s^{(l)}(m)} + \pd{t^{(l)}(m)} > M'_l\). But since \(\pd{s^{(l)}(m)} + \pd{t^{(l)}(m)}\) is strictly increasing for any \(m \geq N_0\), then for \(N_0 < m < N'_l\), we have \(M'_l < \pd{s^{(l)}(m)} + \pd{t^{(l)}(m)} < \pd{s^{(l)}(N'_l)} + \pd{t^{(l)}(N'_l)} = M'_l\), which is a contradiction. 
\end{proof}

\vspace{5pt}
\begin{lemma}
    \label{lem:5.6}
    Let \(l \in \{1,2,3,4\}\), and \(N_0\), \(N_l\) and \(N'_l\) are defined as Equation \eqref{eq:3.2}, Equation \eqref{eq:3.3} and Equation \eqref{eq:5.10} respectively. Then \(N'_l \geq \max\pc{N_0, N_l}\).
\end{lemma}

\begin{proof}
    From Equation \eqref{eq:5.1},
    \begin{align*}
        M'_l &= \max\pc{\pd{s_{1}^{(l)}}, \pd{s_{\max\pc{N_0, N_l}}^{(l)}}} + \max\pc{\pd{t_{1}^{(l)}}, \pd{t_{\max\pc{N_0, N_l}}^{(l)}}},
    \end{align*}
    \begin{align*}
        &\geq \pd{s_{\max\pc{N_0, N_l}}^{(l)}} + \pd{t_{\max\pc{N_0, N_l}}^{(l)}}, \\
        &= \pd{s^{(l)}\pa{\max\pc{N_0, N_l}}} + \pd{t^{(l)}\pa{\max\pc{N_0, N_l}}}.
    \end{align*}
    This implies that \(\max\pc{N_0, N_l}\) is an element to the set \eqref{eq:5.9}. Since \(N'_l\) is an upper bound of \eqref{eq:5.9}, hence \(N'_l \geq \max\pc{N_0, N_l}\).
\end{proof}

\vspace{5pt}
\section{Sets \(\mD_{A}^{(k)}(x)\) where \(k \in \{0,1,2,3,4\}\)}
In this section, we recall back the set \(\mD_A(x)\) as defined in Equation \eqref{eq:1.14}. Let \(x \in \R\), \(x > 0\). Based from Theorem \ref{thm:1.3}(c), notice that we can rewrite the set \(\mD_A(x)\) as
\begin{align*}
    \mD_{A}(x) &= \Bigl\{(u,v) \in \Z^2: \mF_{A}(u,v) = 0, |u| + |v| \leq x\Bigr\}, \\
    &= \bigcup_{n=0}^{\pe{x}} \Bigl\{(u,v) \in \Z^2: \mF_{A}(u,v) = 0, |u| + |v| = n\Bigr\}, \\
    &= \bigcup_{n=0}^{\pe{x}} \biggl\{\pa{\tilde{u} - \lambda \tilde{v} + \dfrac{E}{D}\lambda - \dfrac{d}{2},\tilde{v} - \dfrac{E}{D}}: \\
    &\hspace{100pt} \pd{\tilde{u} - \lambda \tilde{v} + \dfrac{E}{D}\lambda - \dfrac{d}{2}} + \pd{\tilde{v} - \dfrac{E}{D}} = n, (\tilde{u}, \tilde{v}) \in \mP_{\tau}\biggr\}.
\end{align*}
By applying the fact \(\mP_{\tau} = \displaystyle\bigcup_{k=0}^{4} \mP_{\tau}^{(k)}\), we then have
\begin{align*}
    \mD_{A}(x) &= \bigcup_{n=0}^{\pe{x}} \bigcup_{k=0}^{4} \biggl\{\pa{\tilde{u} - \lambda \tilde{v} + \dfrac{E}{D}\lambda - \dfrac{d}{2},\tilde{v} - \dfrac{E}{D}}: \\
    &\hspace{100pt} \pd{\tilde{u} - \lambda \tilde{v} + \dfrac{E}{D}\lambda - \dfrac{d}{2}} + \pd{\tilde{v} - \dfrac{E}{D}} = n, (\tilde{u}, \tilde{v}) \in \mP^{(k)}_{\tau}\biggr\}, \\
    &= \bigcup_{k=0}^{4} \bigcup_{n=0}^{\pe{x}} \biggl\{\pa{\tilde{u} - \lambda \tilde{v} + \dfrac{E}{D}\lambda - \dfrac{d}{2},\tilde{v} - \dfrac{E}{D}}: \\
    &\hspace{100pt} \pd{\tilde{u} - \lambda \tilde{v} + \dfrac{E}{D}\lambda - \dfrac{d}{2}} + \pd{\tilde{v} - \dfrac{E}{D}} = n, (\tilde{u}, \tilde{v}) \in \mP^{(k)}_{\tau}\biggr\}, \\
    &= \bigcup_{k=0}^{4} \biggl\{\pa{\tilde{u} - \lambda \tilde{v} + \dfrac{E}{D}\lambda - \dfrac{d}{2},\tilde{v} - \dfrac{E}{D}}: \\
    &\hspace{100pt} \pd{\tilde{u} - \lambda \tilde{v} + \dfrac{E}{D}\lambda - \dfrac{d}{2}} + \pd{\tilde{v} - \dfrac{E}{D}} \leq \pe{x}, (\tilde{u}, \tilde{v}) \in \mP^{(k)}_{\tau}\biggr\}.
\end{align*}
For \(k \in \{0,1,2,3,4\}\), define the set
\begin{align}
    \notag
    \mD_{A}^{(k)}(x) &= \biggl\{\pa{\tilde{u} - \lambda \tilde{v} + \dfrac{E}{D}\lambda - \dfrac{d}{2},\tilde{v} - \dfrac{E}{D}}: \\
    \label{eq:6.1}
    &\hspace{70pt} \pd{\tilde{u} - \lambda \tilde{v} + \dfrac{E}{D}\lambda - \dfrac{d}{2}} + \pd{\tilde{v} - \dfrac{E}{D}} \leq \pe{x}, (\tilde{u}, \tilde{v}) \in \mP^{(k)}_{\tau}\biggr\},
\end{align}
so that \(\mD_A(x) = \displaystyle\bigcup_{k=0}^4 \mD_{A}^{(k)}(x)\). By applying Equation \eqref{eq:1.7} and Equation \eqref{eq:1.8}, we can simplify each \(\mD_{A}^{(k)}(x)\), \(k = 0,1,2,3,4\) in Equation \eqref{eq:6.1} as
\begin{align}
    \notag
    \mD_{A}^{(0)}(x) &= \pc{\pa{1 + \dfrac{E}{D}\lambda - \dfrac{d}{2}, - \dfrac{E}{D}} : \pd{1 + \dfrac{E}{D}\lambda- \dfrac{d}{2}} + \dfrac{\pd{E}}{D} \leq \pe{x}} \\
    \label{eq:6.2}
    &\hspace{30pt} \cup \pc{\pa{-1 + \dfrac{E}{D}\lambda - \dfrac{d}{2}, -\dfrac{E}{D}} : \pd{-1 + \dfrac{E}{D}\lambda- \dfrac{d}{2}} + \dfrac{\pd{E}}{D} \leq \pe{x}}, \\
    \label{eq:6.3}
    \mD_{A}^{(l)}(x) &= \Bigl\{\pa{s_m^{(l)}, t_m^{(l)}} : m \in \N, \pd{s_m^{(l)}} + \pd{t_m^{(l)}} \leq \pe{x}\Bigr\}, \ \text{for} \ l \in \{1,2,3,4\},
\end{align}
where \(s_m^{(l)}\) and \(t_m^{(l)}\) are integers as defined in Equation \eqref{eq:1.18} and Equation \eqref{eq:1.19} respectively.

\vspace{5pt}
\begin{lemma}
    \label{lem:6.1}
    Let \(x \in \R\), \(x > 0\). For any \(i,j \in \{0,1,2,3,4\}\) with \(i \neq j\),
    \[\mD_{A}^{(i)}(x) \cap \mD_{A}^{(j)}(x) = \varnothing.\]
\end{lemma}

\begin{proof}
    For \(k = 0,1,2,3,4\), we consider the set \(\mathcal{E}_{A}^{(k)}\):
    \begin{align}
        \label{eq:6.4}
        \mathcal{E}_{A}^{(0)} &= \pc{\pa{1 + \dfrac{E}{D}\lambda - \dfrac{d}{2}, - \dfrac{E}{D}}, \pa{-1 + \dfrac{E}{D}\lambda - \dfrac{d}{2}, - \dfrac{E}{D}}}, \ \text{and} \\
        \label{eq:6.5}
        \mathcal{E}_{A}^{(l)} &= \Bigl\{\pa{s_m^{(l)}, t_m^{(l)}} : m \in \N \Bigr\}, \ \text{for} \ l \in \{1,2,3,4\},
    \end{align}
    where \(s_m^{(l)}\) and \(t_m^{(l)}\) are integers as defined in Equation \eqref{eq:1.18} and Equation \eqref{eq:1.19} respectively. Notice that \(\mD_{A}^{(k)}(x) \subseteq \mathcal{E}_{A}^{(k)}\) for any \(x \in \R\), \(x > 0\) and \(k \in \{0,1,2,3,4\}\). We claim that for any \(i,j \in \{0,1,2,3,4\}\) with \(i \neq j\),
    \[\mathcal{E}_{A}^{(i)} \cap \mathcal{E}_{A}^{(j)} = \varnothing,\]
    so that it implies \(\mD_{A}^{(i)}(x) \cap \mD_{A}^{(j)}(x) = \varnothing\), and therefore the Lemma follows. \\
    
    We first prove \(\mathcal{E}_{A}^{(i)} \cap \mathcal{E}_{A}^{(0)} = \varnothing\) for all \(i \in \{1,2,3,4\}\). Assume that there exists \(\pa{u,v} \in \mathcal{E}_{A}^{(i)} \cap \mathcal{E}_{A}^{(0)}\) for some \(i \in \{1,2,3,4\}\). Since \(\pa{u,v} \in \mathcal{E}_{A}^{(i)}\), then from Equation \eqref{eq:6.5} and Equation \eqref{eq:1.19}, there exists \(m \in \N\) such that \(v = (-1)^{\pe{\frac{i-1}{2}}}v_m - \dfrac{E}{D}\). On the other hand, since \(\pa{u,v} \in \mathcal{E}_{A}^{(0)}\), then from Equation \eqref{eq:6.4}, \(v = - \dfrac{E}{D}\). But this yields \((-1)^{\pe{\frac{i-1}{2}}}v_m = 0\), which is a contradiction since \(v_m \neq 0\) for all \(m \in \N\). \\
    
    Next, we show \(\mathcal{E}_{A}^{(i)} \cap \mathcal{E}_{A}^{(j)} = \varnothing\) for any \(i = 1,2\) and \(j = 3,4\). Assume that there exists \(\pa{u,v} \in \mathcal{E}_{A}^{(i)} \cap \mathcal{E}_{A}^{(j)}\) for some \(i \in \{1,2\}\) and \(j \in \{3,4\}\). Then by Equation \eqref{eq:6.5} and Equation \eqref{eq:1.19}, there exists \(m \in \N\) such that \(v = v_m - \dfrac{E}{D}\) and \(v = -v_m - \dfrac{E}{D}\). But this yields \(v_m = 0\), which is a contradiction. \\
    
    The rest is to show \(\mathcal{E}_{A}^{(1)} \cap \mathcal{E}_{A}^{(2)} = \varnothing\) and \(\mathcal{E}_{A}^{(3)} \cap \mathcal{E}_{A}^{(4)} = \varnothing\). Assume that there exists \(\pa{u,v} \in \mathcal{E}_{A}^{(i)} \cap \mathcal{E}_{A}^{(i+1)}\) for some \(i \in \{1,3\}\). Then by Equation \eqref{eq:6.5} and Equation \eqref{eq:1.18}, there exists \(m \in \N\) such that
    \begin{align}
        \label{eq:6.6}
        u &= (-1)^{\pe{\frac{i}{2}}}u_m - (-1)^{\pe{\frac{i-1}{2}}}\lambda v_m + \dfrac{E}{D}\lambda- \dfrac{d}{2}, \ \text{and} \\
        \label{eq:6.7}
        u &= (-1)^{\pe{\frac{i+1}{2}}}u_m - (-1)^{\pe{\frac{i}{2}}}\lambda v_m + \dfrac{E}{D}\lambda- \dfrac{d}{2}.
    \end{align}
    Since \((-1)^{\pe{\frac{i+1}{2}}} = -(-1)^{\pe{\frac{i}{2}}}\) and \((-1)^{\pe{\frac{i}{2}}} = (-1)^{\pe{\frac{i-1}{2}}}\) for any \(i \in \{1,3\}\), Equation \eqref{eq:6.7} can be rewrite as
    \begin{align}
        \label{eq:6.8}
        u = -(-1)^{\pe{\frac{i}{2}}}u_m - (-1)^{\pe{\frac{i-1}{2}}}\lambda v_m + \dfrac{E}{D}\lambda- \dfrac{d}{2}.
    \end{align}
    Compare Equation \eqref{eq:6.6} and Equation \eqref{eq:6.8}, we get \(u_m = 0\), which is a contradiction since \(u_m \neq 0\) for all \(m \in \N\).
\end{proof}

From \(\mD_A(x) = \displaystyle\bigcup_{k=0}^4 \mD_{A}^{(k)}(x)\) and Lemma \ref{lem:6.1}, we can conclude that
\begin{align}
    \label{eq:6.9}
    \mD_{A}(x) &= \mD_{A}^{(0)}(x) \cup \pa{\bigcup_{l=1}^{4}\mD_{A}^{(l)}(x)}, \ \text{and} \\
    \label{eq:6.10}
    \Bigl|\mD_{A}(x)\Bigr| &= \pd{\mD_{A}^{(0)}(x)} + \pa{\sum_{l=1}^{4}\pd{\mD_{A}^{(l)}(x)}}.
\end{align}

\vspace{5pt}
\section{Number of Elements \(\pd{\mD_{A}^{(k)}(x)}\) where \(k \in \{0,1,2,3,4\}\)}

\vspace{5pt}
\subsection{Number of Elements \(\pd{\mD_{A}^{(l)}(x)}\) where \(l \in \{1,2,3,4\}\)}

\vspace{5pt}
\begin{lemma}
    \label{lem:7.1}
    Let \(l \in \{1,2,3,4\}\), and \(M'_l\) and \(N'_l\) are defined as Equation \eqref{eq:5.1} and Equation \eqref{eq:5.10} respectively. Then
    \[\pd{\mD_{A}^{(l)}\pa{M'_l}} = \pe{N'_l}.\]
\end{lemma}

\begin{proof}
    By Equation \eqref{eq:6.3}, \(\pd{\mD_{A}^{(l)}\pa{M'_l}}\) is equivalent to the number of positive integers \(m\) such that \(\pd{s_m^{(l)}} + \pd{t_m^{(l)}} \leq M'_l\). By applying Lemma \ref{lem:5.5}(c), any integer \(1 \leq m \leq \pe{N'_l}\) satisfies the condition, while by applying Lemma \ref{lem:5.5}(a), any integer \(m \geq \pe{N'_l} + 1\) does not satisfy the condition. Therefore, we must have \(\pd{\mD_{A}^{(l)}\pa{M'_l}} = \pe{N'_l}\).
\end{proof}

\vspace{5pt}
\begin{lemma}
    \label{lem:7.2}
    Let \(l \in \{1,2,3,4\}\) and \(M'_l\) as defined in Equation \eqref{eq:5.1}. For any real number \(x > M'_l\), consider the unique \(M \in [1,+\infty)\) such that \(\pd{s^{(l)}(M)} + \pd{t^{(l)}(M)} = x\). Then
    \[\pd{\mD_{A}^{(l)}\pa{x}} = \pe{M}.\]
\end{lemma}
Note: The number \(M\) exists and is unique by Lemma \ref{lem:5.4}.

\begin{proof}
    By Equation \eqref{eq:6.3}, \(\pd{\mD_{A}^{(l)}\pa{x}}\) is equivalent to the number of \(m \in \N\) such that \(\pd{s^{(l)}(m)} + \pd{t^{(l)}(m)} \leq \pe{x}\). We claim that any integer \(1 \leq m \leq \pe{M}\) satisfies the condition, while any integer \(m \geq \pe{M}+1\) does not satisfy the condition, which implies \(\pd{\mD_{A}^{(l)}\pa{x}} = \pe{M}\). \\

    Pick any integer \(m \geq \pe{M}+1\). Since \(x > M'_l\), then from Lemma \ref{lem:5.3}, \(M > \max\{N_0, N_l\} \geq N_0\). By Lemma \ref{eq:4.3}(b), \(\pd{s^{(l)}(m)} + \pd{t^{(l)}(m)}\) is strictly increasing for any \(m \geq N_0\). Since \(m \geq \pe{M}+1 > M > N_0\), then
    \[\pd{s^{(l)}(m)} + \pd{t^{(l)}(m)} > \pd{s^{(l)}(M)} + \pd{t^{(l)}(M)} = x \geq \pe{x},\]
    which does not satisfy the condition \(\pd{s^{(l)}(m)} + \pd{t^{(l)}(m)} \leq \pe{x}\). Now pick any integer \(1 \leq m \leq \pe{N'_l}\). Then from Lemma \ref{lem:5.5}(c), \(\pd{s^{(l)}(m)} + \pd{t^{(l)}(m)} \leq M'_l\). Since \(M'_l \leq \pe{x}\), then \(m\) satisfies the condition \(\pd{s^{(l)}(m)} + \pd{t^{(l)}(m)} \leq \pe{x}\). \\
    
    Notice that we must have
    \begin{align}
        \label{eq:7.1}
        \Bigl\{m \in \N : 1 \leq m \leq \pe{N'_l}\Bigr\} \cap \Bigl\{m \in \N : m \geq \pe{M}+1\Bigr\} = \varnothing,
    \end{align}
    otherwise if not, there exists \(m \in \N\) such that \(\pd{s^{(l)}(m)} + \pd{t^{(l)}(m)} \leq \pe{x}\) (since \(1 \leq m \leq \pe{N'_l}\)) and \(\pd{s^{(l)}(m)} + \pd{t^{(l)}(m)} > \pe{x}\) (since \(m \geq \pe{M}+1\)), which is a contradiction. Therefore, Equation \eqref{eq:7.1} implies \(\pe{N'_l} \leq \pe{M}\). If \(\pe{N'_l} = \pe{M}\), then we are done based on the proof above. Suppose that \(\pe{N'_l} < \pe{M}\). Then we are left to show any integer \(\pe{N'_l} + 1 \leq m \leq \pe{M}\) satisfies the condition \(\pd{s^{(l)}(m)} + \pd{t^{(l)}(m)} \leq \pe{x}\). \\
    
    Suppose that there exists integer \(\pe{N'_l} + 1 \leq m \leq \pe{M}\) such that \(\pd{s^{(l)}(m)} + \pd{t^{(l)}(m)} > \pe{x}\). Since \(m \in \Z\), then \(\pd{s^{(l)}(m)} + \pd{t^{(l)}(m)} \in \Z\) so \(\pd{s^{(l)}(m)} + \pd{t^{(l)}(m)} \geq \pe{x} + 1\). Since \(\pe{N'_l} + 1 > N'_l\), then by Lemma \ref{lem:5.5}(a), \(\pd{s^{(l)}(\pe{N'_l} + 1)} + \pd{t^{(l)}(\pe{N'_l} + 1)} > M'_l\). By applying Lemma \ref{lem:5.3}, this implies \(\pe{N'_l} + 1 > \max\{N_0, N_l\} \geq N_0\). By Lemma \ref{eq:4.3}(b), \(\pd{s^{(l)}(m)} + \pd{t^{(l)}(m)}\) is strictly increasing for any \(m \geq N_0\). So for \(N_0 < \pe{N'_l} + 1 \leq m \leq \pe{M} \leq M\), we obtain
    \[\pe{x} + 1\leq \pd{s^{(l)}(m)} + \pd{t^{(l)}(m)} \leq \pd{s^{(l)}(M)} + \pd{t^{(l)}(M)} = x,\]
    which is a contradiction since \(\pe{x} + 1 > x\) for any \(x \in \R\).
\end{proof}

\vspace{5pt}
\begin{lemma}
    \label{lem:7.3}
    Let \(l \in \{1,2,3,4\}\) and \(M'_l\) as defined in Equation \eqref{eq:5.1}. Pick \(n \in \N\) with \(n \geq M'_l\). If \(n = M'_l\), consider the value \(m = N'_l\) as defined in Equation \eqref{eq:5.10} so that \(\pd{s^{(l)}(N'_l)} + \pd{t^{(l)}(N'_l)} = M'_l\). If \(n > M'_l\), consider the unique \(m \in [1,+\infty)\) such that \(\pd{s^{(l)}(m)} + \pd{t^{(l)}(m)} = n\). Then 
    \[\pd{s^{(l)}(m)} + \pd{t^{(l)}(m)} = \pe{\dfrac{P_l}{2\sqrt{\tau}}\pa{\alpha + \beta\sqrt{\tau}}^m + Q_l} + R_l,\]
    where \(P_l\), \(Q_l\) and \(R_l\) are defined as Equation \eqref{eq:1.15}, Equation \eqref{eq:1.16} and Equation \eqref{eq:1.17} respectively.
\end{lemma}

\begin{proof}
    \textit{Proof for \(l \in \{1,3\}\).} Let \(\lambda < \sqrt{\tau}\).  If \(n = M'_l\), then by assumption \(m = N'_l\), so by Lemma \ref{lem:5.6}, \(m = N'_l \geq \max\{N_0, N_l\} \geq N_0\). Otherwise if \(n > M'_l\), then by Lemma \ref{lem:5.3}, \(m > \max\{N_0, N_l\} \geq N_0\). Hence \(n \geq M'_l\) implies \(m \geq \max\{N_0, N_l\} \geq N_0\). \\
    
    By Lemma \ref{lem:4.1}(a), function \(s^{(1)}(m)\) is strictly increasing and function \(s^{(3)}(m)\) is strictly decreasing. Then by Lemma \ref{lem:4.2}(a), \(s^{(1)}(m) > 0\) and \(s^{(3)}(m) < 0\). Similarly, by Lemma \ref{lem:4.1}(e), function \(t^{(1)}(m)\) is strictly increasing and function \(t^{(3)}(m)\) is strictly decreasing. Then by Lemma \ref{lem:4.2}(b), \(t^{(1)}(m) > 0\) and \(t^{(3)}(m) < 0\). By applying Equation \eqref{eq:4.1} and Equation \eqref{eq:4.2}, and notice that \((-1)^{\pe{\frac{l-1}{2}}} = (-1)^{\pe{\frac{l}{2}}}\) for \(l \in \{1,3\}\), we get
    \begin{align}
        \notag
        &\pd{s^{(l)}(m)} + \pd{t^{(l)}(m)} \\
        \notag
        &= (-1)^{\pe{\frac{l}{2}}}\pa{(-1)^{\pe{\frac{l}{2}}}u(m) - (-1)^{\pe{\frac{l-1}{2}}}\lambda v(m) + \dfrac{E}{D}\lambda- \dfrac{d}{2}} \\
        \notag
        &\hspace{200pt} +(-1)^{\pe{\frac{l}{2}}}\pa{(-1)^{\pe{\frac{l-1}{2}}}v(m) - \dfrac{E}{D}}, \\
        \notag
        &= (-1)^{\pe{\frac{l}{2}}}\pa{(-1)^{\pe{\frac{l}{2}}}u(m) - (-1)^{\pe{\frac{l}{2}}}\lambda v(m) + \dfrac{E}{D}\lambda- \dfrac{d}{2}} \\
        \notag
        &\hspace{200pt} +(-1)^{\pe{\frac{l}{2}}}\pa{(-1)^{\pe{\frac{l}{2}}}v(m) - \dfrac{E}{D}}, \\
        \notag
        &= \pa{u(m) - \lambda v(m) + (-1)^{\pe{\frac{l}{2}}}\pa{\dfrac{E}{D}\lambda- \dfrac{d}{2}}} + \pa{v(m) - (-1)^{\pe{\frac{l}{2}}}\dfrac{E}{D}}, \\
        \label{eq:7.2}
        &= u(m) + \pa{1 -\lambda} v(m) + (-1)^{\pe{\frac{l}{2}}}\pa{\dfrac{E}{D}\lambda - \dfrac{E}{D} - \dfrac{d}{2}}.
    \end{align}
    By substituting Equation \eqref{eq:3.4} and Equation \eqref{eq:3.5} into Equation \eqref{eq:7.2}, we get
    \begin{align}
        \notag
        &\pd{s^{(l)}(m)} + \pd{t^{(l)}(m)} \\
        \notag
        &= \dfrac{1}{2}\pc{\pa{\alpha + \beta\sqrt{\tau}}^m + \pa{\alpha - \beta\sqrt{\tau}}^m} \\
        \notag
        &\hspace{15pt} + \dfrac{1}{2\sqrt{\tau}}\pa{1 -\lambda} \pc{\pa{\alpha + \beta\sqrt{\tau}}^m - \pa{\alpha - \beta\sqrt{\tau}}^m} + (-1)^{\pe{\frac{l}{2}}}\pa{\dfrac{E}{D}\lambda - \dfrac{E}{D} - \dfrac{d}{2}}, \\
        \notag
        &= \pa{\dfrac{1}{2\sqrt{\tau}}\pa{1-\lambda + \sqrt{\tau}}\pa{\alpha + \beta\sqrt{\tau}}^m + (-1)^{\pe{\frac{l}{2}}}\pa{\dfrac{E}{D}\lambda - \dfrac{E}{D} - \dfrac{d}{2}}} \\
        \label{eq:7.3}
        &\hspace{180pt} - \pa{\dfrac{1}{2\sqrt{\tau}}\pa{1-\lambda - \sqrt{\tau}}\pa{\alpha - \beta\sqrt{\tau}}^m}.
    \end{align}
    By assumption, \(\lambda < \sqrt{\tau}\). Since \(m \geq \max\{N_0, N_l\} \geq N_l\), then by Lemma \ref{lem:3.8}, \(m\) satisfies Equation \eqref{eq:3.12}. Since \(\dfrac{1}{2\sqrt{\tau}} > 0\) and \(\pa{\alpha - \beta\sqrt{\tau}}^m > 0\), so Equation \eqref{eq:3.12} implies
    \begin{align}
        \label{eq:7.4}
        0 &< \dfrac{1}{2\sqrt{\tau}}\pa{1-\lambda - \sqrt{\tau}}\pa{\alpha - \beta\sqrt{\tau}}^m < 1 \ & \ &\text{if} \ \lambda < 1-\sqrt{\tau}, \\
        \label{eq:7.5}
        -1 &< \dfrac{1}{2\sqrt{\tau}}\pa{1-\lambda - \sqrt{\tau}}\pa{\alpha - \beta\sqrt{\tau}}^m < 0 \ & \ &\text{if} \  1- \sqrt{\tau} < \lambda < \sqrt{\tau}.
    \end{align}
    If \(\lambda < 1-\sqrt{\tau}\), then from Equation \eqref{eq:7.3} and Equation \eqref{eq:7.4},
    \begin{align*}
        \dfrac{1}{2\sqrt{\tau}}\pa{1-\lambda + \sqrt{\tau}}\pa{\alpha + \beta\sqrt{\tau}}^m + (-1)^{\pe{\frac{l}{2}}}\pa{\dfrac{E}{D}\lambda - \dfrac{E}{D} - \dfrac{d}{2}} &> \pd{s^{(l)}(m)} + \pd{t^{(l)}(m)} \ \text{and}, \\
        \dfrac{1}{2\sqrt{\tau}}\pa{1-\lambda + \sqrt{\tau}}\pa{\alpha + \beta\sqrt{\tau}}^m + (-1)^{\pe{\frac{l}{2}}}\pa{\dfrac{E}{D}\lambda - \dfrac{E}{D} - \dfrac{d}{2}} &< \pd{s^{(l)}(m)} + \pd{t^{(l)}(m)} + 1.
    \end{align*}
    By assumption, \(\pd{s^{(l)}(m)} + \pd{t^{(l)}(m)} \in \Z\). Therefore,
    \[\pd{s^{(l)}(m)} + \pd{t^{(l)}(m)} = \pe{\dfrac{1}{2\sqrt{\tau}}\pa{1-\lambda + \sqrt{\tau}}\pa{\alpha + \beta\sqrt{\tau}}^m + (-1)^{\pe{\frac{l}{2}}}\pa{\dfrac{E}{D}\lambda - \dfrac{E}{D} - \dfrac{d}{2}}}.\]
    On the other hand, if \(1- \sqrt{\tau} < \lambda < \sqrt{\tau}\), then from Equation \eqref{eq:7.3} and Equation \eqref{eq:7.5},
    \begin{align*}
        \dfrac{1}{2\sqrt{\tau}}\pa{1-\lambda + \sqrt{\tau}}\pa{\alpha + \beta\sqrt{\tau}}^m + (-1)^{\pe{\frac{l}{2}}}\pa{\dfrac{E}{D}\lambda - \dfrac{E}{D} - \dfrac{d}{2}} &> \pd{s^{(l)}(m)} + \pd{t^{(l)}(m)} - 1 \ \text{and}, \\
        \dfrac{1}{2\sqrt{\tau}}\pa{1-\lambda + \sqrt{\tau}}\pa{\alpha + \beta\sqrt{\tau}}^m + (-1)^{\pe{\frac{l}{2}}}\pa{\dfrac{E}{D}\lambda - \dfrac{E}{D} - \dfrac{d}{2}} &< \pd{s^{(l)}(m)} + \pd{t^{(l)}(m)}.
    \end{align*}
    By assumption, \(\pd{s^{(l)}(m)} + \pd{t^{(l)}(m)} \in \Z\). Therefore,
    \[\pd{s^{(l)}(m)} + \pd{t^{(l)}(m)} = \pe{\dfrac{1}{2\sqrt{\tau}}\pa{1-\lambda + \sqrt{\tau}}\pa{\alpha + \beta\sqrt{\tau}}^m + (-1)^{\pe{\frac{l}{2}}}\pa{\dfrac{E}{D}\lambda - \dfrac{E}{D} - \dfrac{d}{2}}} + 1.\]

    Now let \(\lambda > \sqrt{\tau}\). If \(n = M'_l\), then by assumption \(m = N'_l\), so by Lemma \ref{lem:5.6}, \(m = N'_l \geq \max\{N_0, N_l\} \geq N_0\). Otherwise if \(n > M'_l\), then by Lemma \ref{lem:5.3}, \(m > \max\{N_0, N_l\} \geq N_0\). Hence \(n \geq M'_l\) implies \(m \geq \max\{N_0, N_l\} \geq N_0\). \\
    
    By Lemma \ref{lem:4.1}(b), function \(s^{(1)}(m)\) is strictly decreasing and function \(s^{(3)}(m)\) is strictly increasing. Then by Lemma \ref{lem:4.2}(a), \(s^{(1)}(m) < 0\) and \(s^{(3)}(m) > 0\). Similarly, by Lemma \ref{lem:4.1}(e) and Lemma \ref{lem:4.2}(b), \(t^{(1)}(m) > 0\) and \(t^{(3)}(m) < 0\). By applying Equation \eqref{eq:4.1} and Equation \eqref{eq:4.2}, and notice that \((-1)^{\pe{\frac{l-1}{2}}} = (-1)^{\pe{\frac{l}{2}}}\) for \(l \in \{1,3\}\), we get
    \begin{align}
        \notag
        &\pd{s^{(l)}(m)} + \pd{t^{(l)}(m)} \\
        \notag
        &= -(-1)^{\pe{\frac{l}{2}}}\pa{(-1)^{\pe{\frac{l}{2}}}u(m) - (-1)^{\pe{\frac{l-1}{2}}}\lambda v(m) + \dfrac{E}{D}\lambda- \dfrac{d}{2}} \\
        \notag
        &\hspace{200pt} +(-1)^{\pe{\frac{l}{2}}}\pa{(-1)^{\pe{\frac{l-1}{2}}}v(m) - \dfrac{E}{D}}, \\
        \notag
        &= -(-1)^{\pe{\frac{l}{2}}}\pa{(-1)^{\pe{\frac{l}{2}}}u(m) - (-1)^{\pe{\frac{l}{2}}}\lambda v(m) + \dfrac{E}{D}\lambda- \dfrac{d}{2}} \\
        \notag
        &\hspace{200pt} +(-1)^{\pe{\frac{l}{2}}}\pa{(-1)^{\pe{\frac{l}{2}}}v(m) - \dfrac{E}{D}}, \\
        \notag
        &= \pa{-u(m) + \lambda v(m) - (-1)^{\pe{\frac{l}{2}}}\pa{\dfrac{E}{D}\lambda- \dfrac{d}{2}}} + \pa{v(m) - (-1)^{\pe{\frac{l}{2}}}\dfrac{E}{D}}, \\
        \label{eq:7.6}
        &= -u(m) + \pa{1 +\lambda} v(m) - (-1)^{\pe{\frac{l}{2}}}\pa{\dfrac{E}{D}\lambda + \dfrac{E}{D} - \dfrac{d}{2}}.
    \end{align}
    By substituting Equation \eqref{eq:3.4} and Equation \eqref{eq:3.5} into Equation \eqref{eq:7.6}, we get
    \begin{align}
        \notag
        &\pd{s^{(l)}(m)} + \pd{t^{(l)}(m)} \\
        \notag
        &= -\dfrac{1}{2}\pc{\pa{\alpha + \beta\sqrt{\tau}}^m + \pa{\alpha - \beta\sqrt{\tau}}^m} \\
        \notag
        &\hspace{15pt} + \dfrac{1}{2\sqrt{\tau}}\pa{1 +\lambda} \pc{\pa{\alpha + \beta\sqrt{\tau}}^m - \pa{\alpha - \beta\sqrt{\tau}}^m} - (-1)^{\pe{\frac{l}{2}}}\pa{\dfrac{E}{D}\lambda + \dfrac{E}{D} - \dfrac{d}{2}}, \\
        \notag
        &= \pa{\dfrac{1}{2\sqrt{\tau}}\pa{1+\lambda - \sqrt{\tau}}\pa{\alpha + \beta\sqrt{\tau}}^m - (-1)^{\pe{\frac{l}{2}}}\pa{\dfrac{E}{D}\lambda + \dfrac{E}{D} - \dfrac{d}{2}}} \\
        \label{eq:7.7}
        &\hspace{180pt} - \pa{\dfrac{1}{2\sqrt{\tau}}\pa{1+\lambda + \sqrt{\tau}}\pa{\alpha - \beta\sqrt{\tau}}^m}.
    \end{align}
    By assumption, \(\lambda > \sqrt{\tau}\). Since \(m \geq \max\{N_0, N_l\} \geq N_l\), then by Lemma \ref{lem:3.8}, \(m\) satisfies Equation \eqref{eq:3.13}. Since \(\dfrac{1}{2\sqrt{\tau}} > 0\), \(\pa{1+\lambda + \sqrt{\tau}} > 0\) and \(\pa{\alpha - \beta\sqrt{\tau}}^m > 0\), so Equation \eqref{eq:3.13} implies
    \begin{align}
        \label{eq:7.8}
        0 &< \dfrac{1}{2\sqrt{\tau}}\pa{1+\lambda + \sqrt{\tau}}\pa{\alpha - \beta\sqrt{\tau}}^m < 1.
    \end{align}
    From Equation \eqref{eq:7.7} and Equation \eqref{eq:7.8},
    \begin{align*}
        \dfrac{1}{2\sqrt{\tau}}\pa{1+\lambda - \sqrt{\tau}}\pa{\alpha + \beta\sqrt{\tau}}^m - (-1)^{\pe{\frac{l}{2}}}\pa{\dfrac{E}{D}\lambda + \dfrac{E}{D} - \dfrac{d}{2}} &> \pd{s^{(l)}(m)} + \pd{t^{(l)}(m)} \ \text{and}, \\
        \dfrac{1}{2\sqrt{\tau}}\pa{1+\lambda - \sqrt{\tau}}\pa{\alpha + \beta\sqrt{\tau}}^m - (-1)^{\pe{\frac{l}{2}}}\pa{\dfrac{E}{D}\lambda + \dfrac{E}{D} - \dfrac{d}{2}} &< \pd{s^{(l)}(m)} + \pd{t^{(l)}(m)} + 1.
    \end{align*}
    By assumption, \(\pd{s^{(l)}(m)} + \pd{t^{(l)}(m)} \in \Z\). Therefore,
    \[\pd{s^{(l)}(m)} + \pd{t^{(l)}(m)} = \pe{\dfrac{1}{2\sqrt{\tau}}\pa{1+\lambda - \sqrt{\tau}}\pa{\alpha + \beta\sqrt{\tau}}^m - (-1)^{\pe{\frac{l}{2}}}\pa{\dfrac{E}{D}\lambda + \dfrac{E}{D} - \dfrac{d}{2}}}.\]

    \textit{Proof for \(l \in \{2,4\}\).} Let \(\lambda < -\sqrt{\tau}\). If \(n = M'_l\), then by Lemma \ref{lem:5.6}, \(m = N'_l \geq \max\{N_0, N_l\} \geq N_0\). Otherwise if \(n > M'_l\), then by Lemma \ref{lem:5.3}, \(m > \max\{N_0, N_l\} \geq N_0\). Hence \(n \geq M'_l\) implies \(m \geq \max\{N_0, N_l\} \geq N_0\). \\
    
    By Lemma \ref{lem:4.1}(c), function \(s^{(2)}(m)\) is strictly increasing and function \(s^{(4)}(m)\) is strictly decreasing. Then by Lemma \ref{lem:4.2}(a), \(s^{(2)}(m) > 0\) and \(s^{(4)}(m) < 0\). Similarly, by Lemma \ref{lem:4.1}(e), function \(t^{(2)}(m)\) is strictly increasing and function \(t^{(4)}(m)\) is strictly decreasing. Then by Lemma \ref{lem:4.2}(b), \(t^{(2)}(m) > 0\) and \(t^{(4)}(m) < 0\). By applying Equation \eqref{eq:4.1} and Equation \eqref{eq:4.2}, and notice that \((-1)^{\pe{\frac{l-1}{2}}} = -(-1)^{\pe{\frac{l}{2}}}\) for \(l \in \{2,4\}\), we get
    \begin{align}
        \notag
        &\pd{s^{(l)}(m)} + \pd{t^{(l)}(m)} \\
        \notag
        &= -(-1)^{\pe{\frac{l}{2}}}\pa{(-1)^{\pe{\frac{l}{2}}}u(m) - (-1)^{\pe{\frac{l-1}{2}}}\lambda v(m) + \dfrac{E}{D}\lambda- \dfrac{d}{2}} \\
        \notag
        &\hspace{200pt} -(-1)^{\pe{\frac{l}{2}}}\pa{(-1)^{\pe{\frac{l-1}{2}}}v(m) - \dfrac{E}{D}}, \\
        \notag
        &= -(-1)^{\pe{\frac{l}{2}}}\pa{(-1)^{\pe{\frac{l}{2}}}u(m) + (-1)^{\pe{\frac{l}{2}}}\lambda v(m) + \dfrac{E}{D}\lambda- \dfrac{d}{2}} \\
        \notag
        &\hspace{200pt} -(-1)^{\pe{\frac{l}{2}}}\pa{-(-1)^{\pe{\frac{l}{2}}}v(m) - \dfrac{E}{D}}, \\
        \notag
        &= \pa{-u(m) - \lambda v(m) - (-1)^{\pe{\frac{l}{2}}}\pa{\dfrac{E}{D}\lambda- \dfrac{d}{2}}} + \pa{v(m) + (-1)^{\pe{\frac{l}{2}}}\dfrac{E}{D}}, \\
        \label{eq:7.9}
        &= -u(m) + \pa{1 -\lambda} v(m) - (-1)^{\pe{\frac{l}{2}}}\pa{\dfrac{E}{D}\lambda - \dfrac{E}{D} - \dfrac{d}{2}}.
    \end{align}
    By substituting Equation \eqref{eq:3.4} and Equation \eqref{eq:3.5} into Equation \eqref{eq:7.9}, we get
    \begin{align}
        \notag
        &\pd{s^{(l)}(m)} + \pd{t^{(l)}(m)} \\
        \notag
        &= -\dfrac{1}{2}\pc{\pa{\alpha + \beta\sqrt{\tau}}^m + \pa{\alpha - \beta\sqrt{\tau}}^m} \\
        \notag
        &\hspace{15pt} + \dfrac{1}{2\sqrt{\tau}}\pa{1 -\lambda} \pc{\pa{\alpha + \beta\sqrt{\tau}}^m - \pa{\alpha - \beta\sqrt{\tau}}^m} - (-1)^{\pe{\frac{l}{2}}}\pa{\dfrac{E}{D}\lambda - \dfrac{E}{D} - \dfrac{d}{2}}, \\
        \notag
        &= \pa{\dfrac{1}{2\sqrt{\tau}}\pa{1-\lambda - \sqrt{\tau}}\pa{\alpha + \beta\sqrt{\tau}}^m - (-1)^{\pe{\frac{l}{2}}}\pa{\dfrac{E}{D}\lambda - \dfrac{E}{D} - \dfrac{d}{2}}} \\
        \label{eq:7.10}
        &\hspace{180pt} - \pa{\dfrac{1}{2\sqrt{\tau}}\pa{1-\lambda + \sqrt{\tau}}\pa{\alpha - \beta\sqrt{\tau}}^m}.
    \end{align}
    By assumption, \(\lambda < -\sqrt{\tau}\). Since \(m \geq \max\{N_0, N_l\} \geq N_l\), then by Lemma \ref{lem:3.8}, \(m\) satisfies Equation \eqref{eq:3.12}. Since \(\dfrac{1}{2\sqrt{\tau}} > 0\), \(\pa{1-\lambda + \sqrt{\tau}} > 0\) and \(\pa{\alpha - \beta\sqrt{\tau}}^m > 0\), so Equation \eqref{eq:3.12} implies
    \begin{align}
        \label{eq:7.11}
        0 &< \dfrac{1}{2\sqrt{\tau}}\pa{1-\lambda + \sqrt{\tau}}\pa{\alpha - \beta\sqrt{\tau}}^m < 1.
    \end{align}
    From Equation \eqref{eq:7.10} and Equation \eqref{eq:7.11},
    \begin{align*}
        \dfrac{1}{2\sqrt{\tau}}\pa{1-\lambda - \sqrt{\tau}}\pa{\alpha + \beta\sqrt{\tau}}^m - (-1)^{\pe{\frac{l}{2}}}\pa{\dfrac{E}{D}\lambda - \dfrac{E}{D} - \dfrac{d}{2}} &> \pd{s^{(l)}(m)} + \pd{t^{(l)}(m)} \ \text{and}, \\
        \dfrac{1}{2\sqrt{\tau}}\pa{1-\lambda - \sqrt{\tau}}\pa{\alpha + \beta\sqrt{\tau}}^m - (-1)^{\pe{\frac{l}{2}}}\pa{\dfrac{E}{D}\lambda - \dfrac{E}{D} - \dfrac{d}{2}} &< \pd{s^{(l)}(m)} + \pd{t^{(l)}(m)} + 1.
    \end{align*}
    By assumption, \(\pd{s^{(l)}(m)} + \pd{t^{(l)}(m)} \in \Z\). Therefore,
    \[\pd{s^{(l)}(m)} + \pd{t^{(l)}(m)} = \pe{\dfrac{1}{2\sqrt{\tau}}\pa{1-\lambda - \sqrt{\tau}}\pa{\alpha + \beta\sqrt{\tau}}^m - (-1)^{\pe{\frac{l}{2}}}\pa{\dfrac{E}{D}\lambda - \dfrac{E}{D} - \dfrac{d}{2}}}.\]
    Now let \(\lambda > -\sqrt{\tau}\). If \(n = M'_l\), then by Lemma \ref{lem:5.6}, \(m = N'_l \geq \max\{N_0, N_l\} \geq N_0\). Otherwise if \(n > M'_l\), then by Lemma \ref{lem:5.3}, \(m > \max\{N_0, N_l\} \geq N_0\). Hence \(n \geq M'_l\) implies \(m \geq \max\{N_0, N_l\} \geq N_0\). \\
    
    By Lemma \ref{lem:4.1}(d), function \(s^{(2)}(m)\) is strictly decreasing and function \(s^{(4)}(m)\) is strictly increasing. Then by Lemma \ref{lem:4.2}(a), \(s^{(2)}(m) < 0\) and \(s^{(4)}(m) > 0\). Similarly, by Lemma \ref{lem:4.1}(e) and Lemma \ref{lem:4.2}(b), \(t^{(2)}(m) > 0\) and \(t^{(4)}(m) < 0\). By applying Equation \eqref{eq:4.1} and Equation \eqref{eq:4.2}, and notice that \((-1)^{\pe{\frac{l-1}{2}}} = -(-1)^{\pe{\frac{l}{2}}}\) for \(l \in \{2,4\}\), we get
    \begin{align}
        \notag
        &\pd{s^{(l)}(m)} + \pd{t^{(l)}(m)} \\
        \notag
        &= (-1)^{\pe{\frac{l}{2}}}\pa{(-1)^{\pe{\frac{l}{2}}}u(m) - (-1)^{\pe{\frac{l-1}{2}}}\lambda v(m) + \dfrac{E}{D}\lambda- \dfrac{d}{2}} \\
        \notag
        &\hspace{200pt} -(-1)^{\pe{\frac{l}{2}}}\pa{(-1)^{\pe{\frac{l-1}{2}}}v(m) - \dfrac{E}{D}}, \\
        \notag
        &= (-1)^{\pe{\frac{l}{2}}}\pa{(-1)^{\pe{\frac{l}{2}}}u(m) + (-1)^{\pe{\frac{l}{2}}}\lambda v(m) + \dfrac{E}{D}\lambda- \dfrac{d}{2}} \\
        \notag
        &\hspace{200pt} -(-1)^{\pe{\frac{l}{2}}}\pa{-(-1)^{\pe{\frac{l}{2}}}v(m) - \dfrac{E}{D}}, \\
        \notag
        &= \pa{u(m) + \lambda v(m) + (-1)^{\pe{\frac{l}{2}}}\pa{\dfrac{E}{D}\lambda- \dfrac{d}{2}}} + \pa{v(m) + (-1)^{\pe{\frac{l}{2}}}\dfrac{D}{E}}, \\
        \label{eq:7.12}
        &= u(m) + \pa{1 +\lambda} v(m) + (-1)^{\pe{\frac{l}{2}}}\pa{\dfrac{E}{D}\lambda + \dfrac{E}{D} - \dfrac{d}{2}}.
    \end{align}
    By substituting Equation \eqref{eq:3.4} and Equation \eqref{eq:3.5} into Equation \eqref{eq:7.12}, we get
    \begin{align}
        \notag
        &\pd{s^{(l)}(m)} + \pd{t^{(l)}(m)} \\
        \notag
        &= \dfrac{1}{2}\pc{\pa{\alpha + \beta\sqrt{\tau}}^m + \pa{\alpha - \beta\sqrt{\tau}}^m} \\
        \notag
        &\hspace{15pt} + \dfrac{1}{2\sqrt{\tau}}\pa{1 +\lambda} \pc{\pa{\alpha + \beta\sqrt{\tau}}^m - \pa{\alpha - \beta\sqrt{\tau}}^m} + (-1)^{\pe{\frac{l}{2}}}\pa{\dfrac{E}{D}\lambda + \dfrac{E}{D} - \dfrac{d}{2}}, \\
        \notag
        &= \pa{\dfrac{1}{2\sqrt{\tau}}\pa{1+\lambda + \sqrt{\tau}}\pa{\alpha + \beta\sqrt{\tau}}^m + (-1)^{\pe{\frac{l}{2}}}\pa{\dfrac{E}{D}\lambda + \dfrac{E}{D} - \dfrac{d}{2}}} \\
        \label{eq:7.13}
        &\hspace{180pt} - \pa{\dfrac{1}{2\sqrt{\tau}}\pa{1+\lambda - \sqrt{\tau}}\pa{\alpha - \beta\sqrt{\tau}}^m}.
    \end{align}
    By assumption, \(\lambda > -\sqrt{\tau}\). Since \(m \geq \max\{N_0, N_l\} \geq N_l\), then by Lemma \ref{lem:3.8}, \(m\) satisfies Equation \eqref{eq:3.13}. Since \(\dfrac{1}{2\sqrt{\tau}} > 0\) and \(\pa{\alpha - \beta\sqrt{\tau}}^m > 0\), so Equation \eqref{eq:3.13} implies
    \begin{align}
        \label{eq:7.14}
        0 &< \dfrac{1}{2\sqrt{\tau}}\pa{1+\lambda - \sqrt{\tau}}\pa{\alpha - \beta\sqrt{\tau}}^m < 1 \ & \ &\text{if} \ \lambda > -1+\sqrt{\tau}, \\
        \label{eq:7.15}
        -1 &< \dfrac{1}{2\sqrt{\tau}}\pa{1+\lambda - \sqrt{\tau}}\pa{\alpha - \beta\sqrt{\tau}}^m < 0 \ & \ &\text{if} \  -\sqrt{\tau} < \lambda < -1+\sqrt{\tau}.
    \end{align}
    If \(\lambda > -1+\sqrt{\tau}\), then from Equation \eqref{eq:7.13} and Equation \eqref{eq:7.14},
    \begin{align*}
        \dfrac{1}{2\sqrt{\tau}}\pa{1+\lambda + \sqrt{\tau}}\pa{\alpha + \beta\sqrt{\tau}}^m + (-1)^{\pe{\frac{l}{2}}}\pa{\dfrac{E}{D}\lambda + \dfrac{E}{D} - \dfrac{d}{2}} &> \pd{s^{(l)}(m)} + \pd{t^{(l)}(m)} \ \text{and}, \\
        \dfrac{1}{2\sqrt{\tau}}\pa{1+\lambda + \sqrt{\tau}}\pa{\alpha + \beta\sqrt{\tau}}^m + (-1)^{\pe{\frac{l}{2}}}\pa{\dfrac{E}{D}\lambda + \dfrac{E}{D} - \dfrac{d}{2}} &< \pd{s^{(l)}(m)} + \pd{t^{(l)}(m)} + 1.
    \end{align*}
    By assumption, \(\pd{s^{(l)}(m)} + \pd{t^{(l)}(m)} \in \Z\). Therefore,
    \[\pd{s^{(l)}(m)} + \pd{t^{(l)}(m)} = \pe{\dfrac{1}{2\sqrt{\tau}}\pa{1+\lambda + \sqrt{\tau}}\pa{\alpha + \beta\sqrt{\tau}}^m + (-1)^{\pe{\frac{l}{2}}}\pa{\dfrac{E}{D}\lambda + \dfrac{E}{D} - \dfrac{d}{2}}}.\]
    On the other hand, if \(-\sqrt{\tau} < \lambda < -1+\sqrt{\tau}\), then from Equation \eqref{eq:7.13} and Equation \eqref{eq:7.15},
    \begin{align*}
        \dfrac{1}{2\sqrt{\tau}}\pa{1+\lambda + \sqrt{\tau}}\pa{\alpha + \beta\sqrt{\tau}}^m + (-1)^{\pe{\frac{l}{2}}}\pa{\dfrac{E}{D}\lambda + \dfrac{E}{D} - \dfrac{d}{2}} &> \pd{s^{(l)}(m)} + \pd{t^{(l)}(m)} - 1 \ \text{and}, \\
        \dfrac{1}{2\sqrt{\tau}}\pa{1+\lambda + \sqrt{\tau}}\pa{\alpha + \beta\sqrt{\tau}}^m + (-1)^{\pe{\frac{l}{2}}}\pa{\dfrac{E}{D}\lambda + \dfrac{E}{D} - \dfrac{d}{2}} &< \pd{s^{(l)}(m)} + \pd{t^{(l)}(m)}.
    \end{align*}
    By assumption, \(\pd{s^{(l)}(m)} + \pd{t^{(l)}(m)} \in \Z\). Therefore,
    \[\pd{s^{(l)}(m)} + \pd{t^{(l)}(m)} = \pe{\dfrac{1}{2\sqrt{\tau}}\pa{1+\lambda + \sqrt{\tau}}\pa{\alpha + \beta\sqrt{\tau}}^m + (-1)^{\pe{\frac{l}{2}}}\pa{\dfrac{E}{D}\lambda + \dfrac{E}{D} - \dfrac{d}{2}}} + 1.\]
\end{proof}

\vspace{5pt}
\begin{lemma}
    \label{lem:7.4}
    Let \(l \in \{1,2,3,4\}\) and \(M'_l\) as defined in Equation \eqref{eq:5.1}. For any positive integer \(n \geq M'_l\), 
    \[\pd{\mD_{A}^{(l)}(n)} = \pe{\dfrac{\log\pa{n - R_l + 1- Q_l} - \log P_l + \log \pa{2\sqrt{\tau}}}{\log \pa{\alpha + \beta\sqrt{\tau}}}},\]
    where \(P_l\), \(Q_l\) and \(R_l\) are defined as Equation \eqref{eq:1.15}, Equation \eqref{eq:1.16} and Equation \eqref{eq:1.17} respectively.
\end{lemma}

\begin{proof}
    If \(n = M'_l\), then from Lemma \ref{lem:5.5}(b), there exists \(m = N'_l\) such that \(\pd{s^{(l)}(N'_l)} + \pd{t^{(l)}(N'_l)} = M'_l\), where \(N'_l\) is defined as Equation \eqref{eq:5.10}. Otherwise if \(n > M'_l\), then from Lemma \ref{lem:5.4}, there exists unique \(m \in [1, +\infty)\) such that \(\pd{s^{(l)}(m)} + \pd{t^{(l)}(m)} = n\). Therefore for any positive integer \(n \geq M'_l\) with its respective \(m \in [1, +\infty)\), by referring to Lemma \ref{lem:7.1} and Lemma \ref{lem:7.2}, we have \(\pd{\mD_{A}^{(l)}\pa{n}} = \pe{m}\). Thus, it is sufficient to show
    \begin{align}
        \label{eq:7.16}
        \pe{\dfrac{\log\pa{n - R_l + 1- Q_l} - \log P_l + \log \pa{2\sqrt{\tau}}}{\log \pa{\alpha + \beta\sqrt{\tau}}}} = \pe{m}.
    \end{align}
    From Lemma \ref{lem:7.3},
    \begin{align*}
        n = \pd{s^{(l)}(m)} + \pd{t^{(l)}(m)} = \pe{\dfrac{P_l}{2\sqrt{\tau}}\pa{\alpha + \beta\sqrt{\tau}}^m + Q_l} + R_l,
    \end{align*}
    where \(P_l\), \(Q_l\) and \(R_l\) are defined as Equation \eqref{eq:1.15}, Equation \eqref{eq:1.16} and Equation \eqref{eq:1.17} respectively. This implies
    \begin{align}
        \label{eq:7.17}
        n-R_l \leq \dfrac{P_l}{2\sqrt{\tau}}\pa{\alpha + \beta\sqrt{\tau}}^m + Q_l < n - R_l + 1.
    \end{align}
    We first claim that \(\dfrac{P_l}{2\sqrt{\tau}}\pa{\alpha + \beta\sqrt{\tau}}^m + Q_l \notin \Z\) by separating into four cases. \\
    
    Consider the case \(l \in \{1,3\}\) and \(\lambda < \sqrt{\tau}\). Then \(\pd{s^{(l)}(m)} + \pd{t^{(l)}(m)}\) is equal to Equation \eqref{eq:7.3}. From Equation \eqref{eq:1.15}, Equation \eqref{eq:1.16} and Equation \eqref{eq:7.3},
    \begin{align*}
        &\dfrac{P_l}{2\sqrt{\tau}}\pa{\alpha + \beta\sqrt{\tau}}^m + Q_l \\
        &= \dfrac{1}{2\sqrt{\tau}}\pa{1-\lambda + \sqrt{\tau}}\pa{\alpha + \beta\sqrt{\tau}}^m + (-1)^{\pe{\frac{l}{2}}}\pa{\dfrac{E}{D}\lambda - \dfrac{E}{D} - \dfrac{d}{2}}, \\
        &= \pd{s^{(l)}(m)} + \pd{t^{(l)}(m)} + \pa{\dfrac{1}{2\sqrt{\tau}}\pa{1-\lambda - \sqrt{\tau}}\pa{\alpha - \beta\sqrt{\tau}}^m}.
    \end{align*}
    Based on Equation \eqref{eq:7.4} and Equation \eqref{eq:7.5}, \(\dfrac{1}{2\sqrt{\tau}}\pa{1-\lambda - \sqrt{\tau}}\pa{\alpha - \beta\sqrt{\tau}}^m \notin \Z\) and by assumption, \(\pd{s^{(l)}(m)} + \pd{t^{(l)}(m)} = n \in \Z\), so \(\dfrac{P_l}{2\sqrt{\tau}}\pa{\alpha + \beta\sqrt{\tau}}^m + Q_l \notin \Z\). \\
    
    Consider the case \(l \in \{1,3\}\) and \(\lambda > \sqrt{\tau}\). Then \(\pd{s^{(l)}(m)} + \pd{t^{(l)}(m)}\) is equal to Equation \eqref{eq:7.7}. From Equation \eqref{eq:1.15}, Equation \eqref{eq:1.16} and Equation \eqref{eq:7.7},
    \begin{align*}
        &\dfrac{P_l}{2\sqrt{\tau}}\pa{\alpha + \beta\sqrt{\tau}}^m + Q_l \\
        &= \dfrac{1}{2\sqrt{\tau}}\pa{1+\lambda - \sqrt{\tau}}\pa{\alpha + \beta\sqrt{\tau}}^m - (-1)^{\pe{\frac{l}{2}}}\pa{\dfrac{E}{D}\lambda + \dfrac{E}{D} - \dfrac{d}{2}}, \\
        &= \pd{s^{(l)}(m)} + \pd{t^{(l)}(m)} + \pa{\dfrac{1}{2\sqrt{\tau}}\pa{1+\lambda + \sqrt{\tau}}\pa{\alpha - \beta\sqrt{\tau}}^m}.
    \end{align*}
    Based on Equation \eqref{eq:7.8}, \(\dfrac{1}{2\sqrt{\tau}}\pa{1+\lambda + \sqrt{\tau}}\pa{\alpha - \beta\sqrt{\tau}}^m \notin \Z\) and by assumption, \(\pd{s^{(l)}(m)} + \pd{t^{(l)}(m)} = n \in \Z\), so \(\dfrac{P_l}{2\sqrt{\tau}}\pa{\alpha + \beta\sqrt{\tau}}^m + Q_l \notin \Z\). \\

    Consider the case \(l \in \{2,4\}\) and \(\lambda < -\sqrt{\tau}\). Then \(\pd{s^{(l)}(m)} + \pd{t^{(l)}(m)}\) is equal to Equation \eqref{eq:7.10}. From Equation \eqref{eq:1.15}, Equation \eqref{eq:1.16} and Equation \eqref{eq:7.10},
    \begin{align*}
        &\dfrac{P_l}{2\sqrt{\tau}}\pa{\alpha + \beta\sqrt{\tau}}^m + Q_l \\
        &= \dfrac{1}{2\sqrt{\tau}}\pa{1-\lambda - \sqrt{\tau}}\pa{\alpha + \beta\sqrt{\tau}}^m - (-1)^{\pe{\frac{l}{2}}}\pa{\dfrac{E}{D}\lambda - \dfrac{E}{D} - \dfrac{d}{2}}, \\
        &= \pd{s^{(l)}(m)} + \pd{t^{(l)}(m)} + \pa{\dfrac{1}{2\sqrt{\tau}}\pa{1-\lambda + \sqrt{\tau}}\pa{\alpha - \beta\sqrt{\tau}}^m}.
    \end{align*}
    Based on Equation \eqref{eq:7.11}, \(\dfrac{1}{2\sqrt{\tau}}\pa{1-\lambda + \sqrt{\tau}}\pa{\alpha - \beta\sqrt{\tau}}^m \notin \Z\) and by assumption, \(\pd{s^{(l)}(m)} + \pd{t^{(l)}(m)} = n \in \Z\), so \(\dfrac{P_l}{2\sqrt{\tau}}\pa{\alpha + \beta\sqrt{\tau}}^m + Q_l \notin \Z\). \\

    Consider the case \(l \in \{2,4\}\) and \(\lambda > -\sqrt{\tau}\). Then \(\pd{s^{(l)}(m)} + \pd{t^{(l)}(m)}\) is equal to Equation \eqref{eq:7.13}. From Equation \eqref{eq:1.15}, Equation \eqref{eq:1.16} and Equation \eqref{eq:7.13},
    \begin{align*}
        &\dfrac{P_l}{2\sqrt{\tau}}\pa{\alpha + \beta\sqrt{\tau}}^m + Q_l \\
        &= \dfrac{1}{2\sqrt{\tau}}\pa{1+\lambda + \sqrt{\tau}}\pa{\alpha + \beta\sqrt{\tau}}^m + (-1)^{\pe{\frac{l}{2}}}\pa{\dfrac{E}{D}\lambda + \dfrac{E}{D} - \dfrac{d}{2}}, \\
        &= \pd{s^{(l)}(m)} + \pd{t^{(l)}(m)} + \pa{\dfrac{1}{2\sqrt{\tau}}\pa{1+\lambda - \sqrt{\tau}}\pa{\alpha - \beta\sqrt{\tau}}^m}.
    \end{align*}
    Based on Equation \eqref{eq:7.14} and Equation \eqref{eq:7.15}, \(\dfrac{1}{2\sqrt{\tau}}\pa{1+\lambda - \sqrt{\tau}}\pa{\alpha - \beta\sqrt{\tau}}^m \notin \Z\) and by assumption, \(\pd{s^{(l)}(m)} + \pd{t^{(l)}(m)} = n \in \Z\), so \(\dfrac{P_l}{2\sqrt{\tau}}\pa{\alpha + \beta\sqrt{\tau}}^m + Q_l \notin \Z\). \\
    
    Inequality Equation \eqref{eq:7.17} now can be rewrite as
    \begin{align}
        \label{eq:7.18}
        n-R_l < \dfrac{P_l}{2\sqrt{\tau}}\pa{\alpha + \beta\sqrt{\tau}}^m + Q_l < n - R_l + 1.
    \end{align}
    On one hand of Equation \eqref{eq:7.18} (note that \(P_l > 0\) and \(\alpha + \beta\sqrt{\tau} > 1\)),
    \begin{align}
        \notag
        \dfrac{P_l}{2\sqrt{\tau}}\pa{\alpha + \beta\sqrt{\tau}}^m + Q_l &< n - R_l + 1, \\
        \notag
        \implies 0 < \pa{\alpha + \beta\sqrt{\tau}}^m &< \dfrac{2\sqrt{\tau}}{P_l}\pa{n - R_l + 1 - Q_l}, \\
        \notag
        \implies m &< \dfrac{\log\pa{n - R_l + 1 - Q_l} - \log P_l + \log \pa{2\sqrt{\tau}}}{\log \pa{\alpha + \beta\sqrt{\tau}}}, \\
        \label{eq:7.19}
        \implies \pe{m} &< \dfrac{\log\pa{n - R_l + 1 - Q_l} - \log P_l + \log \pa{2\sqrt{\tau}}}{\log \pa{\alpha + \beta\sqrt{\tau}}}.
    \end{align}
    On the other hand of Equation \eqref{eq:7.18} (note that \(P_l > 0\)),
    \begin{align}
        \notag
        n - R_l &< \dfrac{P_l}{2\sqrt{\tau}}\pa{\alpha + \beta\sqrt{\tau}}^m + Q_l, \\
        \label{eq:7.20}
        \implies \pa{\alpha + \beta\sqrt{\tau}}^m &> \dfrac{2\sqrt{\tau}}{P_l}\pa{n - R_l - Q_l}.
    \end{align}
    We claim that \(n - R_l - Q_l > 0\) so that we can take logs on both sides of inequality Equation \eqref{eq:7.20}. \\
    
    Consider the case \(l \in \{1,3\}\) and \(\lambda < \sqrt{\tau}\). Then \(n = \pd{s^{(l)}(m)} + \pd{t^{(l)}(m)}\) is equal to Equation \eqref{eq:7.2}. From Equation \eqref{eq:7.2}, Equation \eqref{eq:1.16} and Equation \eqref{eq:1.17},
    \begin{align*}
        &n - R_l - Q_l \\
        &= \pd{s^{(l)}(m)} + \pd{t^{(l)}(m)} - R_l - Q_l, \\
        &= \pa{u(m) + \pa{1 -\lambda} v(m) + (-1)^{\pe{\frac{l}{2}}}\pa{\dfrac{E}{D}\lambda - \dfrac{E}{D} - \dfrac{d}{2}}} \\
        &\hspace{150pt} - R_l - \pa{(-1)^{\pe{\frac{l}{2}}}\pa{\dfrac{E}{D}\lambda - \dfrac{E}{D} - \dfrac{d}{2}}}, \\
        &= u(m) + \pa{1 -\lambda} v(m) - R_l, \\
        &\geq u(m) + \pa{1 -\lambda} v(m) - 1, \\
        &> u(m) + \pa{1 - \sqrt{\tau}} v(m) - 1 = \bigl(u(m) - v(m)\sqrt{\tau}\bigr) + v(m) - 1 > 0 + v(1) - 1 \geq 0.
    \end{align*}

    Consider the case \(l \in \{1,3\}\) and \(\lambda > \sqrt{\tau}\). Then \(n = \pd{s^{(l)}(m)} + \pd{t^{(l)}(m)}\) is equal to Equation \eqref{eq:7.6}. From Equation \eqref{eq:7.6}, Equation \eqref{eq:1.16} and Equation \eqref{eq:1.17},
    \begin{align*}
        &n - R_l - Q_l \\
        &= \pd{s^{(l)}(m)} + \pd{t^{(l)}(m)} - R_l - Q_l, \\
        &= \pa{-u(m) + \pa{1 +\lambda} v(m) - (-1)^{\pe{\frac{l}{2}}}\pa{\dfrac{E}{D}\lambda + \dfrac{E}{D} - \dfrac{d}{2}}} \\
        &\hspace{150pt} - R_l - \pa{-(-1)^{\pe{\frac{l}{2}}}\pa{\dfrac{E}{D}\lambda - \dfrac{E}{D} - \dfrac{d}{2}}}, \\
        &= -u(m) + \pa{1 +\lambda} v(m), \\
        &> -u(m) + \pa{1 + \sqrt{\tau}} v(m) = -\bigl(u(m) - v(m)\sqrt{\tau}\bigr) + v(m) > -1 + v(1) \geq 0.
    \end{align*}

    Consider the case \(l \in \{2,4\}\) and \(\lambda < -\sqrt{\tau}\). Then \(n = \pd{s^{(l)}(m)} + \pd{t^{(l)}(m)}\) is equal to Equation \eqref{eq:7.9}. From Equation \eqref{eq:7.9}, Equation \eqref{eq:1.16} and Equation \eqref{eq:1.17},
    \begin{align*}
        &n - R_l - Q_l \\
        &= \pd{s^{(l)}(m)} + \pd{t^{(l)}(m)} - R_l - Q_l, \\
        &= \pa{-u(m) + \pa{1 -\lambda} v(m) - (-1)^{\pe{\frac{l}{2}}}\pa{\dfrac{E}{D}\lambda - \dfrac{E}{D} - \dfrac{d}{2}}} \\
        &\hspace{150pt} - R_l - \pa{- (-1)^{\pe{\frac{l}{2}}}\pa{\dfrac{E}{D}\lambda - \dfrac{E}{D} - \dfrac{d}{2}}}, \\
        &= -u(m) + \pa{1 -\lambda} v(m), \\
        &> -u(m) + \pa{1 + \sqrt{\tau}} v(m) = -\bigl(u(m) - v(m)\sqrt{\tau}\bigr) + v(m) > -1 + v(1) \geq 0.
    \end{align*}

    Consider the case \(l \in \{2,4\}\) and \(\lambda > -\sqrt{\tau}\). Then \(n = \pd{s^{(l)}(m)} + \pd{t^{(l)}(m)}\) is equal to Equation \eqref{eq:7.12}. From Equation \eqref{eq:7.12}, Equation \eqref{eq:1.16} and Equation \eqref{eq:1.17},
    \begin{align*}
        &n - R_l - Q_l \\
        &= \pd{s^{(l)}(m)} + \pd{t^{(l)}(m)} - R_l - Q_l, \\
        &= \pa{u(m) + \pa{1 +\lambda} v(m) + (-1)^{\pe{\frac{l}{2}}}\pa{\dfrac{E}{D}\lambda + \dfrac{E}{D} - \dfrac{d}{2}}} \\
        &\hspace{150pt} - R_l - \pa{(-1)^{\pe{\frac{l}{2}}}\pa{\dfrac{E}{D}\lambda + \dfrac{E}{D} - \dfrac{d}{2}}}, \\
        &= u(m) + \pa{1 +\lambda} v(m) - R_l, \\
        &\geq u(m) + \pa{1 +\lambda} v(m) - 1, \\
        &> u(m) + \pa{1 - \sqrt{\tau}} v(m) - 1 = \bigl(u(m) - v(m)\sqrt{\tau}\bigr) + v(m) - 1 > 0 + v(1) - 1 \geq 0.
    \end{align*}
    Taking logs on both sides of Equation \eqref{eq:7.20} (note that \(\alpha + \beta\sqrt{\tau} > 1\)), we obtain
    \begin{align}
        \label{eq:7.21}
        m &> \dfrac{\log\pa{n - R_l - Q_l} - \log P_l + \log \pa{2\sqrt{\tau}}}{\log \pa{\alpha + \beta\sqrt{\tau}}}.
    \end{align}
    Consider the integer \(\pe{m} + 1\), and let
    \[n' = \pd{s^{(l)}\pa{\pe{m} + 1}} + \pd{t^{(l)}\pa{\pe{m} + 1}} \in \Z.\]
    Since \(n \geq M'_l\), then by Lemma \ref{lem:5.3} and Lemma \ref{lem:5.6}, \(m \geq \max\{N_0, N_l\} \geq N_0\). Since \(\pd{s^{(l)}(m)} + \pd{t^{(l)}(m)}\) is strictly increasing for all \(m \geq N_0\), then from \(\pe{m} + 1 > m \geq N_0\), 
    \[n' = \pd{s^{(l)}\pa{\pe{m} + 1}} + \pd{t^{(l)}\pa{\pe{m} + 1}} > \pd{s^{(l)}(m)} + \pd{t^{(l)}(m)} = n.\]
    Notice that inequality Equation \eqref{eq:7.21} is true for all positive integers \(n \geq M'_l\) (with the chosen \(m \in [1,+\infty)\) such that \(\pd{s^{(l)}(m)} + \pd{t^{(l)}(m)} = n\)). Since \(n' > n \geq M'_l\), therefore we can replace \(m\) and \(n\) in Equation \eqref{eq:7.21} with \(\pe{m} + 1\) and \(n'\) respectively:
    \begin{align}
        \label{eq:7.22}
        \pe{m} + 1 &> \dfrac{\log\pa{n' - R_l - Q_l} - \log P_l + \log \pa{2\sqrt{\tau}}}{\log \pa{\alpha + \beta\sqrt{\tau}}}.
    \end{align}
    Since both \(n'\) and \(n\) are integers, and \(n' > n\), we must have \(n' \geq n + 1\). Therefore from Equation \eqref{eq:7.22},
    \begin{align}
        \notag
        \pe{m} + 1 &> \dfrac{\log\pa{n' - R_l - Q_l} - \log P_l + \log \pa{2\sqrt{\tau}}}{\log \pa{\alpha + \beta\sqrt{\tau}}}, \\
        \label{eq:7.23}
        \implies \pe{m} + 1 &> \dfrac{\log\pa{n - R_l + 1 - Q_l} - \log P_l + \log \pa{2\sqrt{\tau}}}{\log \pa{\alpha + \beta\sqrt{\tau}}}.
    \end{align}
    We no need to worry about the value \(\log\pa{n - R_l + 1 - Q_l}\) being undefined due to we already showed \(n - R_l - Q_l > 0\) previously. Finally from both Equation \eqref{eq:7.19} and Equation \eqref{eq:7.23}, we get Equation \eqref{eq:7.16}.
\end{proof}

\vspace{5pt}
\begin{lemma}
    \label{lem:7.5}
    Let \(l \in \{1,2,3,4\}\) and \(M'_l\) as defined in Equation \eqref{eq:5.1}. For any real number \(x \geq M'_l\), 
    \[\pd{\mD_{A}^{(l)}(x)} = \pe{\dfrac{\log\pa{\pe{x} - R_l + 1- Q_l} - \log P_l + \log \pa{2\sqrt{\tau}}}{\log \pa{\alpha + \beta\sqrt{\tau}}}},\]
    where \(P_l\), \(Q_l\) and \(R_l\) are defined as Equation \eqref{eq:1.15}, Equation \eqref{eq:1.16} and Equation \eqref{eq:1.17} respectively.
\end{lemma}

\begin{proof}
    Notice that from Equation \eqref{eq:6.3},
    \begin{align*}
        \mD_{A}^{(l)}(x) &= \Bigl\{\pa{s_m^{(l)}, t_m^{(l)}} : m \in \N, \pd{s_m^{(l)}} + \pd{t_m^{(l)}} \leq \pe{x}\Bigr\} = \mD_{A}^{(l)}\pa{\pe{x}},
    \end{align*}
    hence \(\pd{\mD_{A}^{(l)}(x)} = \pd{\mD_{A}^{(l)}(\pe{x})}\). To compute \(\pd{\mD_{A}^{(l)}(\pe{x})}\), this is just direct application of Lemma \ref{lem:7.4}.
\end{proof}

\vspace{5pt}
\subsection{Number of Elements \(\pd{\mD_{A}^{(0)}(x)}\)}

\vspace{5pt}
\begin{lemma}
    \label{lem:7.6}
    There exists \(l \in \{1,2,3,4\}\) such that
    \[\pd{\pm 1 + \dfrac{E}{D}\lambda - \dfrac{d}{2}} + \dfrac{\pd{E}}{D} < \pd{s^{(l)}_{N_0}} + \pd{t^{(l)}_{N_0}},\]
    where \(N_0\) is defined as Equation \eqref{eq:3.2}.
\end{lemma}

\begin{proof}
    From Equation \eqref{eq:1.18} and Equation \eqref{eq:1.19},
    \begin{align}
        \notag
        &\pd{s^{(l)}_{N_0}} + \pd{t^{(l)}_{N_0}} \\
        \label{eq:7.24}
        &= \pd{(-1)^{\pe{\frac{l}{2}}} u_{N_0} - (-1)^{\pe{\frac{l-1}{2}}} \lambda v_{N_0} + \pa{\dfrac{E}{D}\lambda - \dfrac{d}{2}}} + \pd{(-1)^{\pe{\frac{l-1}{2}}} v_{N_0} - \dfrac{E}{D}}.
    \end{align}
    We now pick \(l \in \{1,2,3,4\}\) by considering the cases below:
    \begin{itemize}
        \item If \(\lambda < 0\) and \(\dfrac{E}{D}\lambda - \dfrac{d}{2} \geq 0\), then \(l = 1\);
        \item If \(\lambda \geq 0\) and \(\dfrac{E}{D}\lambda - \dfrac{d}{2} < 0\), then \(l = 2\);
        \item If \(\lambda < 0\) and \(\dfrac{E}{D}\lambda - \dfrac{d}{2} < 0\), then \(l = 3\);
        \item If \(\lambda \geq 0\) and \(\dfrac{E}{D}\lambda - \dfrac{d}{2} \geq 0\), then \(l = 4\).
    \end{itemize}
    This implies
    \begin{align}
        \label{eq:7.25}
        \pd{(-1)^{\pe{\frac{l}{2}}} u_{N_0} - (-1)^{\pe{\frac{l-1}{2}}} \lambda v_{N_0} + \pa{\dfrac{E}{D}\lambda - \dfrac{d}{2}}} = u_{N_0} + |\lambda| v_{N_0} + \pd{\dfrac{E}{D}\lambda - \dfrac{d}{2}}.
    \end{align}
    On the other hand, by triangle inequality,
    \begin{align}
        \label{eq:7.26}
        \pd{\pm 1 + \dfrac{E}{D}\lambda - \dfrac{d}{2}} + \dfrac{\pd{E}}{D} \leq 1 + \pd{\dfrac{E}{D}\lambda - \dfrac{d}{2}} + \dfrac{\pd{E}}{D}.
    \end{align}
    Hence, from Equation \eqref{eq:7.24}, Equation \eqref{eq:7.25} and Equation \eqref{eq:7.26}, we are sufficient to show
    \begin{align}
        \label{eq:7.27}
        1 + \pd{\dfrac{E}{D}\lambda - \dfrac{d}{2}} + \dfrac{\pd{E}}{D} < u_{N_0} + |\lambda| v_{N_0} + \pd{\dfrac{E}{D}\lambda - \dfrac{d}{2}} + \pd{(-1)^{\pe{\frac{l-1}{2}}} v_{N_0} - \dfrac{E}{D}}.
    \end{align}
    From Equation \eqref{eq:3.2}, \(v_{N_0} > \dfrac{\pd{E}}{D}\). From Equation \eqref{eq:1.9} and Equation \eqref{eq:1.10}, it is clear that \(u_{N_0} > v_{N_0}\), hence \(u_{N_0} \geq v_{N_0} + 1\). So we have
    \begin{align}
        \notag
        1 + \dfrac{\pd{E}}{D} < 1 + v_{N_0} \leq u_{N_0} &\leq u_{N_0} + |\lambda| v_{N_0} + \pd{(-1)^{\pe{\frac{l-1}{2}}} v_{N_0} - \dfrac{E}{D}}, \\
        \label{eq:7.28}
        \implies 1 + \dfrac{\pd{E}}{D} &< u_{N_0} + |\lambda| v_{N_0} + \pd{(-1)^{\pe{\frac{l-1}{2}}} v_{N_0} - \dfrac{E}{D}}.
    \end{align}
    Adding \(\pd{\dfrac{E}{D}\lambda - \dfrac{d}{2}}\) to both sides of Equation \eqref{eq:7.28}, we get Equation \eqref{eq:7.27}.
\end{proof}

\vspace{5pt}
\begin{lemma}
    \label{lem:7.7}
    Let \(l \in \{1,2,3,4\}\) and \(M'_l\) as defined in Equation \eqref{eq:5.1}. For any real number \(x \geq \mL := \max\Bigl\{M'_l: l \in \{1,2,3,4\}\Bigr\}\), 
    \[\pd{\mD_{A}^{(0)}(x)} = 2.\]
\end{lemma}

\begin{proof}
    Let \(x \geq \mL\). Clearly \(\pe{x} \geq \mL\) since \(\mL \in \Z\). By Equation \eqref{eq:6.2}, \(\pd{\mD_{A}^{(0)}(x)} = 2\) if
    \begin{align}
        \label{eq:7.29}
        \pd{\pm 1 + \dfrac{E}{D}\lambda - \dfrac{d}{2}} + \dfrac{\pd{E}}{D} \leq \pe{x},
    \end{align}
    so we are sufficient to show Equation \eqref{eq:7.29}. From Lemma \ref{lem:7.6}, there exists \(l \in \{1,2,3,4\}\) such that \(\pd{\pm 1 + \dfrac{E}{D}\lambda - \dfrac{d}{2}} + \dfrac{\pd{E}}{D} < \pd{s^{(l)}_{N_0}} + \pd{t^{(l)}_{N_0}}\). It follows that
    \begin{align*}
        &\pd{\pm 1 + \dfrac{E}{D}\lambda - \dfrac{d}{2}} + \dfrac{\pd{E}}{D} \\
        &\leq \pd{s^{(l)}_{N_0}} + \pd{t^{(l)}_{N_0}}, \\
        &\leq \pd{s_{\max\pc{N_0, N_l}}^{(l)}} + \pd{t_{\max\pc{N_0, N_l}}^{(l)}}, \\
        &\leq \max\pc{\pd{s_{1}^{(l)}}, \pd{s_{\max\pc{N_0, N_l}}^{(l)}}} + \max\pc{\pd{t_{1}^{(l)}}, \pd{t_{\max\pc{N_0, N_l}}^{(l)}}}, \\
        &= M'_l, \\
        &\leq \max\pc{M'_l: l \in \{1,2,3,4\}} = \mL, \\
        &\leq \pe{x}.
    \end{align*}
\end{proof}

\vspace{5pt}
\section{Proof of Theorem \ref{thm:1.4}}
\begin{proof}
    From Equation \eqref{eq:6.9},
    \[\Bigl|\mD_{A}(x)\Bigr| = \pd{\mathcal{D}_{A}^{(0)}(x)} + \pa{\sum_{l=1}^{4}\pd{\mathcal{D}_{A}^{(l)}(x)}}.\]
    By applying Lemma \ref{lem:7.5} and Lemma \ref{lem:7.7},
    \[\Bigl|\mD_{A}(x)\Bigr| = 2 + \displaystyle\sum_{l=1}^4\pe{\dfrac{\log\pa{\pe{x} - R_l + 1- Q_l} - \log P_l + \log \pa{2\sqrt{\tau}}}{\log \pa{\alpha + \beta\sqrt{\tau}}}},\]
    where \(P_l\), \(Q_l\) and \(R_l\) are defined as Equation \eqref{eq:1.15}, Equation \eqref{eq:1.16} and Equation \eqref{eq:1.17} respectively.
\end{proof}

\vspace{5pt}
\section{Proof of Theorem \ref{thm:1.5}}
\begin{proof}
    From Equation \eqref{eq:6.10},
    \[\mD_{A}(x) = \mathcal{D}_{A}^{(0)}(x) \cup \pa{\bigcup_{l=1}^{4}\mathcal{D}_{A}^{(l)}(x)}.\]
    Since \(\pd{\mathcal{D}_{A}^{(0)}(x)} = 2\) from Lemma \ref{lem:7.7}, so by Equation \eqref{eq:6.2}, we must have
    \[\mathcal{D}_{A}^{(0)}(x) = \pc{\pa{1 + \dfrac{E}{D}\lambda- \dfrac{d}{2}, - \dfrac{E}{D}}, \pa{-1 + \dfrac{E}{D}\lambda- \dfrac{d}{2}, - \dfrac{E}{D}}}.\]
    We now claim that for any \(l \in \{1,2,3,4\}\) and \(x \geq M'_l\),
    \[\mathcal{D}_{A}^{(l)}(x) = \pc{\pa{s_m^{(l)}, t_m^{(l)}} : 1 \leq m \leq \pe{\dfrac{\log\pa{\pe{x} - R_l + 1- Q_l} - \log P_l + \log \pa{2\sqrt{\tau}}}{\log \pa{\alpha + \beta\sqrt{\tau}}}}},\]
    so that Theorem \ref{thm:1.5} follows. We separate into two cases: \\

    Let \(x = M'_l\). Then by Lemma \ref{lem:7.1},
    \[\pd{\mathcal{D}_{A}^{(l)}\pa{M'_l}} = \pe{N'_l},\]
    where \(N'_l\) is defined as Equation \eqref{eq:5.10}, and from Lemma \ref{lem:7.5},
    \[\pd{\mathcal{D}_{A}^{(l)}(M'_l)} = \pe{\dfrac{\log\pa{M'_l - R_l + 1- Q_l} - \log P_l + \log \pa{2\sqrt{\tau}}}{\log \pa{\alpha + \beta\sqrt{\tau}}}}.\]
    Hence we must have
    \begin{align}
        \label{eq:9.1}
        \pe{N'_l} = \pe{\dfrac{\log\pa{M'_l - R_l + 1- Q_l} - \log P_l + \log \pa{2\sqrt{\tau}}}{\log \pa{\alpha + \beta\sqrt{\tau}}}}.
    \end{align}
    By applying Lemma \ref{lem:5.5}(c), any integer \(1 \leq m \leq \pe{N'_l}\) satisfies the condition \(\pd{s^{(l)}(m)} + \pd{t^{(l)}(m)} \leq M'_l\), while by applying Lemma \ref{lem:5.5}(a), any integer \(m \geq \pe{N'_l} + 1\) does not satisfy the condition. By applying Equation \eqref{eq:9.1}, we have
    \begin{align*}
        &\Bigl\{m \in \N : \pd{s_m^{(l)}} + \pd{t_m^{(l)}} \leq M'_l\Bigr\} \\
        &\hspace{20pt} = \Bigl\{m \in \N : 1 \leq m \leq \pe{N'_l}\Bigr\}, \\
        &\hspace{20pt} = \pc{m \in \N : 1 \leq m \leq \pe{\dfrac{\log\pa{M'_l - R_l + 1- Q_l} - \log P_l + \log \pa{2\sqrt{\tau}}}{\log \pa{\alpha + \beta\sqrt{\tau}}}}}.
    \end{align*}
    Therefore, from Equation \eqref{eq:6.3},
    \begin{align*}
        \mathcal{D}_{A}^{(l)}(M'_l) &= \Bigl\{\pa{s_m^{(l)}, t_m^{(l)}} : m \in \N, \pd{s_m^{(l)}} + \pd{t_m^{(l)}} \leq M'_l\Bigr\}, \\
        &= \pc{\pa{s_m^{(l)}, t_m^{(l)}} : 1 \leq m \leq \pe{\dfrac{\log\pa{M'_l - R_l + 1 - Q_l} - \log P_l + \log \pa{2\sqrt{\tau}}}{\log \pa{\alpha + \beta\sqrt{\tau}}}}},
    \end{align*}
    as desired. \\

    Let \(x > M'_l\). Then by Lemma \ref{lem:7.2},
    \[\pd{\mathcal{D}_{A}^{(l)}\pa{x}} = \pe{M},\]
    where \(\pd{s^{(l)}(M)} + \pd{t^{(l)}(M)} = x\), and from Lemma \ref{lem:7.5},
    \[\pd{\mathcal{D}_{A}^{(l)}(x)} = \pe{\dfrac{\log\pa{\pe{x} - R_l + 1- Q_l} - \log P_l + \log \pa{2\sqrt{\tau}}}{\log \pa{\alpha + \beta\sqrt{\tau}}}}.\]
    Hence we must have
    \begin{align}
        \label{eq:9.2}
        \pe{M} = \pe{\dfrac{\log\pa{\pe{x} - R_l + 1- Q_l} - \log P_l + \log \pa{2\sqrt{\tau}}}{\log \pa{\alpha + \beta\sqrt{\tau}}}}.
    \end{align}
    In the proof of Lemma \ref{lem:7.2}, we showed any integer \(1 \leq m \leq \pe{M}\) satisfies the condition \(\pd{s^{(l)}(m)} + \pd{t^{(l)}(m)} \leq \pe{x}\), while any integer \(m \geq \pe{M} + 1\) does not satisfy the condition. By applying Equation \eqref{eq:9.2}, we have
    \begin{align*}
        &\Bigl\{m \in \N : \pd{s_m^{(l)}} + \pd{t_m^{(l)}} \leq \pe{x}\Bigr\} \\
        &\hspace{20pt} = \Bigl\{m \in \N : 1 \leq m \leq \pe{M}\Bigr\}, \\
        &\hspace{20pt} = \pc{m \in \N : 1 \leq m \leq \pe{\dfrac{\log\pa{\pe{x} - R_l + 1- Q_l} - \log P_l + \log \pa{2\sqrt{\tau}}}{\log \pa{\alpha + \beta\sqrt{\tau}}}}}.
    \end{align*}
    Therefore, from Equation \eqref{eq:6.3},
    \begin{align*}
        \mathcal{D}_{A}^{(l)}(x) &= \Bigl\{\pa{s_m^{(l)}, t_m^{(l)}} : m \in \N, \pd{s_m^{(l)}} + \pd{t_m^{(l)}} \leq \pe{x}\Bigr\}, \\
        &= \pc{\pa{s_m^{(l)}, t_m^{(l)}} : 1 \leq m \leq \pe{\dfrac{\log\pa{\pe{x} - R_l + 1- Q_l} - \log P_l + \log \pa{2\sqrt{\tau}}}{\log \pa{\alpha + \beta\sqrt{\tau}}}}},
    \end{align*}
    as desired.
\end{proof}

\vskip 20pt\noindent {\bf Acknowledgement.}
The work in this paper was done while the first author was a master's degree student at Universiti Kebangsaan Malaysia (UKM) under the advisement of the second author. The first author would like to express the deepest gratitude under the advisement of the second author for the continuous guidance, advice and encouragement. \\

\end{document}